\documentclass{amsart}
\usepackage{etex}
\usepackage{fixltx2e}

\newcommand{\Hb}{\mathbf{H}}
\newcommand{\Gb}{\mathbf{G}}

\newcommand{\Mb}{\mathbf{M}}
\newcommand{\Mbt}{\widetilde{\mathbf{M}}}
\newcommand{\Pb}{\mathbf{P}}
\newcommand{\Pbt}{\widetilde{\mathbf{P}}}
\newcommand{\tbP}{\widetilde{\mathbf{P}}}
\renewcommand{\epsilon}{\varepsilon}

\newcommand{\tbG}{\widetilde{\bG}}
\newcommand{\tbM}{\widetilde{\bM}}
\newcommand{\WF}{\mathrm{W}_F}
\newcommand{\tpi}{\widetilde{\pi}}

\newcommand{\tpsi}{\widetilde{\psi}}
\newcommand{\tPi}{\widetilde{\Pi}}
\newcommand{\Gammab}{\mathbf{\Gamma}}
\newcommand{\tGamma}{\widetilde{\mathbf{\Gamma}}}

\newcommand{\tPsi}{\widetilde{\Psi}}
\newcommand{\sign}{\mathrm{sign}}
\newcommand{\nor}{\mathrm{norm}}
\newcommand{\Jac}{\mathrm{Jac}}
\newcommand{\Ghat}{\widehat{\mathbf{G}}}
\newcommand{\Lhat}{\widehat{\mathbf{L}}}
\newcommand{\Hhat}{\widehat{\mathbf{H}}}
\newcommand{\That}{\widehat{\mathbf{T}}}
\newcommand{\muhat}{\widehat{\mu}}
\newcommand{\deltahat}{\widehat{\delta}}

\usepackage{tikz}
\usepackage{tikz-cd}
\usetikzlibrary{arrows, matrix}
\usepackage[shortlabels]{enumitem}\usetikzlibrary{decorations.pathmorphing}

\newcommand{\Lie}{{\operatorname{Lie}\,}}
\newcommand{\Cent}{{\operatorname{Cent}\,}}
\newcommand{\Out}{{\operatorname{Out}\,}}
\newcommand{\semis}{{\operatorname{ss}}}

\newcommand{\fg}{{\mathfrak{g}}}

\newcommand{\fe}{{\mathfrak{e}}}
\newcommand{\fw}{{\mathfrak{w}}}

\newcommand{\rec}{{\operatorname{rec}}}

\newcommand{\CA}{{\mathcal{A}}}

\newcommand{\CE}{{\mathcal{E}}}

\newcommand{\CH}{{\mathcal{H}}}

\newcommand{\CO}{{\mathcal{O}}}

\newcommand{\CS}{{\mathcal{S}}}

\newcommand{\CZ}{{\mathcal{Z}}}

\usepackage{amssymb,amsmath,amsfonts,amsthm,epsfig,amscd,stmaryrd}
\usepackage{stmaryrd}
\usepackage[all,cmtip,poly]{xy}
\usepackage{color}

\newcommand{\ad}{\operatorname{ad}}
\newcommand{\diag}{\operatorname{diag}}
\newcommand{\tr}{\operatorname{tr}}

\newcommand{\pr}{\operatorname{pr}}

\newcommand{\Isom}{\operatorname{Isom}}

\newcommand{\Ad}{{\mathrm{Ad}}}

\newtheorem{theorem}[subsubsection]{Theorem}
\newtheorem{thm}[subsubsection]{Theorem}

\newtheorem{lem}[subsubsection]{Lemma}

\newtheorem{conj}[subsubsection]{Conjecture}

\newtheorem{prop}[subsubsection]{Proposition}

\theoremstyle{definition}

\newtheorem{defn}[subsubsection]{Definition}

\newtheorem{rem}[subsubsection]{Remark}

\def\numequation{\addtocounter{subsubsection}{1}\begin{equation}}
\def\nummultline{\addtocounter{subsubsection}{1}\begin{multline}}
\def\anumequation{\addtocounter{subsection}{1}\begin{equation}}
\def\anummultline{\addtocounter{subsection}{1}\begin{multline}}

\newif\iffinalrun
\iffinalrun
\else
 \fi

\iffinalrun
  \newcommand{\need}[1]{}
  \newcommand{\mar}[1]{}
\else
  \newcommand{\need}[1]{{\tiny *** #1}}
  \newcommand{\mar}[1]{\marginpar{\raggedright\tiny TODO  #1}}\fi

\newcommand{\A}{\AA}
\newcommand{\C}{\CC}

\newcommand{\Q}{\QQ}
\newcommand{\R}{\RR}
\newcommand{\Z}{\ZZ}

\renewcommand{\AA}{{\mathbb A}}

\newcommand{\CC}{{\mathbb C}}

\newcommand{\QQ}{{\mathbb Q}}
\newcommand{\RR}{{\mathbb R}}

\newcommand{\ZZ}{{\mathbb Z}}

\newcommand{\bA}{\ensuremath{\mathbf{A}}}
\newcommand{\bB}{\ensuremath{\mathbf{B}}}

\newcommand{\bD}{\ensuremath{\mathbf{D}}}

\newcommand{\bG}{\ensuremath{\mathbf{G}}}
\newcommand{\bH}{\ensuremath{\mathbf{H}}}

\newcommand{\bL}{\ensuremath{\mathbf{L}}}
\newcommand{\bM}{\ensuremath{\mathbf{M}}}
\newcommand{\bN}{\ensuremath{\mathbf{N}}}

\newcommand{\bP}{\ensuremath{\mathbf{P}}}

\newcommand{\bR}{\ensuremath{\mathbf{R}}}

\newcommand{\bT}{\ensuremath{\mathbf{T}}}
\newcommand{\bU}{\ensuremath{\mathbf{U}}}

\newcommand{\bZ}{\ensuremath{\mathbf{Z}}}

\renewcommand{\bf}{\ensuremath{\mathbf{f}}}

\newcommand{\cA}{{\mathcal A}}

\newcommand{\cH}{{\mathcal H}}

\newcommand{\cL}{{\mathcal L}}
\newcommand{\CL}{{\mathcal{L}}}

\newcommand{\cS}{{\mathcal S}}

\newcommand{\frake}{\mathfrak{e}}

\DeclareMathOperator{\Aut}{Aut}

\DeclareMathOperator{\disc}{disc}
\DeclareMathOperator{\der}{der}
\DeclareMathOperator{\End}{End}

\DeclareMathOperator{\Gal}{Gal}
\DeclareMathOperator{\GL}{GL}
\DeclareMathOperator{\GO}{GO}
\DeclareMathOperator{\GSO}{GSO}
\newcommand{\GLb}{\mathbf{GL}}
\newcommand{\GLnb}{\mathbf{GL}_n}
\newcommand{\SLb}{\mathbf{SL}}
\newcommand{\GSOb}{\mathbf{GSO}}
\newcommand{\GOb}{\mathbf{GO}}
\newcommand{\SOb}{\mathbf{SO}}
\newcommand{\Spb}{\mathbf{Sp}}
\newcommand{\GSpb}{\mathbf{GSp}}

\newcommand{\GSpinb}{\mathbf{GSpin}}
\newcommand{\Spinb}{\mathbf{Spin}}
\DeclareMathOperator{\GSp}{GSp}
\DeclareMathOperator{\GSpin}{GSpin}

\DeclareMathOperator{\Hom}{Hom}

\DeclareMathOperator{\Ind}{Ind}

\DeclareMathOperator{\Norm}{N}

\DeclareMathOperator{\SL}{SL}
\DeclareMathOperator{\SO}{SO}
\DeclareMathOperator{\Sp}{Sp}
\DeclareMathOperator{\PGSp}{PGSp}

\DeclareMathOperator{\SU}{SU}

\DeclareMathOperator{\Sym}{Sym}

\DeclareMathOperator{\WD}{WD}

\newcommand{\Frob}{\mathrm{Frob}}

\newcommand{\reg}{\mathrm{reg}}

\newcommand{\Id}{\mathrm{Id}}

\newcommand{\into}{\hookrightarrow}

\newcommand{\Res}{\operatorname{Res}}
\newcommand{\res}{\operatorname{res}}

\DeclareMathOperator{\Std}{Std}

\usepackage[bookmarksopen,bookmarksdepth=2]{hyperref}

\begin{document}
\title{Arthur's multiplicity formula for $\GSpb_4$ and restriction to $\Sp_4$}

\author{Toby Gee} \email{toby.gee@imperial.ac.uk} \address{Department of
  Mathematics, Imperial College London,
  London SW7 2AZ, UK}

\author{Olivier Ta\"ibi} \email{olivier.taibi@ens-lyon.fr} \address{CNRS et
Unit\'e de Math\'ematiques Pures et Appliqu\'ees, ENS de Lyon}  
  
\thanks{T.G.\ was supported in part by a Leverhulme Prize, EPSRC grant
  EP/L025485/1, ERC Starting Grant 306326, and a Royal Society Wolfson Research
  Merit Award. O.T.\ was supported in part by ERC Starting Grant 306326.}

\begin{abstract}
  We prove the classification of discrete automorphic representations
  of~$\GSpb_4$ explained in~\cite{MR2058604}, as well as a compatibility between
  the local Langlands correspondences for~$\GSpb_4$ and~$\Spb_4$.
\end{abstract}

\maketitle

\setcounter{tocdepth}{1}
\tableofcontents

\section{Introduction} \subsection{}

In the paper~\cite{MR2058604}, Arthur explained his classification of the
discrete automorphic spectrum for classical groups in the particular case
of~$\GSpb_4 \cong \GSpinb_5$. Later, in~\cite{MR3135650} he proved this
classification for quasi-split special orthogonal and symplectic groups of
arbitrary rank, but now with trivial similitude factor. The classification
stated in~\cite{MR2058604} is important for applications of the Langlands
program to arithmetic. In particular, it is used in~\cite{MR3200667} to
associate Galois representations to Hilbert--Siegel modular forms, and these
Galois representations have been used to prove modularity lifting theorems
relating to abelian surfaces, for example in
\cite{BCGPGSp4}. It is therefore desirable to have an unconditional proof of
this classification. While it is expected that the methods of~\cite{MR3135650}
could be used to handle~$\GSpinb$ groups, the proofs involve a very complicated
induction, which even in the case of~$\GSpinb_5$ would involve the use of groups
of much higher rank, so there does not seem to be any way to give a (short)
direct proof of the classification of~\cite{MR2058604} by following the
arguments of~\cite{MR3135650}.

In this paper, we fill this gap in the literature by giving a proof of the
classification announced in~\cite{MR2058604}. We also prove some new results
concerning the compatibility of the local Langlands correspondences for~$\Spb_4$
and~$\GSpb_4$. While, like Arthur, our main technique is the stable (twisted)
trace formula, and we make substantial use of the results of~\cite{MR2058604}
for the group~$\Spb_4$, we also rely on a number of additional ingredients that
are only available in the particular case of~$\GSpb_4$; in particular, we
crucially use:
\begin{itemize}
	\item the exterior square functoriality for~$\GLb_4$ proved
		in~\cite{MR1937203} (and completed in~\cite{MR2567395});
	\item the results of~\cite{MR2800725}: the local Langlands correspondence
		for~$\GSpb_4$ (established using theta correspondences), and the generic
		transfer to~$\GSpb_4$ (with local-global compatibility at all places) for
		essentially self dual cuspidal automorphic representations of~$\GLb_4$ of
		symplectic type;
	\item  the results of~\cite{MR3267112}, which check the compatibility of the
		local Langlands correspondence of~\cite{MR2846304} with the predicted
		twisted endoscopic character relations of~\cite{MR2058604} in the tempered
		case.
\end{itemize}

We now briefly explain the strategy of our proof, and the structure of the
paper. We begin in Section~\ref{sec: Arthur} with a precise statement of the
results of~\cite{MR3135650} and of their conjectural extension to~$\GSpinb$
groups. Roughly speaking, these statements consist of:
\begin{enumerate}
\item An assignment of global parameters (formal sums of essentially
  self-dual discrete automorphic representations of~$\GLb_n$) to
  discrete automorphic representations of classical groups.
\item A description of packets of local representations in terms of
  local versions of the global parameters (which in particular gives
  the local Langlands correspondence for classical groups).
\item A multiplicity formula, precisely describing which elements of
  global packets are automorphic, and the multiplicities with which
  they appear in the discrete spectrum.
\end{enumerate}
In Arthur's work these statements are all proved together as part of a
complicated induction, but in this paper (which of course uses Arthur's results
for~$\Spb_4$) we are able to prove the first two statements independently, and
then use them as inputs to the proof of the third statement.

In section~\ref{sec: missing local packets} we study the local packets. In the
tempered case, the work has already been done in~\cite{MR3267112}, and by again
using that~\cite{MR3135650} has taken care of the cases where the similitude
character is a square, we are reduced to constructing the local packets in two
special non-tempered cases. We do this ``by hand'', following the much more
general results proved in~\cite{MWtransfertGLtordu} and \cite{AMR}.

As a consequence of the stabilisation of the twisted trace
formula~\cite{SFTT1,SFTT2}, we can apply the twisted trace formula for~$\GLb_4
\times \GLb_1$ to associate a global parameter to any discrete automorphic
representation of~$\GSpinb_5$ (which is a twisted endoscopic group for~$\GLb_4
\times \GLb_1$ endowed with the automorphism $g \mapsto {}^t g^{-1}$). We recall
the details of this twisted trace formula in section~\ref{sec: STTF}, which we
hope can serve as an introduction to the results of~\cite{SFTT1,SFTT2}  for the
reader not already familiar with them. In section~\ref{sec:restriction} we
briefly recall results about the restriction of representations to subgroups,
which we apply to the case of restriction from~$\GSpb_4$ to~$\Spb_4$.

In section~\ref{sec: global parameters} we show that the global parameter
associated to a discrete automorphic representation of~$\GSpb_4$ by the stable
twisted trace formula is of the form predicted by Arthur, by making use of the
symplectic/orthogonal alternative for~$\GLb_2$ and~$\GLb_4$, the (known)
description of automorphic representations of quasi-split inner forms
of~$\GSpinb_4$ in terms of Asai representations, and the tensor product
functoriality~$\GLb_2 \times \GLb_2 \to \GLb_4$ of~\cite{MR1792292}. We also
make use of~\cite{MR3135650} in two ways: if the similitude character is a
square, then by twisting we can immediately reduce to the results
of~\cite{MR3135650}. If the similitude character is not a square, then the
possibilities for the parameter are somewhat constrained, and  we are able to
further constrain them by using the fact that by restricting to~$\Spb_4$ and
applying the results of~\cite{MR3135650}, we know the possible forms of the
exterior square of the parameter.

In section~\ref{sec: multiplicity formula}, we prove the global multiplicity
formula in much the same way as~\cite{MR3135650}, as a consequence of the stable
(twisted) trace formulas for~$\GLb_4 \times \GLb_1$ and~$\GSpinb_5$, together
with the twisted endoscopic character relations already established.

Finally, in section~\ref{sec: local Langlands for Sp4} we show that the local
Langlands correspondences for~$\Spb_4$ established in~\cite{MR2673717}
and~\cite{MR3135650} coincide. The correspondence of~\cite{MR2673717} was
constructed by restricting the correspondence for~$\GSpb_4$ of~\cite{MR2800725}
to~$\Spb_4$, which by the results of~\cite{MR3267112} is characterised using
twisted endoscopy for $\GLb_4 \times \GLb_1$. The correspondence for $\Spb_4$
obtained in \cite{MR3135650} is characterised using twisted endoscopy for
$\GLb_5$. 

In the discrete case we prove this by a global argument, by realising the
parameter as a local factor of a cuspidal automorphic representation, and using
the exterior square functoriality for~$\GLb_4$ of~\cite{MR1937203}
and~\cite{MR2567395}. In the remaining cases the parameter arises via parabolic
induction, and we are able to treat it by hand. We are also able to use these
arguments to give a precise description in terms of Arthur parameters of the
restrictions to~$\Spb_4$ of irreducible admissible tempered representations
of~$\GSpb_4$ over a $p$-adic field.

We end this introduction with a small disclosure, and a comparison to other
work. While we have said that the results of this paper are unconditional, they
are only as unconditional as the results of~\cite{MR3135650}
and~\cite{SFTT1,SFTT2}. In particular, they depend on cases of the twisted
weighted fundamental lemma that were announced in~\cite{MR2735371}, but whose
proofs have not yet appeared in print, as well as on the references [A24],
[A25], [A26] and~[A27] in~\cite{MR3135650}, which at the time of writing have
not appeared publicly. 

The strategy of using restriction to compare the representation theory
of reductive groups related by a central isogeny is not a new one;
indeed it goes back at least as far to the comparison of~$\GLb_2$
and~$\SLb_2$ in~\cite{LabLan}. In the case of symplectic groups, there
is the paper~\cite{MR2673717} mentioned above; while this does not
make any use of trace formula techniques, we use some of its ideas in
Section~\ref{sec: local Langlands for Sp4}, when we compare the
different constructions of the local Langlands correspondence. 

More recently, there is the work of Xu, in
particular~\cite{BinXu2017Annalen,MR3568940}, which also builds
on~\cite{MR3135650}, using the groups~$\GSpb_n$ and~$\GOb_n$ where we use the
groups~$\GSpinb_n$ (of course, these cases overlap for~$\GSpb_4$). However, the
emphasis of Xu's work is rather different, and is aimed at constructing ``coarse
$L$-packets'' (which in the case of~$\GSpb_4$ are unions of $L$-packets lying
over a common $L$-packet for~$\Spb_4$), and proving a multiplicity formula for
automorphic representations grouped together in a similar way. Xu's results are
more general than ours in that they apply to groups of arbitrary rank, but are
less precise in the special case of~$\GSpb_4$, and our proofs are independent.

\subsection{Acknowledgements}We would like to thank George Boxer,
Frank Calegari, Ga\"etan Chenevier, Matthew Emerton and Wee Teck Gan for helpful
conversations.

\subsection{Notation and conventions}

\subsubsection{Algebraic groups} \label{subsubsec: notation alg groups}

We will use the boldface notation~$\bG$ for an algebraic group over a local
field or a number field, and we use the Roman version~$G$ for reductive groups
over~$\C$, or their complex points. Thus for example if~$F$ is a number field,
we will write~$\GLnb$ for the general linear group over~$F$, with Langlands dual
group $\widehat{\GLb}_n=\GL_n$, which we will also sometimes write as
$\widehat{\GLb}_n=\GL_n(\C)$.

For a real connected reductive group $\bG$, write $\mathfrak{g} = \C
\otimes_{\R} \Lie(\bG(\R))$, and let  $K$ be a maximal compact subgroup of
$\bG(\R)$. When working adelically we will sometimes abusively call
$(\mathfrak{g}, K)$-modules ``representations of $\bG(\R)$''. This should cause
no confusion as we will mostly be considering unitary representations in this
global setting (see \cite[Theorem 3.4.11]{Wallach_book1}, \cite[Theorem
4.4.6.6]{Warner_book1}), and distinguish between $(\mathfrak{g}, K)$-modules and
representations of $\bG(\R)$ when considering non-unitary representations.

\subsubsection{The local Langlands correspondence}
	\label{subsubsec: notation LL}

If~$K$ is a field of characteristic zero then we write~$\Gal_K$ for its absolute
Galois group~$\Gal(\overline{K}/K)$. If~$K$ is a local or global field of
characteristic zero, then we write~$W_K$ for its Weil group. If~$K$ is a local
field of characteristic zero, then we write~$\WD_K$ for its Weil--Deligne group,
which is~$W_K$ if~$K$ is Archimedean, and $W_K\times\SU(2)$ otherwise. 

If $\pi$ is an irreducible admissible representation of $\GLb_N(F)$ ($F$ local)
or $\GLb_N(\A_F)$ ($F$ global), then $\omega_{\pi}$ will denote its central
character. We write~$\rec$ for the local Langlands correspondence normalised as
in~\cite{ht}, so that if~$F$ is a local field of characteristic zero,
then~$\rec(\pi)$ is an $N$-dimensional representation of~$\WD_F$. If $F$ is
$p$-adic then for this normalisation a uniformiser of $F$ corresponds to the
\emph{geometric} Frobenius automorphism.

\subsubsection{The discrete spectrum}
  \label{subsubsec: discrete spectrum intro}
 
Let~$\Gb$ be a connected  reductive group over a number field~$F$. Write
\[ \bG(\A_F)^1 = \left\{ g \in \bG(\A_F)\,\middle|\, \forall \beta \in
X^*(\bG)^{\Gal_F},\,|\beta(g)|=1 \right\}, \]
so that $\bG(F) \backslash \bG(\A_F)^1$ has finite measure. Let $\mathbf{A}_\bG$
be the biggest central split torus in $\Res_{F/\Q}(\bG)$, and let
$\mathfrak{A}_\bG$ be the vector group $\mathbf{A}_\bG(\R)^0$. Then $\bG(\A_F) =
\bG(\A_F)^1 \times \mathfrak{A}_\bG$. We write
\[ \cA^2(\Gb) = \CA^2(\bG(F) \mathfrak{A}_\bG \backslash \bG(\A_F) ) =
  \CA^2(\bG(F) \backslash \bG(\A_F)^1) \]
for the space of square integrable automorphic forms. This decomposes
discretely, i.e.\ it is canonically the direct sum, over the countable set
$\Pi_{\disc}(\bG)$ of discrete automorphic representations $\pi$ for $\bG$, of
isotypical components
\[ \CA^2(\Gb)_{\pi} \]
which have finite length. 

If~$\chi_{\Gb}$ is a character of~$\mathfrak{A}_\bG$, we could more generally
consider the space of $\chi_\Gb$-equivariant square integrable automorphic forms 
\[ \cA^2(\Gb)=\CA^2(\bG(F) \backslash \bG(\A_F),\chi_{\Gb}). \]
Since we can reduce to the case~$\chi_{\Gb}=1$ considered above by twisting, we
will almost never use this more general definition.

\section{Arthur's classification}
\label{sec: Arthur}

\subsection{\texorpdfstring{$\GSpinb$}{GSpin} groups}
\label{subsec: Gspin groups}

We now recall the results announced in~\cite{MR2058604} for $\GSpb_4$, as well
as those for~$\Spb_4$ proved in~\cite{MR3135650}. In fact, for convenience we
begin by recalling the conjectural extension of Arthur's results to~$\GSpinb$
groups of arbitrary rank, and then explain what is proved in~\cite{MR3135650}.

We work with the following quasi-split groups over a local or global field~$F$
of characteristic zero:

\begin{itemize}
  \item The split groups $\GSpinb_{2n+1}$.
  \item The split groups $\Spb_{2n}\times\GLb_1$.
  \item The quasi-split groups $\GSpinb_{2n}^\alpha$.
\end{itemize}
Here we can define the groups~$\GSpinb_{2n+1}$ and~$\GSpinb_{2n}^\alpha$ as
follows. If~$\alpha\in F^\times/(F^\times)^2$, we have the quasi-split special
orthogonal group~$\SOb_{2n}^\alpha$, which is defined as the special orthogonal
group of the quadratic space given by the direct sum of~$(n-1)$ hyperbolic
planes and the plane~$F[X]/(X^2-\alpha)$ equipped with the quadratic form equal
to the norm. We have the spin double cover
\[0 \to \mu_2 \to \Spinb_{2n}^\alpha \to \SOb_{2n}^\alpha \to 0,\]
and we set
\[\GSpinb_{2n}^\alpha:=(\Spinb_{2n}^\alpha \times \GLb_1)/\mu_2\]
where~$\mu_2$ is embedded diagonally. Note that~$\GSpinb_{2n}^\alpha$ is split
if and only if~$\alpha=1$. We define the split group~$\GSpinb_{2n+1}$ in the
same way. This expedient definition is of course equivalent to the usual, more
geometric one (see \cite[Ch.\ IV, \S 6]{Knus}). The spinor norm is induced by
$(g, \lambda) \mapsto \lambda^2$. It is convenient to let $\GSpinb_0^1 =
\GSpinb_1 = \GLb_1$. 

The corresponding dual groups are as follows.
\begin{center}
  \begin{tabular}{ c | c }
    $\bG$ & $\Ghat$ \\ 
    \hline
    $\GSpinb_{2n+1}$ & $\GSp_{2n}(\C)$\\
    $\Spb_{2n}\times\GLb_1$ & $\GSO_{2n+1}(\C)=\SO_{2n+1}(\C)\times\GL_1(\C)$ \\
    $\GSpinb_{2n}^\alpha$ & $\GSO_{2n}(\C)$
  \end{tabular}
\end{center}

Let $\mu : \GLb_1 \rightarrow \bZ(\bG)$ be dual to the surjective ``similitude
factor'' morphism $\muhat : \Ghat \rightarrow \GL_1(\C)$. Note that in the
case~$\bG=\Spb_{2n} \times\GLb_1$, $\mu: \GLb_1 \to \bZ(\bG)$ is the map
$x\mapsto(1,x^2)$, and it is the only case where it is not injective. Moreover
the image of $\mu$ is $\bZ(\bG)^0$ except in the case $\bG =
\GSpinb_2^{\alpha}$.

We set $ {}^L \bG = \Ghat \rtimes \WF$, where the action of~$\WF$ on~$\Ghat$ is
trivial except in the case that $\bG=\GSpinb_{2n}^\alpha$ with~$\alpha \ne 1$,
in which case the action of~$\WF$ factors through
$\Gal(F(\sqrt{\alpha}/F)=\{1,\sigma\}$, and~$\sigma$ acts by outer conjugation
on~$\GSO_{2n}$. More precisely, in this case we identify~$\Ghat \rtimes
\Gal(F(\sqrt{\alpha}/F)$ with~$\GO_{2n}(\C)$ as follows: if $\SO_{2n}$ is
obtained from the symmetric bilinear form~$B$ on~$\C e_1 \oplus \dots \oplus \C
e_{2n}$ given by~$B(e_i,e_j)=\delta_{i,2n+1-j}$, then~$1\rtimes\sigma$ is the
element of~$\mathrm{O}_{2n}(\C)$ which interchanges~$e_{n}$ and~$e_{n+1}$ and
fixes the other~$e_i$.

We have the standard representation
\[ \Std_{\bG}: {}^L \bG \to \GL_N(\C) \times \GL_1(\C), \]
where~$N=N(\Ghat)=2n$ if $\bG=\GSpinb^\alpha_{2n}$ or~$\bG=\GSpinb_{2n+1}$,
and~$N={2n}+1$ if~$\bG=\Spb_{2n} \times \GLb_1$. In the first two cases the
representation is trivial on~$\WF$, and is given by the product of the standard
$N$-dimensional representation of~$\Ghat$ and the similitude character. In the
final case it is given by the product of the natural
inclusion~$\mathrm{O}_{2n+1}(\C)\subset\GL_{2n+1}(\C)$ and the identity
on~$\GL_1(\C)$. The standard representation realises~$\bG$ as an elliptic
twisted endoscopic subgroup of~$\GLb_N \times \GLb_1$, as we will explain below.

We set $\sign(\bG)=1$ if $\bG=\GSpinb^\alpha_{2n}$ or $\GLb_1 \times \Spb_{2n}$,
and~$\sign(\bG)=-1$ if $\bG=\GSpinb_{2n+1}$ (equivalently, we
set~$\sign(\bG)=-1$ if and only if~$\Ghat$ is symplectic).

\subsection{Levi subgroups and dual embeddings}
\label{subsec:Levi_parametrisation}

As in our description of the dual group ~$\SO_{2n}$ above, we may realise the
groups $\SOb_{2n}^{\alpha}$ and $\SOb_{2n+1}$ as matrix groups using an
antidiagonal symmetric bilinear form (block antidiagonal with a $2 \times 2$
block in the middle for $\SOb_{2n}^{\alpha}$ with $\alpha \neq 1$). Let $\bB$ be
the Borel subgroup consisting of upper diagonal elements (block upper diagonal
in the case of $\SOb_{2n}^{\alpha}$). Let $\bT$ be the subgroup of diagonal
(resp.\ block diagonal) elements. This Borel pair being given, we can now
consider standard parabolic subgroups and standard Levi subgroups. (We recall
that we only need to consider Levi subgroups up to conjugacy; indeed, given a
Levi subgroup~$\bL$ of a parabolic~$\bP$, we obtain an $L$-embedding ${}^L \bL
\into {}^L \bG$, which up to $\Ghat$-conjugacy is independent of the choice
of~$\bP$.)

It is well-known that the standard Levi subgroups are parametrised as follows.
Consider ordered partitions $n = \sum_{i=1}^r n_i + m$, where $m > 0$ if $\bG =
\SOb_{2n}^{\alpha}$ with $\alpha \neq 1$, and $m \neq 1$ if $\bG = \SOb_{2n}^1$.
Such a partition yields a standard Levi subgroup $\bL$ of $\bG$ isomorphic to
$\GLb_{n_1} \times \dots \times \GLb_{n_r} \times \bG_m$ where $\bG_m$ is a
group of the same type as $\bG$ of absolute rank $m$. Explicitly, an isomorphism
is given by
\numequation \label{eq:coord_Levi_SO}
	(g_1, \dots, g_r, h) \longmapsto \diag \left( g_1, \dots, g_r, h, S_{n_r}^{-1}
	\,{}^t g_r^{-1} S_{n_r}, \dots, S_{n_1}^{-1} \,{}^t g_1^{-1} S_{n_1} \right),
\end{equation}
where~$S_n$ denotes the antidiagonal $n\times n$ matrix with $1$'s
along the antidiagonal. For $\bG = \SOb_{2n}^1$ and $m=0$ and $n_r >1$, there
are two standard Levi subgroups of~$\bG$ corresponding to the partition $n =
\sum_{i=1}^r n_i$: the one described above and its image under the outer
automorphism of $\bG$. This completes the parameterisation of all standard Levi
subgroups of special orthogonal groups. Standard Levi subgroups of $\Spb$ and
$\GSpb$ admit a similar description.

Denote $\bG' = \GSpinb_{2n}^{\alpha}$ if $\bG = \SOb_{2n}^{\alpha}$ and $\bG' =
\GSpinb_{2n+1}$ if $\bG = \SOb_{2n+1}$. Parabolic subgroups of $\bG'$ correspond
bijectively to parabolic subgroups of ~$\bG$, and the same goes for their Levi
subgroups. Consider $\bL$ as above, and let $\bL'$ be its preimage in $\bG'$. An
easy root-theoretic exercise shows that there exists a unique isomorphism
\[ \GLb_{n_1} \times \dots \times \GLb_{n_r} \times \bG'_m \simeq \bL'  \]
lifting \eqref{eq:coord_Levi_SO} such that for any $1 \leq i \leq r$, the
composition of the induced embedding of $\GLb_{n_i}$ in $\bG'$ with the spinor
norm $\bG' \rightarrow \GLb_1$ is $\det$. Alternatively, the embeddings
$\GLb_{n_i} \rightarrow \GSpinb_{2 n_i}^1$ can be constructed geometrically
using the definition of $\GSpinb$ groups via Clifford algebras (see \cite[Ch.\
IV, \S 6.6]{Knus}), and the above parameterisation of $\bL'$ easily follows. The
conjugacy class of $\bL'$ under $\bG'(F)$ is determined by the multi-set $\{
n_1, \dots, n_r\}$.

Dually, this corresponds to identifying the dual Levi subgroup $\Lhat$ of $\Ghat
= \GSO_{2n}$ or $\GSp_{2n}$ with $\GL_{n_1} \times \dots \times \GL_{n_r} \times
\widehat{\bG_m'}$ via the block diagonal embedding:
\[ (g_1, \dots, g_r, h) \mapsto \diag \left( g_1, \dots, g_r, h, \muhat(h)
S_{n_r} \,{}^t g_r^{-1} S_{n_r}^{-1}, \dots, \muhat(h) S_{n_1} \,{}^t g_1^{-1}
S_{n_1}^{-1} \right) \]

\subsection{Endoscopic groups and transfer} \label{subsec: endoscopic
  groups and transfer}

Before stating the conjectural parameterisation, we need to recall some
definitions and results about endoscopy. We begin by recalling that an
endoscopic datum for a connected reductive group~$\bG$ over a local field $F$ is
a tuple $(\bH,\cH,s,\xi)$ (almost) as in~\cite[\S 2.1]{MR1687096}:
\begin{itemize}
	\item $\bH$ is a quasi-split connected reductive group over~$F$,
	\item $\xi : \Hhat \rightarrow \Ghat$ is a continuous embedding,
	\item $\cH$ is a closed subgroup of ${}^L \bG$ which surjects onto $W_F$ with
		kernel $\xi(\Hhat)$, such that the induced outer action of $W_F$ on
		$\xi(\Hhat)$ coincides with the usual one on $\Hhat$ transported by $\xi$,
		and such that there exists a continuous splitting $W_F \rightarrow \cH$,
	\item and $s \in \Ghat$ is a semisimple element whose connected centraliser
    in $\Ghat$ is~$\xi(\Hhat)$ and such that the map $W_F \rightarrow \Ghat$
    induced by $h \in \cH \mapsto s h s^{-1} h^{-1}$ takes values in $Z(\Ghat)$
    and is trivial in $H^1(W_F, Z(\Ghat))$.
\end{itemize}
Note that we modified the notation slightly: in \cite{MR1687096} $\cH$ is not
contained in ${}^L \bG$ and instead $\xi$ is an embedding of $\cH$ in ${}^L
\bG$. We choose this convention because in contrast to the general case where
$z$-extensions are a necessary complication, in all cases that we will consider
the embedding $\xi : \Hhat \rightarrow \Ghat$ will admit a (non-unique)
extension as ${}^L \xi : {}^L \bH \rightarrow {}^L \bG$. Of particular
importance are the elliptic endoscopic data, which are those for which the
identity component of~$\xi(Z(\Hhat)^{\Gal_F})$ is contained in~$Z(\Ghat)$.

For $\bG$ belonging to the three families introduced in Section~\ref{subsec:
Gspin groups} the groups~$\bH$ will be products whose factors are either general
linear groups, or quotients by~$\GLb_1$ of products of groups of the form
considered in Section~\ref{subsec: Gspin groups}. At this level of generality we
content ourselves with specifying the group~$\bH$, for each equivalence class of
non-trivial ($s \not\in Z(\Ghat)$) elliptic endoscopic datum of $\bG$. They are
as follows.

\begin{itemize}
	\item If $\bG=\GSpinb_{2n+1}$, then $\bH=(\GSpinb_{2a+1} \times
		\GSpinb_{2b+1})/\GLb_1$ with $a+b=n$, $ab \neq 0$, and the quotient is
		by~$\GLb_1$ embedded as $z \mapsto (\mu(z), \mu(z)^{-1})$. 
	\item If $\bG=\Spb_{2n} \times \GLb_1$, then $\bH=(\Spb_{2a} \times \GLb_1
		\times \GSpinb_{2b}^\alpha)/\GLb_1 \cong \Spb_{2a} \times \SOb_{2b}^{\alpha}
		\times \GLb_1$, where $a+b=n$, $ab \neq 0$, and $\alpha\ne 1$ if~$b=1$.
		\item If $\bG = \GSpinb_{2n}^\alpha$, then $\bH = (\GSpinb_{2a}^\beta \times
		\GSpinb_{2b}^\gamma)/\GLb_1$, where $a+b=n$, $\beta \gamma = \alpha$, $\beta
		\ne 1$ if~$a=1$, and $\gamma \ne 1$ if~$b=1$.
\end{itemize}

In this paper we will also need one case of twisted endoscopy. Recall~\cite[\S
I.1.1]{SFTT1} that if~$F$ is a local field of characteristic zero (in the paper
we will also take $F$ to be a number field), and~$\bG$ is a connected reductive
group defined over~$F$, then a twisted space~$\widetilde{\bG}$ for~$\bG$ is an
algebraic variety over~$F$ which is simultaneously a left and right torsor
for~$\bG$. Consider the split group~$\GLb_n\times\GLb_1$ over a local or global
field of characteristic zero~$F$, and let $\theta$ be the automorphism
of~$\GLb_n\times\GLb_1$ given by $\theta(g,x)=(J{}^tg^{-1}J^{-1},x\det g)$,
where $J$ is the antidiagonal matrix with alternating entries $-1,1,-1,\dots$
(that is, $J_{ij}=(-1)^i\delta_{i,n+1-j}$). The reason for defining $\theta$ in
this way is that it fixes the usual pinning $\CE$ of $\bG$ consisting of the
upper-triangular Borel subgroup, the diagonal maximal torus and $((\delta_{i,a}
\delta_{j,a+1})_{i,j})_{1 \leq a \leq n-1}$. Then $\widetilde{\bG} = \GLb_n
\times \GLb_1 \rtimes \{\theta\}$ is a twisted space which happens to be a
connected component of the non-connected reductive group $\GLb_n \times \GLb_1
\rtimes \{1,\theta\}$.

There is a notion of a twisted endoscopic datum~$(\bH,\cH,s,\xi)$ for the
pair~$(\GLb_n\times\GLb_1,\theta)$, for which we again refer to~\cite[\S
2.1]{MR1687096} (taking~$\omega$ there to be equal to~$1$, as we will throughout
this paper, and using the same convention as above for $\xi$) and \cite[\S
VI.3.1]{SFTT2}. We will explicitly describe all of the elliptic twisted
endoscopic data (up to isomorphism) in the case~$n=4$ in Section~\ref{subsec:
elliptic endoscopic for GL4 times GL1} below. In the present section we shall
only need the fact that if~$\bH$ is one of the groups considered in Section
\ref{subsec: Gspin groups} (denoted $\bG$ there), then~$\bH$ is part of an
elliptic twisted endoscopic subgroup of~$(\GLb_{N(\Hhat)}\times\GLb_1,\theta)$.

\begin{rem}
  The definitions in \cite{SFTT1} and \cite{SFTT2}, using twisted
  spaces rather than a fixed automorphism of $\bG$ (not fixing a base
  point), are more general than those used in most of \cite{MR1687096}, due to
  an assumption in \cite{MR1687096} that is only removed in (5.4) there. Note in
  particular the notion of twisted endoscopic space \cite[\S I.1.7]{SFTT1}. In
  the cases considered in this paper, where $\widetilde{\bG}$ is either $\bG$
  (standard endoscopy) or $\bG \rtimes \theta$ where $\theta \in \Aut(\bG)$
  fixes a pinning $\CE$ of $\bG$ (defined over $F$, i.e.\ stable under
  $\Gal_F$), this notion simplifies and we are under the assumption of
  \cite[(3.1)]{MR1687096}. Namely, the torsor $\CZ(\widetilde{\bG},
  \CE)$ under $\CZ(\bG) := Z(\bG)/(1-\theta)Z(\bG)$ defined in
  \cite[I.1.2]{SFTT1} is trivial with a natural base point $1 \rtimes \theta$,
  and so for any endoscopic datum $(\bH, \CH, \tilde{s}, \xi)$ for
  $\widetilde{\bG}$, the twisted endoscopic space $\widetilde{\bH} := \bH
  \times_{\CZ(\bG)} \CZ(\widetilde{\bG})$ is trivial with natural base point $1
  \rtimes \theta$, where $\theta$ now acts trivially on $\bH$. For this reason
  we can ignore twisted endoscopic spaces in the rest of the paper, and simply
  consider endoscopic groups as in most of \cite{MR1687096}.
\end{rem}

We now very briefly recall the notion of (geometric) transfer in the setting of
endoscopy. Suppose that~$F$ is a local field of characteristic zero, and that
$(\bG, \widetilde{\bG})$ belongs to one of the four families of twisted spaces
considered above, that is $\bG = \GSpinb_{2n+1}$, $\Spb_{2n} \times \GLb_1$ or
$\bG = \GSpinb_{2n}^{\alpha}$ with $\widetilde{\bG} = \bG$, or $\bG = \GLb_n
\times \GLb_1$ with $\widetilde{\bG} = \bG \rtimes \theta$. Given an endoscopic
datum~$\fe = (\bH,\cH,s,\xi)$ for~$\widetilde{\bG}$, and a choice of an
extension ${}^L \xi : {}^L \bH \rightarrow {}^L \bG$ of the embedding $\xi$,
Kottwitz and Shelstad defined \emph{transfer factors} in \cite{MR1687096}, that
is a function on the set of matching pairs of strongly regular semisimple
$\bG(F)$-conjugacy classes in $\widetilde{\bG}(F)$ and regular semisimple stable
conjugacy classes in $\bH(F)$. In general such a function is only canonical up
to $\C^{\times}$, but in all cases considered in this paper there is a Whittaker
datum $\fw = (\bU, \lambda)$ of $\bG$ fixed by an element of
$\widetilde{\bG}(F)$ and this provides \cite[\S 5.3]{MR1687096} a normalisation
of transfer factors, which we denote by $\Delta[\fe, {}^L \xi, \fw]$. To be more
precise we use the transfer factors called $\Delta_D$ in \cite{KS12},
corresponding to the normalisation of the local Langlands correspondence
identifying uniformizers to \emph{geometric} Frobenii. In all cases of ordinary
endoscopy one can choose an arbitrary Whittaker datum of $\bG$.

In the case that $\bG=\GSpinb_{2n}^\alpha$, there is an outer
automorphism~$\delta$ of~$\bG$ which preserves the Whittaker datum. This
$\delta$ can be chosen to have order $2$ and be induced by an element of the
orthogonal group having determinant $-1$; if $F$ is Archimedean, for simplicity
we can and do choose the maximal compact subgroup~$K$ of $\bG(F)$ to be
$\delta$-stable.

In this paper we are particularly interested in the case $\bG = \GSpinb_5$. By
Hilbert's theorem $90$ the morphism $\GSpinb_{2n+1}(F) \rightarrow
\SOb_{2n+1}(F)$ is surjective, so $\GSpinb_{2n+1}$ is of adjoint type and there
is up to conjugation by $\GSpinb_{2n+1}(F)$ only one Whittaker datum in this
case.

For $\widetilde{\bG} = (\GLb_n \times \GLb_1) \rtimes \theta$ we choose for
$\bU$ the subgroup of unipotent upper triangular matrices in $\GLb_n$ and
$\lambda((g_{i,j})_{i,j}) = \kappa(\sum_{i=1}^{n-1} g_{i,i+1})$ where $\kappa :
F \rightarrow S^1$ is a non-trivial continuous character. This is the Whittaker
datum associated to $\mathcal{E}$ and $\kappa$. This Whittaker datum is fixed by
$\theta$ (this is the reason for the choice of this particular $\theta$ in its
$\bG(F)$-orbit).

\begin{defn} \label{defn: Hecke alg}
  If~$F$ is $p$-adic, then we let $\cH(\widetilde{\bG})$ denote the space of
  smooth compactly supported distributions on~$\widetilde{\bG}(F)$ with
  $\C$-coefficients. Then $\cH(\widetilde{\bG}) = \varinjlim_K
  \cH(\widetilde{\bG}(F) // K)$ where the limit is over compact open subgroups
  of $\bG(F)$ and $\cH(\widetilde{\bG}(F) // K)$ is the subspace of
  bi-$K$-invariant distributions. If $F$ is Archimedean, then we fix a maximal
  compact subgroup~$K$ of~$\bG(F)$, and write~$\cH(\widetilde{\bG})$ for the
  algebra of bi-$K$-finite smooth compactly supported distributions
  on~$\widetilde{\bG}(F)$ with $\C$-coefficients.
\end{defn}

Under convolution, the space $\cH(\tbG)$ is a bi-$\cH(\bG)$-module, where
$\cH(\bG)$ is the usual (non-twisted) Hecke algebra for $\bG$.

In the case that~$\bG=\GSpinb_{2n}^\alpha$, we let $\widetilde{\cH}(\bG)$ denote
the subalgebra of~$\cH(\bG)$ consisting of $\delta$-stable distributions, and
otherwise we set~$\widetilde{\cH}(\bG)=\cH(\bG)$ and $\delta=1$. 

An admissible twisted representation of~$\widetilde{\bG}$ is by definition a
pair~$(\pi,\tpi)$ consisting of an admissible representation~$\pi$
of~$\bG(F)$ and a map~$\tpi$ from~$\widetilde{\bG}$ to the automorphism group of
the underlying vector space of~$\pi$, which satisfies
\[ \tpi(g \gamma g') = \pi(g) \tpi(\gamma) \pi(g')\]
for all $g,g'\in \bG(F)$, $\gamma \in \widetilde{\bG}$. (This is the special
case~$\omega=1$ of the notion of an $\omega$-representation of a twisted space,
which is defined in~\cite{SFTT1}.) If $F = \R$ or $\C$ there is an obvious
notion of $(\mathfrak{g}, \widetilde{K})$-module where $\widetilde{K} \subset
\bG(F)$ is a torsor under $K$ normalising $K$.

We will consider (invariant) linear forms on~$\widetilde{\cH}(\widetilde{\bG})$.
In particular, for each admissible representation~$\pi$ of~$\bG(F)$, there is
the linear form
\[ \tr(\pi(f(g)dg)) = \tr \left( \int_{G(F)} f(g) \pi(g) dg \right). \]
If $F$ is Archimedean and $\pi$ is an admissible $(\fg, \widetilde{K})$-module
the action of $\widetilde{\cH}(\widetilde{\bG})$ is not obviously well-defined
but it is so when $\pi$ arises as the space of $K$-finite vectors of an
admissible Banach representation of $\widetilde{\bG}(F)$, independently of the
choice of this realisation (see \cite[p.\ 326, Theorem 4.5.5.2]{Warner_book1}).
In this paper all $(\fg, \widetilde{K})$-modules will naturally arise in this
way, even with ``Hilbert'' instead of ``Banach'', although not all of them will
be unitary.

We write~$I(\widetilde{\bG})$ for the quotient
of~$\widetilde{\cH}(\widetilde{\bG})$ by the subspace of those
distributions~$f(g)dg$ with the property that for any semisimple strongly
regular $\gamma\in \widetilde{\bG}(F)$, the orbital integral~$O_\gamma(f(g)dg)$
vanishes. There is a natural topology on $I(\widetilde{\bG})$: see
\cite[I.5.2]{SFTT1}. Similarly, we write~$SI(\widetilde{\bG})$ for the quotient
by the subspace for which the stable orbital integrals~$SO_\gamma(f(g)dg)$
vanish. We say that a continuous linear form on~$\widetilde{\cH}(\bG)$ is stable
if it descends to a linear form on~$SI(\widetilde{\bG})$.

Given an endoscopic datum~$(\bH,\cH,s,\xi)$ for~$\widetilde{\bG}$, and our
choice of Whittaker datum, there is a notion of transfer
from~$I(\widetilde{\bG})$ to~$SI(\bH)$ (see \cite[\S 5.5]{MR1687096}, \cite[\S
I.2.4 and IV.3.4]{SFTT1}); this transfer is defined by the property that it
relates the values of orbital integrals on~$\widetilde{\bG}$ to stable orbital
integrals on~$\bH$, using the transfer factors recalled above. Most importantly,
this transfer \emph{exists} (\cite{MR1440722}, \cite{MR2653248},
\cite{Shelstad_twistedtransfer}). Dually, we may transfer \emph{stable}
continuous linear forms on~$\widetilde{\cH}(\bH)$ to continuous linear forms
on~$\widetilde{\cH}(\bG)$.

In the twisted case where $\widetilde{\bG} = \left( \GLb_N \times \GLb_1 \right)
\rtimes \theta$ over a $p$-adic field $F$, the chosen Whittaker datum yields a
hyperspecial maximal compact subgroup $K$ of $\bG(F)$ (see \cite{MR581582}),
which is stable under $\theta$, so it is natural to consider the hyperspecial
subspace (see \cite[\S I.6]{SFTT1}) $\widetilde{K} = K \rtimes \theta$ of
$\widetilde{\bG}(F)$. For any unramified endoscopic datum $(\bH, \CH, \tilde{s},
\xi)$ for $\widetilde{\bG}$ (also defined in \cite[\S I.6]{SFTT1}), with the
above trivialisation of $\widetilde{\bH}$, the associated $\bH_{\ad}(F)$-orbit
of hyperspecial subspaces of $\widetilde{\bH}$ is simply the obvious one, that
is the set of $K' \rtimes \theta$ where $K'$ is a hyperspecial maximal compact
subgroup of $\bH(F)$.

By the existence of transfer and \cite{TFLsmallchar1}, \cite{TFLsmallchar2}
(\cite{Hales} in the case of standard endoscopy), the twisted fundamental lemma
is now known for all elements of the unramified Hecke algebra, with no
assumption on the residual characteristic. We formulate it in our situation,
which is slightly simpler than the general case by the above remarks.

\begin{thm} \label{thm: fundamental lemma}
  Let $\widetilde{\bG}$ be a twisted group over a $p$-adic field $F$ belonging
  to one of the four families introduced at the beginning of this section.
  Assume that $\bG$ is unramified. Let $(\bH, \CH, \tilde{s}, \xi)$ be an
  unramified endoscopic datum for $\widetilde{\bG}$. Choose an unramified
  L-embedding ${}^L \xi : {}^L \bH \rightarrow {}^L \bG$ extending $\xi$. Let
  $\widetilde{K}$ be the hyperspecial subspace of $\widetilde{\bG}(F)$
  associated to the chosen Whittaker datum for $\bG$. Let $1_{\widetilde{K}}$ be
  the characteristic function of $\widetilde{K}$ multiplied by the
  $\bG(F)$-invariant measure on $\widetilde{\bG}(F)$ such that $\widetilde{K}$
  has volume $1$. Let $b : \CH(\bG(F_v)//K_v) \rightarrow \CH(\bH(F_v) // K'_v)$
  be the morphism dual to
	\[ \left( \Hhat \rtimes \Frob \right)^{\semis} / \Hhat-\mathrm{conj}
	\rightarrow \left( \Ghat \rtimes \Frob \right)^{\semis} /
	\Ghat-\mathrm{conj} \]
  via the Satake isomorphisms \emph{(}see \cite[\S 7]{MR546608}\emph{)}. Then
  for any $f \in \CH(\bG(F_v) // K)$, $b(f)$ is a transfer of $f \ast
  1_{\widetilde{K}}$.
\end{thm}

\begin{rem} \label{rem:twisted_unr_rep}
  In the above setting, there is a natural notion of unramified twisted
  representation: extend an unramified representation $(\pi, V)$ of $\bG(F)$
  which is isomorphic to its twist by $\tbG(F)$ to a twisted representation by
  imposing that $\tilde{K}$ acts trivially on $V^K$.
\end{rem}

\subsection{Local parameters}
\label{subsec: local parameters}

Let~$F$ be a local field of characteristic zero. Let $\Psi^+(G)$
denote the set of $\Ghat$-conjugacy classes of continuous morphisms
\[ \psi: \WD_F \times \SL_2(\C) \to {}^L \bG \]
such that
\begin{itemize}
	\item the composite with the projection ${}^L \bG \to \WF$ is the natural
		projection $\WD_F \times \SL_2(\C) \to \WF$,
	\item for any~$w\in \WD_F$, $\psi(w)$ is semisimple, and
	\item the restriction $\psi|_{\SL_2(\C)}$ is algebraic.
\end{itemize}
We let~$\Psi(\bG)\subset\Psi^+(\bG)$ be the subset of bounded parameters.
By a standard argument (see for example the proof of~\cite[Lem.\
6.1]{MR2800725}), the $\{1,\deltahat\}$-orbit of a parameter~$\psi$ is
determined by the data of the conjugacy class of~$\Std_{\bG} \circ \psi$. Let
$\tPsi(G)$ and $\tPsi^+(\bG)$ be the set of $\{1,\deltahat\}$-orbits of
parameters as above.

For $\psi \in \Psi^+(\bG)$ let $\varphi_{\psi}$ be the Langlands parameter
associated to $\psi$, that is $\psi$ composed with the embedding
\[ w \in \WD_F \longmapsto \left( w, \mathrm{diag}(|w|^{1/2}, |w|^{-1/2})
\right) \in \WD_F \times \SL_2(\C). \]

We write $C_\psi$ for the centraliser of~$\psi$ in~$\Ghat$,
$S_\psi=Z(\Ghat)C_\psi$, and
\[\cS_\psi=\pi_0(S_\psi/Z(\Ghat)), \]
an abelian 2-group. We let~$\cS_\psi^\vee=\Hom(\cS_\psi,\C^\times)$ be the
character group of~$\cS_\psi$. Write~$s_\psi$ for the image in~$C_\psi$ of
$-1 \in \SL_2(\C)$.

We can now formulate the conjectures on local Arthur packets in terms
of endoscopic transfer relations.
\begin{conj} \label{conj:local_Arthur_packets GSpin}
  Let $\bG = \GSpinb_{2n+1}$, $\Spb_{2n} \times \GLb_1$ or
  $\GSpinb_{2n}^{\alpha}$. Then there is a unique way to associate to
	each~$(\psi) \in \tPsi(\bG)$ a multi-set $\Pi_{\psi}$ of $\{1,\delta\}$-orbits
	of irreducible smooth unitary representations of $\bG(F)$, together with a map
	$\Pi_{\psi} \rightarrow \CS_{\psi}^{\vee}$, which we will denote by $\pi
	\mapsto \langle \cdot, \pi \rangle$, such that the following properties hold.
	\begin{enumerate}
		\item Let $\pi_{\psi}^{\GLb}$ be the representation of
			$\GLb_{N(\Ghat)}(F) \times \GLb_1(F)$ associated to $(\Std_{\bG} \circ
      \varphi_{\psi})$ by the local Langlands correspondence for
      $\GLb_{N(\Ghat)}\times\GLb_1$, and let $\widetilde{\pi}_{\psi}^{\GLb}$ be
      its extension to $\left( \GLb_{N(\Ghat)}(F) \times \GLb_1(F)
        \right)\rtimes\theta$ recalled in Section \ref{subsec: Whittaker
      normalisation}. Then $\sum_{\pi \in \Pi_{\psi}} \langle s_{\psi}, \pi
      \rangle \tr \pi$ is stable and its transfer to $\GLb_{N(\Ghat)}(F) \times
      \GLb_1(F) \rtimes \theta$ is $\tr \widetilde{\pi}_{\psi}^{\GLb}$, i.e.\
      for any $f \in I(\left( \GLb_{N(\Ghat)}(F) \times \GLb_1(F)
      \right)\rtimes\theta)$ having transfer $f' \in SI(\bG)$ we have
      \[ \tr \widetilde{\pi}_{\psi}^{\GLb}(f) = \sum_{\pi \in \Pi_{\psi}}
      \langle s_{\psi}, \pi \rangle \tr \pi(f'). \]
    \item Consider a semisimple $s \in C_\psi$ with image $\bar{s}$ in
      $\CS_{\psi}$. The pair $(\psi, s)$ determines an endoscopic datum $(\bH,
      \CH, s, \xi)$ for $\bG$ \emph{(}with $\CH = \Cent(s, \Ghat)
      \psi(\WD_F)$\emph{)}, and if we fix an L-embedding ${}^L \xi : {}^L \bH
      \rightarrow {}^L \bG$ extending $\xi$ we obtain $\psi' : \WD_{F} \times
      \SL_2(\C) \rightarrow {}^L \bH$ such that $\psi = {}^L \xi \circ \psi'$.
      Then for any $f \in I(\bG)$ with transfer $f'\in SI(\bH)$, we have:
			\[ \sum_{\pi \in \Pi_{\psi}} \langle \bar{s} s_{\psi}, \pi \rangle \tr
				\pi(f) = \sum_{\pi' \in \Pi_{\psi'}} \langle s_{\psi'}, \pi' \rangle \tr
			\pi'(f'). \]
		\item If $\psi|_{\SL_2(\C)}=1$, then the elements of~$\Pi_\psi$ are tempered
      and $\Pi_{\psi}$ is multiplicity free, and the map~$\Pi_{\psi}
      \to\cS_\psi^\vee$ is injective; if~$F$ is non-Archimedean, then it is
      bijective. Every tempered irreducible representation of~$\bG(F)$ belongs
      to exactly one such~$\Pi_\psi$.
	\end{enumerate}
\end{conj}
\begin{rem}
  Note that the uniqueness of the classification is clear from properties~(1)
  and~(2) and Proposition \ref{prop: surj transfer} below, as irreducible
  representations are determined by their traces.
\end{rem}

\begin{prop}[Arthur] \label{prop: surj transfer}
	In the situation of Conjecture \ref{conj:local_Arthur_packets GSpin}, the
	transfer map $I(\widetilde{\GLb_{N(\Ghat)} \times \GLb_1}) \rightarrow
	SI(\bG)^{\delta}$ is surjective.
\end{prop}
\begin{proof}
	This is \cite[Cor.\ 2.1.2]{MR3135650} slightly generalised from
	$\widetilde{\GLb_N}$ to $\widetilde{\GLb_N \times \GLb_1}$. Note that the
	general version of \cite[Prop.\ 2.1.1]{MR3135650} was later proved in
	\cite[\S I.4.11]{SFTT1} (see \S IV.3.4 loc.\ cit.\ to extend to the
	Archimedean case with $K$-finiteness).
\end{proof}

\begin{rem} \label{rem: extension of packets}
	Part~(3) of this conjecture gives the local Langlands correspondence for
  tempered representations of~$\bG(F)$ (up to outer conjugacy in case $\bG =
  \GSpinb_{2n}^{\alpha}$). It can be extended to give the local Langlands
  correspondence for all local parameters ~$\psi\in\Psi^+(\bG)$
  with~$\psi|_{\SL_2(\C)}=1$; indeed if
  Conjecture~\ref{conj:local_Arthur_packets GSpin} is known for all~$\bG$, then
  a version can be deduced for~$\Psi^+(\bG)$ using the Langlands classification
  (see \cite{langlandsrg}, \cite{Silberger_Langlandsquot} and
  \cite{SilbergerZink_Langlandsparam}).
\end{rem}

\begin{rem} \label{rem:conj_local_Apack_real}
  In the case where $F$ is Archimedean and for an arbitrary reductive group the
  local Langlands correspondence was established by Langlands and Shelstad (see
  \cite{Shelstad_temp_end2}, \cite{Shelstad_temp_end3}). Compatibility with
  twisted endoscopy was proved by Mezo \cite{Mezo} (under a minor assumption,
  see (3.10) loc.\ cit., which is satisfied in all cases considered in the
  present article) up to a constant which a priori might depend on the parameter
  (see \cite[Annexe C]{AMR}).
\end{rem}

\begin{rem} \label{rem: unramified packets}
  If~$F$ is $p$-adic and~$\bG$ is unramified over~$F$, then there is a unique
  $\bG(F)$-conjugacy class of hyperspecial maximal compact subgroups of~$\bG(F)$
  which is compatible with the Whittaker datum fixed above (in the sense
  of~\cite{MR581582}), and we will say that a representation of~$\bG(F)$ is
  unramified if it is unramified with respect to a subgroup in this conjugacy
  class.

  If $\psi \in \tPsi^+(G)$ and $\psi|_{\WD_{F}}$ is unramified, then assuming
  the conjecture the packet~$\Psi_\psi$ contains a unique unramified
  (orbit of) representation. It has Satake parameter $\varphi_{\psi}$ (up to
  outer conjugation if $\bG = \GSpinb_{2n}^{\alpha}$) and corresponds to the
  trivial character on~$\cS_{\psi}$. This follows from the fundamental lemma
  (Theorem \ref{thm: fundamental lemma}).
\end{rem}

\begin{rem}
 By \cite{MoeglinClayMathProc} if $F$ is $p$-adic and the conjecture holds then
 the packets $\Pi_{\psi}$ are sets rather than multi-sets.
\end{rem}

\subsection{Global parameters and the conjectural multiplicity formula}
	\label{subsec: global parameters}

Now let~$F$ be a number field, and fix a continuous unitary character
$\chi:\A_F^\times/F^\times\to\C^\times$. If~$\pi$ is a cuspidal automorphic
representation of~$\GLb_N/F$ such that $\pi^\vee\otimes\chi\cong\pi$, then we
say that~$\pi$ is~$\chi$-self dual. Note that this implies
that~$\omega_\pi^2=\chi^N$ (so in particular if~$N$ is odd,
then~$\chi=(\omega_\pi\chi^{(1-N)/2})^2$ is a square).

If~$\pi$ is~$\chi$-self dual and $S$ is a big enough set of places of $F$ then
precisely one of the~$L$-functions $L^S(s, \chi^{-1} \otimes \bigwedge^2(\pi))$
and $L^S(s, \chi^{-1} \otimes \Sym^2(\pi))$ has a pole at $s=1$, and this pole
is simple (see~\cite{MR1610812}). In the former case we say that~$(\pi, \chi)$
is of symplectic type, and set~$\sign(\pi, \chi)=-1$, and in the latter we say
that it is of orthogonal type, and we set~$\sign(\pi, \chi)=1$.

We write~$\Psi(\widetilde{\GLb_N\times\GLb_1},\chi)$ for the set of  formal
unordered sums~$\psi=\boxplus_i\pi_i[d_i]$, where the~$\pi_i$ are $\chi$-self
dual automorphic representations for~$\GLb_{N_i}/F$ and the ~$d_i\ge 1$ are
integers (which are to be thought of as the dimensions of irreducible algebraic
representations of~$\SL_2(\C)$), with the property that~$\sum_iN_id_i=N$. We
refer to such a sum as a \emph{parameter}, and say that it is \emph{discrete} if
the (isomorphism classes of) pairs~$(\pi_i,d_i)$ are pairwise distinct.

\begin{rem} \leavevmode
  \begin{enumerate}
    \item
      By the main result of~\cite{MR1026752}, a discrete automorphic
      representation~$\pi$ of $\GLb_N/F$ with $\pi^{\vee} \otimes \chi \cong
      \pi$ gives rise to an element
      of~$\Psi(\widetilde{\GLb_N\times\GLb_1},\chi)$. Indeed, there is a natural
      bijection between such representations~$\pi$ and the elements
      of~$\Psi(\widetilde{\GLb_N\times\GLb_1},\chi)$ of the form~$\pi[d]$ (that
      is, the elements where the formal sum consists of a single term). We will
      use this bijection without further comment below.
    \item The set of formal parameters $\Psi(\widetilde{\GLb_N\times\GLb_1},
      \chi)$ that we consider does not contain all non-discrete $\chi$-self-dual
      parameters, for example those containing a summand of the form $\pi
      \boxplus (\chi \otimes \pi^{\vee})$ for a non-$\chi$-self-dual cuspidal
      automorphic representation $\pi$ for $\GLb_m$. Our ad hoc definition will
      turn out to be convenient when we will consider the discrete part of (the
      stabilisation of) trace formulas.
  \end{enumerate}
\end{rem}

\begin{defn}
  Let $\bG = \GSpinb_{2n+1}$, $\Spb_{2n} \times \GLb_1$ or
  $\GSpinb_{2n}^{\alpha}$ over $F$. We let $\tPsi_{\disc}(\bG,\chi)$ be the
  subset of~$\tPsi(\GLb_{N(\bG)},\chi)$ given by
  those~$\psi=\boxplus_i\pi_i[d_i]$ with the properties that
  \begin{itemize}
    \item $\psi$ is discrete,
    \item for each~$i$, we have $\sign(\pi_i, \chi) = (-1)^{d_i-1} \sign(\bG)$,
    \item if $\bG=\GSpinb_{2n}^\alpha$, then $\chi^{-n} \prod_i
      \omega_{\pi_i}^{d_i}$ is the quadratic character corresponding to the
      extension~$F_\alpha/F$.
  \end{itemize}
  (Conditions analogous to this last bullet point could be formulated
  for the other groups~$\bG$, but in fact they are conjecturally
  automatically satisfied.)
\end{defn}

If $\bG \neq \GSpinb_{2n}^{\alpha}$ we also let $\Psi_{\disc}(\bG,\chi) =
\tPsi_{\disc}(\bG, \chi)$. The reason for writing $\tPsi$ in the case of even
$\GSpinb$ groups is that this set only sees orbits of (substitutes for)
Arthur-Langlands parameters under outer conjugation.

As a particular case of the above definition, for $\pi$ a cuspidal automorphic
representation for $\GLb_N / F$ such that $\chi \otimes \pi^{\vee} \simeq \pi$
there is a unique group $\bG$ as above such that $N(\Ghat) = N$ and $\pi[1] \in
\tPsi_{\disc}(\bG)$.

\begin{conj} \label{conj: global parameters give local parameters GSpin}
  For $\pi$ and $\bG$ as above and for each place~$v$ of $F$, the
  representation~$(\rec(\pi_v), \rec(\chi_v))$ factors through $\Std_{\bG}: {}^L
  \bG \to \GL_{N(\Ghat)}(\C) \times \GL_1(\C)$, so that we can regard~$(\pi_v,
  \chi_v)$ as an element of $\tPsi^+(\bG(F_v))$.
\end{conj}

\begin{rem} \leavevmode
  \begin{enumerate}
    \item This conjecture is the analogue of \cite[Theorem 1.4.1]{MR3135650}
      (reformulated using Theorem 1.5.3 loc.\ cit.). In particular it holds for
      $\bG = \Spb_{2n} \times \GLb_1$.
    \item Since we do not know the generalised Ramanujan conjecture
      for~$\GLb_n$, and do not wish to assume it, we can at present only hope to
      establish that the local parameters $\psi_v$ are elements
      of~$\tPsi^+(\bG_{F_v})$; they are, however, expected to be elements
      of~$\tPsi(\bG_{F_v})$. 
  \end{enumerate}
\end{rem}

Given a global parameter~$\psi \in \tPsi_{\disc}(\bG, \chi)$, we define groups
$C_{\psi}$, $S_\psi$, $\cS_{\psi}$ as follows. For each~$i$, there is a unique
group~$\bG_i$ of the kind we are considering for
which~$\pi_i\in\tPsi_{\disc}(\bG_i,\chi)$. We let~$\cL_\psi$ denote the fibre
product of the ${}^L \bG_i$ over~$\WF$. Then there is a map $\dot{\psi}:
\cL_{\psi} \times \SL_2(\C) \to {}^L \bG$ such that $\Std_{\bG} \circ
\dot{\psi}$ is conjugate to~$\oplus_i \Std_{\bG_i} \otimes \nu_{d_i}$, where
~$\nu_{d_i}$ is the irreducible representation of~$\SL_2(\C)$ of
dimension~$d_i$. The map $\dot{\psi}$ is well-defined up to the action of
$\Aut({}^L \bG)$. We let $C_\psi$ be the centraliser of~$\dot{\psi}$, and
similarly define~$S_{\psi}$ and~ $\cS_\psi$.

For each finite place~$v$, under Conjecture \ref{conj: global parameters give
local parameters GSpin} (applied to the $\pi_i$'s) we may form a local
Arthur-Langlands parameter~$\psi_v^0: \WD_{F_v} \times \SL_2(\C) \to
\CL_{\psi}$. Composing with $\dot{\psi}$, we obtain $\psi_v \in
\tPsi^+(\bG_{F_v})$. The composition of $\psi_v$ with $\Std_{\bG}$ is given by
\begin{itemize}
  \item $\chi_v$ on the~$\GL_1$ factor,
  \item the direct sum of the representations $\varphi_{\pi_i, v} \otimes
    \nu_{d_i}$ on the~$\GL_{N(\Ghat)}$ factor, where
    $\varphi_{\pi_i, v}=\rec(\pi_{i,v})$.
\end{itemize}

Conjecture~\ref{conj: global multiplicity formula GSpin} below makes precise the
expectation that the elements of the corresponding multi-sets ~$\Pi_{\psi_v}$
of~Conjecture~\ref{conj:local_Arthur_packets GSpin} are the local factors of the
discrete automorphic representations of~$\bG$ with multiplier~$\chi$. Before
stating it, we need to introduce some more notation and terminology.

For each place~$v$ of~$F$, write~$\widetilde{\cH}(\bG_v)$ for the Hecke
algebra defined after Definition \ref{defn: Hecke alg}, and
write~$\widetilde{\cH}(\bG)$ for the restricted tensor product of
the~$\widetilde{\cH}(\bG_v)$. Assuming Conjecture~\ref{conj: global parameters
give local parameters GSpin}, we have an obvious map~$\cS_\psi\to\cS_{\psi_v}$
for each~$v$, and we can associate to~$\psi$ a global packet (a multi-set) of
representations of $\widetilde{\cH}(\bG)$:
\[ \widetilde{\Pi}_\psi:=\{\otimes'_v \pi_v: \pi_v\in\Pi_{\psi_v} \text{ with
	}\pi_v\text{ unramified for all but finitely many }v\}. \]
For each~$\pi \in \widetilde{\Pi}_\psi$, we have the associated character
on~$\cS_\psi$,
\[\langle x,\pi\rangle:=\prod_v\langle x_v,\pi_v\rangle\]
(note that by Remark~\ref{rem: unramified packets}, we have
$\langle\cdot,\pi_v\rangle=1$ for all but finitely many~$v$, so this product
makes sense).

Associated to each~$\psi$ is a character
$\epsilon_\psi:\cS_\psi\to\{\pm 1\}$ which can be defined explicitly
in terms of symplectic $\epsilon$-factors. In the case~$\chi=1$ this
is defined in~\cite[Theorem 1.5.2]{MR3135650}, and this definition can
be extended to the case of general~$\chi$ without difficulty. Since we
will only need the case~$\Gb=\GSpin_5$ in this paper, and in this case
the characters~$\epsilon_\psi$ are given explicitly
in~\cite{MR2058604} and are recalled below in Remark~\ref{rem:
  explicit list of parameters following Arthur}, we do not give the
general definition here.

\begin{defn} \label{defn: Pi psi epsilon}
  $\widetilde{\Pi}_\psi(\epsilon_\psi)$ is the subset of $\widetilde{\Pi}_\psi$
  consisting of those elements for
  which~$\langle\cdot,\pi\rangle=\epsilon_\psi$.
\end{defn}

This is the correct definition only because the groups $\CS_{\psi_v}$ are all
abelian.

Recall that we have fixed a maximal compact subgroup $K_{\infty}$ of $\bG(F
\otimes_{\Q} \R)$ in Section \ref{subsec: endoscopic groups and transfer}. Let
$\fg = \C \otimes_{\R} \Lie(\bG(F \otimes_{\Q} \R))$. We write
$\CA^2(\bG(F)\backslash \bG(\A_F),\chi)$ for the space of~$\chi$-equivariant
(where the action of $\A_F^\times/F^\times$ is via~$\mu$) square integrable
automorphic forms on $\bG(F)\backslash \bG(\A_F)$. It decomposes discretely
under the action of~$\bG(\A_{F,f}) \times (\fg, K_{\infty})$.
\begin{conj}
  \label{conj: global multiplicity formula GSpin}
	Assume that Conjectures~\ref{conj:local_Arthur_packets GSpin} and~\ref{conj:
	global parameters give local parameters GSpin} hold.
	Then there is an isomorphism of $\widetilde{\cH}(\bG)$-modules
	\begin{equation*}
		\CA^2(\bG(F) \backslash \bG(\A_F),\chi) \cong
		\bigoplus_{\psi\in\widetilde{\Psi}_{\disc}(\bG,\chi)} m_\psi \bigoplus_{\pi
		\in \widetilde{\Pi}_\psi(\epsilon_\psi)} \pi,
	\end{equation*}
  where $m_\psi=1$ unless~$\bG=\GSpinb_{2n}^\alpha$, in which case $m_\psi=2$ if
  and only if each~$N_i$ is even.
\end{conj}

\subsection{The results of~\texorpdfstring{\cite{MR3135650}}{Arthur}}

As we have already remarked, the conjectures above are all proved
in~\cite{MR3135650} in the case that~$\chi=1$. As we now explain, the case
that~$\chi$ is a square follows immediately by a twisting argument. The main
results of this paper are a proof of Conjectures~\ref{conj:local_Arthur_packets
GSpin} (Theorem \ref{thm:local_Arthur_packets}) and~\ref{conj: global
multiplicity formula GSpin} (Theorem \ref{thm:glob_mult_form_GSpin5}) in the
case that~$\bG=\GSpinb_5 \cong \GSpb_4$ for general~$\chi$. Conjecture
\ref{conj: global parameters give local parameters GSpin} for $\bG = \GSpinb_5$
is a consequence of \cite{MR2800725}, see Proposition \ref{prop: L parameter of
symplectic is symplectic}. The case that~$\chi$ is a square will be a key
ingredient in our arguments, as if~$\chi$ is not a square, then it is easy to
see that there are considerably fewer possibilities for the parameters~$\psi$,
and this will reduce the number of ad hoc arguments that we need to make.
Moreover in the remaining cases, the statements pertaining to local tempered
representations are covered by \cite{MR3267112}.

\begin{thm}[Arthur]
  \label{thm: Arthur for chi a square}
	If~$\chi=\eta^2$ is a square, then Conjectures~\ref{conj:local_Arthur_packets
	GSpin}, \ref{conj: global parameters give local parameters GSpin}
	and~\ref{conj: global multiplicity formula GSpin} hold.
\end{thm}
\begin{proof}
  Given a~$\chi$-self dual cuspidal automorphic representation~$\pi$, the twist
  $\pi\otimes\eta^{-1}$ is self dual. Similarly, we may twist the local
  parameters by the restriction to~$W_{F_v}$ of the character corresponding
  to~$\eta^{-1}$, and we can also twist representations of~$\bG(F)$
  and~$\bG(F_v)$ by~$\eta^{-1}$. All of the conjectures are easily seen to be
  compatible with these twists, so we reduce to the case~$\chi=1$. In this case,
  representations of~$\GSpinb_{2n+1}$, (resp.\ $\GSpinb_{2n}^\alpha$, resp.\
  $\Spb_{2n} \times \GLb_1$) with trivial similitude factor (recall that this
  was defined in Section \ref{subsec: Gspin groups} as the composition of the
  central character with $\mu$) are equivalent to representations
  of~$\SOb_{2n+1}$, (resp.\ representations of $\SOb_{2n}^\alpha$, resp.\
  pairs given by a representation of $\Spb_{2n}$ and a character of $\GLb_1$
  of order $1$ or $2$), so the conjectures are equivalent to the main results
  of~\cite{MR3135650}.
\end{proof}

In particular, since in the case~$\bG = \Spb_{2n} \times \GLb_1$ the
character~$\chi$ is always a square, Theorem~\ref{thm: Arthur for chi a square}
always holds in this case.

\subsection{Low rank groups}
\label{subsec: low rank groups}

If $N(\Ghat) \le 3$ then Conjectures~\ref{conj:local_Arthur_packets GSpin},
\ref{conj: global parameters give local parameters GSpin} and~\ref{conj: global
multiplicity formula GSpin} also hold unconditionally.
\begin{enumerate}
  \item If $N=1$ the results are tautological.
  \item if $N=2$ then $\bG=\GSpinb_3$ or $\bG=\GSpinb_2^\alpha$. In the first
    case $\bG \simeq \GLb_2$ and the results are also tautological. In the
    second case where $\bG = \GSpinb_2^{\alpha} \simeq \Res_{F(\sqrt{\alpha}) /
    F}(\GLb_1)$ we are easily reduced to the well-known Theorem \ref{thm:
    symplectic orthogonal GL2} below, the symplectic/orthogonal alternative
    for~$\GLb_2$.
  \item If $N=3$ then $\bG=\Spb_2 \times \GLb_1$ and we are reduced to a
    special case of Theorem \ref{thm: Arthur for chi a square}. Note that the
    local Langlands correspondence and the multiplicity formula in this case go
    back to Labesse--Langlands \cite{LabLan} and \cite{MR1792292}.
\end{enumerate}

\begin{thm} \label{thm: symplectic orthogonal GL2}
  Let~$\pi$ be a $\chi$-self dual cuspidal automorphic representation of
  $\GLb_2$. Then either
	\begin{enumerate}
		\item $\chi=\omega_\pi$, and $L^S(s,\bigwedge^2(\pi)\otimes\chi^{-1})$ has a
			pole at~$s=1$; or
		\item $\omega_\pi\chi^{-1}$ is the quadratic character given by some
			quadratic extension~$E/F$, $\pi$ is the automorphic induction of a
			character of~$\A_E^\times/E^\times$ which is not fixed by the non-trivial
			element of $\Gal(E/F)$, and $L^S(s,\Sym^2(\pi)\otimes\chi^{-1})$ has a
			pole at~$s=1$.
	\end{enumerate}
\end{thm}
\begin{proof}
	Certainly $L^S(s,\bigwedge^2(\pi)\otimes\chi^{-1})=L^S(s,\omega_\pi\chi^{-1})$
	has a pole at~$s=1$ if and only if $\chi=\omega_\pi$. So if
	$L^S(s,\Sym^2(\pi)\otimes\chi^{-1})$ has a pole at~$s=1$, we see
	that~$\omega_\pi\chi^{-1}$ is a non-trivial quadratic character corresponding
	to an extension~$E/F$. Since we always have~$\pi^\vee \otimes \omega_{\pi}
	\cong \pi$, this implies that $\pi \cong \pi \otimes (\omega_\pi\chi^{-1})$,
	and it follows (see~\cite[end of \S 2]{MR574808}) that~$\pi$ is the
	automorphic induction of a character of~$\A_E^\times/E^\times$ which is not
	fixed by the non-trivial element of $\Gal(E/F)$.
\end{proof}

\subsection{The local Langlands correspondence for
~\texorpdfstring{$\GSpb_4$}{GSp(4)}}

Let~$F$ be a $p$-adic field. The local Langlands correspondence for~$\GSpb_4(F)$
was established in~\cite{MR2800725}, but was characterised by relations with
$\gamma$-factors, rather than endoscopic character relations. The necessary
endoscopic character relations were then proved in~\cite{MR3267112}. In
particular, we have:

\begin{thm}[Chan--Gan]
	\label{thm: chan gan}
  If $F$ is a $p$-adic field then Conjecture~\ref{conj:local_Arthur_packets
  GSpin} holds for~$\GSpinb_5$ and parameters $\psi$ which are trivial on
  $\SL_2(\C)$, i.e.\ tempered Langlands parameters.
\end{thm}
\begin{proof}
	Parts~(1) and~(2) of Conjecture~\ref{conj:local_Arthur_packets GSpin} are an
  immediate consequence of the main theorem of~\cite{MR3267112} (note that
  bounded parameters are automatically generic, in the sense that their adjoint
  $L$-functions are holomorphic at~$s=1$).  Part~(3) then follows from the main
  theorem of~\cite{MR2800725}.
\end{proof}

\begin{rem} \label{rem: LL GSp4 Archi}
  Recall from Remark \ref{rem:conj_local_Apack_real} that over an Archimedean
  field the local Langlands correspondence and (ordinary) endoscopic character
  relations are known in complete generality, and the twisted endoscopic
  character relations are known up to a constant (which might depend on the
  parameter).

  If $F$ is Archimedean and $\psi$ is a tempered and non discrete Langlands
  parameter for $\GSpinb_5$, then the twisted endoscopic character relation was
  verified in \cite[\S 6]{MR3267112}, which amounts to saying that the above
  constant (the only ambiguity in Mezo's theorem) is $1$. In
  Proposition~\ref{prop: twisted endoscopic character real discrete tempered}
  below we will show using a global argument as in \cite[Annexe
  C]{AMR} that this also holds for the
  discrete tempered $\psi$.
\end{rem}

\section{Construction of missing local Arthur packets for $\GSpinb_5$}
\label{sec: missing local packets}
\subsection{Local packets}\label{subsec: local packets}

Let $F$ be a local field of characteristic zero. In this section we
complete the proof of the following theorem, which completes the proof of
Conjecture~\ref{conj:local_Arthur_packets GSpin} for~$\GSpinb_5$.

\begin{theorem} \label{thm:local_Arthur_packets}
  Let $\psi : \WD_F \times \SL_2 \rightarrow \GSp_4$ be an element of
  $\Psi(\GSpinb_5)$. Then there is a unique multi-set $\Pi_{\psi}$ of
  irreducible smooth unitary representations of $\GSpinb_5(F)$, together with a
  map $\Pi_{\psi} \rightarrow \CS_{\psi}^{\vee}$, which we will simply denote by
  $\pi \mapsto \langle \cdot, \pi \rangle$, such that the following holds:
	\begin{enumerate}
    \item \label{item:thm_local_Apackets1}
      Let $\pi_{\psi}^{\Gammab}$ be the representation of $\Gammab(F)$
      associated to $\Std_{\GSpinb_5} \circ \varphi_{\psi}$ by the local
      Langlands correspondence, and let $\pi_{\psi}^{\tGamma}$ be its extension
      to $\tGamma(F)$ \emph{(}Whittaker-normalised as explained in Section
      \ref{subsec: Whittaker normalisation}\emph{)}. Then the linear form
      $\sum_{\pi \in \Pi_{\psi}} \langle s_{\psi}, \pi \rangle \tr \pi$ on
      $I(\GSpinb_5(F))$ is stable and its transfer to $\tGamma$ is $\tr
      \pi_{\psi}^{\tGamma}$. 
    \item \label{item:thm_local_Apackets2}
      Consider a semisimple $s \in \Cent(\psi, \GSp_4)$, and denote by $\bar{s}$ 
      its image in $\CS_{\psi}$. The pair $(\psi, s)$ determines an endoscopic
      datum $(\bH, \CH, s, \xi)$ for $\GSpinb_5$, as well as $\psi' : \WD_F
      \times \SL_2 \rightarrow \widehat{\bH}$ such that $\psi = \xi \circ
      \psi'$. Then for any $f \in I(\GSpinb_5(F))$ we have
			\[ \sum_{\pi \in \Pi_{\psi}} \langle \bar{s} s_{\psi}, \pi \rangle \tr
				\pi(f) = \sum_{\pi' \in \Pi_{\psi'}} \langle s_{\psi'}, \pi' \rangle \tr
			\pi'(f'). \]
	\end{enumerate}
\end{theorem}
Note that in the second point $\bH$ is either $\GSpinb_5$ or a quotient of a
product of general linear groups by a split torus, and so $\Pi_{\psi'}$ is
well-defined. In the latter case it is a singleton and $\cS_{\psi'}$ is trivial.

As we recalled above (Theorems \ref{thm: Arthur for chi a square}, \ref{thm:
chan gan} and Remark \ref{rem: LL GSp4 Archi}) this theorem is already known in
the following cases:
\begin{itemize}
  \item if $\muhat \circ \psi$ is a square,
  \item if $F$ is $p$-adic and $\psi|_{\SL_2} = 1$,
  \item if $F$ is Archimedean, $\psi|_{\SL_2}$ and $\psi$ is not discrete.
\end{itemize}

We will prove the case where $F$ is Archimedean, $\psi$ tempered discrete and
$\chi$ not a square later in Proposition \ref{prop: twisted endoscopic character
real discrete tempered}, since we will use a global argument using the
stabilisation of the trace formula.

This section is devoted to the proof of Theorem \ref{thm:local_Arthur_packets}
in the remaining cases, where $\psi|_{\SL_2}$ is not trivial and $\muhat \circ
\psi$ is not a square. It is easy to see that $\Std_{\GSpinb_5} \circ \psi
\simeq (\varphi[2], \chi)$, where $\varphi : \WD_F \rightarrow \GL_2$ is
$\chi$-self-dual of orthogonal type. Then $\varphi$ factors through $W_F$ and
$\det \varphi / (\muhat \circ \psi)$ has order $1$ or $2$. There are two cases
to consider.
\begin{enumerate}
	\item If $\varphi$ is irreducible then $\det \varphi / (\muhat \circ \psi)$
    has order $2$. Let $E/F$ be the corresponding quadratic extension and denote
    $c$ the non-trivial element of $\Gal(E/F)$. We have $\varphi \simeq
    \Ind_{E/F} \mu$ for a character $\mu : E^{\times} \rightarrow \C^{\times}$
    such that $\mu^c \neq \mu$ and $\mu|_{F^{\times}} = \chi$. Then $\Cent(\psi,
    \GSp_4) = Z(\GSp_4)$ and so we simply have to produce $\Pi_{\psi} = \{ \pi
    \}$ such that $\tr \pi$ transfers to the trace of $\pi_{\psi}^{\tGamma}$.
	\item If $\varphi$ is reducible then $\varphi = \eta_1 \oplus \eta_2$ with
    $\eta_1 \eta_2 = \chi$ and $\eta_1 \neq \eta_2$. Then $\Cent(\psi, \GSp_4) =
    \{ \diag( u_1 I_2, u_2 I_2) \}$ and so we are led to define $\Pi_{\psi} = \{
    \Ind_{\bL}^{\GSpinb_5} ((\rec(\eta_1) \circ \det) \otimes \rec(\chi)) \}$
    where $\bL \simeq \GLb_2 \times \GSpinb_1$. Then the second point in Theorem
    \ref{thm:local_Arthur_packets} is automatically satisfied (see
    \cite[\S 6.6]{MR3267112}), and again we have to check that the twisted
    endoscopic character relation holds.
\end{enumerate}

We will prove these two cases separately, distinguishing between the cases where
$F$ is $p$-adic, real, or complex (in which case only the second case
occurs). Before doing so, we recall some material on Whittaker normalisations.

\subsection{Whittaker normalisation for general linear groups}
\label{subsec: Whittaker normalisation}

In this section $F$ denotes a local field of characteristic zero, $\bG = \GLb_n
\times \GLb_1$ over $F$ and $\widetilde{\bG} = \bG \rtimes \theta$. Following
\cite[\S 5]{MWtransfertGLtordu}, \cite{MR2683009}, \cite[\S 8]{AMR} we briefly
recall the Whittaker normalisation of extensions to $\widetilde{\bG}(F)$ of
irreducible representations of $\bG(F)$ fixed by $\theta$. Recall that we have
fixed a $\theta$-stable Whittaker datum $(\bU, \lambda)$ for $\bG$. If $F$ is
Archimedean for simplicity we choose the maximal compact subgroup $K$ to be
$\operatorname{O}_n(F) \times \{ \pm 1 \}$ (resp.\ $\operatorname{U}(n) \times
\operatorname{U}(1)$) if $F$ is real (resp.\ complex), so that $\theta(K) = K$.

First consider the case of essentially tempered representations. Let $\pi$ be an
essentially tempered (in particular, essentially unitary) irreducible
representation of $\bG(F)$. By \cite{MR0348047} there exists a continuous
Whittaker functional $\Omega$ for $\pi$. If $F$ is $p$-adic this is just an
element of the algebraic dual of the space $\pi_K$ of smooth vectors. If $F$ is
Archimedean this is a continuous functional on the space $\pi_{\infty}$ of
smooth vectors for the topology defined by seminorms as in \cite[p.\
183]{MR0348047}. Now if $\pi$ is fixed by $\theta$, define $\tilde{\pi}(\theta)$
as the unique element $A \in \Isom(\pi, \pi^{\theta})$ such that $\Omega \circ A
= \Omega$. This does not depend on the choice of $\Omega$. So we have an
extension $\tilde{\pi}$ of $\pi$ to a representation of $\widetilde{\bG}(F)$,
well-defined using the Whittaker datum $(\bU, \lambda)$.

Next consider representations parabolically induced from a $\theta$-stable
parabolic subgroup. Fix the usual (diagonal) split maximal torus $\bT$ of $\bG$,
as well as the usual (upper triangular) Borel subgroup $\bB = \bT \bU$ of $\bG$.
Both are $\theta$-stable. Let $w_{\bG}$ be the longest element of the Weyl group
$W(\bT, \bG)$. Let $\bP = \bM \bN$ be a standard parabolic subgroup of $\bG$,
with standard Levi subgroup $\bM \supset \bT$. Assume that $\bP$ is
$\theta$-stable, which means that $\bM = (\GLb_{n_1} \times \dots \times
\GLb_{n_r}) \times \GLb_1$ (block diagonal) with $n_i = n_{r+1-i}$ for all $i$.
Let $\sigma$ be an irreducible admissible representation of $\bM(F)$ fixed by
$\theta$, that is $\sigma \simeq ( \sigma_1 \otimes \dots \otimes \sigma_r)
\otimes \chi$ with $(\det \circ \chi) \otimes \sigma_i^{\vee} \simeq
\sigma_{r+1-i}$ for all $i$. Let $\bD_{\bM}$ be the largest split torus which is
a quotient of $\bM$, so that we have a canonical isogeny $\bA_{\bM} \rightarrow
\bD_{\bM}$. In the present case we have a natural identification $\bD_{\bM}
\simeq \GLb_1^r \times \GLb_1$ via the determinants $\GLb_{n_i} \rightarrow
\GLb_1$. For $\nu \in X^*(\bD_{\bM}) \otimes \C$ inducing a character of
$\bM(F)$, consider the parabolically induced (normalised) representation
$\pi_{\nu} := \Ind_{\bP(F)}^{\bG(F)} \sigma \otimes \nu$. We also assume that
$\nu = (\nu_1, \dots, \nu_r, \nu_0)$ is fixed by $\theta$, i.e.\ $\nu_i +
\nu_{r+1-i} = \nu_0$ for all $i$. Let $w_{\bM}$ be the longest element of
$W(\bT, \bM)$ (for $\bB \cap \bM$) and $w = w_{\bG} w_{\bM}$. Let $\bP^- = \bM
\bN^-$ be the parabolic subgroup of $\bG$ opposite to $\bP$ with respect to
$\bM$, and let $\bP' = \bM' \bN' = w \bP^- w^{-1} = w_{\bG} \bP^- w_{\bG}^{-1}$
be the standard parabolic subgroup conjugated to $\bP^-$. Choose a lift
$\tilde{w}$ of $w$ in $N_{\bG(F)}(\bT)$. Let $\lambda^{\tilde{w}}_{\bM} : (\bM
\cap \bU)(F) \rightarrow S^1$ be the generic character defined by
$\lambda^{\tilde{w}}_{\bM}(u) = \lambda(\tilde{w} u \tilde{w}^{-1})$. Assume
that the space $\Hom_{(\bM \cap \bU)(F)}(\sigma, \lambda^{\tilde{w}}_{\bM})$ of
Whittaker functionals for $\sigma$ with respect to $\lambda^{\tilde{w}}_{\bM}$
is non-zero and thus one-dimensional, and fix a basis $\Omega_{\sigma}$ of this
line. In the $p$-adic case, according to a theorem of Rodier (\cite{MR0354942},
\cite{MR581582}, explained in \cite[\S 3.4]{MR2683009}) we then have that
$\Hom_{\bU(F)}(\Ind_{\bP(F)}^{\bG(F)}(\sigma \otimes \nu), \lambda)$ also has
dimension one. A basis $\Omega_{\pi_{\nu}}$ can be made explicit: for $f$ in the
space of $\Ind_{\bP}^{\bG} \sigma \otimes \nu$ whose support is contained in the
big cell $\bP(F) w^{-1} \bU(F)$,
\numequation \label{equ:int_formula_Whittaker_induced}
  \Omega_{\pi_{\nu}}(f) := \int_{\bN'(F)} \Omega_{\sigma}(f(\tilde{w}^{-1} n))
	\lambda(n)^{-1} dn
\end{equation}
is well-defined (the integrand is smooth and compactly supported). For arbitrary
$f$ the same formula holds with $\bN'(F)$ replaced by large enough open compact
subgroup which depends on $f$ but not on $\nu$ (as usual realising the vector
space underlying $\Ind_{\bP(F)}^{\bG(F)} \sigma \otimes \nu$ independently of
$\nu$ by restriction to $K$), so that $\nu \mapsto \Omega_{\pi_{\nu}}(f)$ is
holomorphic.

The Archimedean case is more subtle, since the notion of Whittaker functional
requires a topology on the underlying space of the representation to be
well-behaved (it is not defined directly on $(\fg, K)$-modules). So in this case
one considers the smooth parabolically induced representation $\pi_{\nu} :=
\Ind_{\bP}^{\bG} (\sigma_{\infty} \otimes \nu)$, whose subspace $\pi_{\nu, K}$
of $K$-finite vectors is naturally isomorphic to the $(\fg, K)$-module
algebraically induced from $\sigma_{\bM(F) \cap K}$ (see \cite[\S
III.7]{MR1721403}). Assume that the central character of $\sigma$ is unitary.
Then the integral \eqref{equ:int_formula_Whittaker_induced} is absolutely
convergent for $\nu \in X^*(\bD_{\bM}) \otimes \C$ satisfying
\numequation \label{equ:pos_cond_ind}
	\forall \alpha \in \Phi(\bT, \bN), \ \langle \alpha^{\vee}, \Re \nu \rangle >
	0,
\end{equation}
and extends analytically to $X^*(\bD_{\bM}) \otimes \C$ (\cite[Theorem
3.6.4]{MR2683009}). The proof of Theorem 3.6.7 in \cite{MR2683009} also shows
uniqueness (up to a scalar) of a Whittaker functional for $\Ind_{\bP}^{\bG}
(\sigma_{\infty} \otimes \nu)$ (note that the argument for uniqueness only
involves the Jordan--H\"older factors of a principal series representation, and
so one may replace $\bP$ by another parabolic subgroup of $\bG$ admitting $\bM$
as a Levi factor and such that the opposite of \eqref{equ:pos_cond_ind} is
satisfied, so that any generic subquotient of $\Ind_{\bP}^{\bG} (\sigma \otimes
\nu)$ appears as a quotient).

We can now treat the $p$-adic and Archimedean cases together. Assume that $\nu$
is chosen so that $\End_{\bG(F)}(\pi_{\nu}) = \C$. This is the case if the
central character of $\sigma$ is unitary and $\nu$ satisfies
\eqref{equ:pos_cond_ind} (this follows from the fact that $\pi_{\nu}$ then has a
unique irreducible quotient which occurs with multiplicity one in its
composition series), or if $- \nu$ satisfies \eqref{equ:pos_cond_ind}
($\pi_{\nu}$ then has a unique irreducible subrepresentation). Then one can
define the action of $\theta$ on $\pi_{\nu}$ to be the unique  $A_{\theta} \in
\End(\pi_{\nu})$ such that $A_{\theta} \circ \pi_{\nu}(g) = \pi_{\nu}(\theta(g))
\circ A_{\theta}$ for all $g \in \bG(F)$ and $\Omega_{\pi_{\nu}} \circ A =
\Omega_{\pi_{\nu}}$. This can be made more explicit in the case at hand, see
\cite[\S 5.2]{MWtransfertGLtordu}. The operator $A_{\theta}$ does not depend on
the choice of $\tilde{w}$ made above.

For this definition we followed \cite[\S 8]{AMR}. As explained there, the
resulting canonical extension of $\pi_{\nu}$ coincides with the extension
defined by Arthur in \cite[\S 2.2]{MR3135650}, by \cite[\S
5.2]{MWtransfertGLtordu} and analytic continuation (see \cite[Remarque
8.3]{AMR}).

Finally, consider an arbitrary irreducible smooth representation $\pi$ of
$\bG(F)$ (admissible $(\fg, K)$-module in the Archimedean case).
By the Langlands classification (\cite[Lemmas 3.14 and 4.2]{langlandsrg},
\cite{Silberger_Langlandsquot}, \cite[Chapter IV]{MR1721403}), $\pi$ is the
unique irreducible quotient of $\Ind_{\bP}^{\bG} (\sigma \otimes \nu)$ (resp.\
unique irreducible subrepresentation of $\Ind_{\bP^-}^{\bG} (\sigma \otimes
\nu)$) for $\nu \in X^*(\bD_{\bM}) \otimes \C$ satisfying
\eqref{equ:pos_cond_ind}, with $\sigma$ tempered (in particular, with unitary
central character) and the pair $(\bP, \sigma \otimes \nu)$ is well-defined up
to conjugation. These two realisations of $\pi$ as quotient (resp.\
subrepresentation) of a parabolically induced representation give two canonical
extensions of $\pi$ to $\widetilde{\bG}$, by the above. In fact these two
canonical extensions coincide: consider the composition
\[ \Ind_{\bP}^{\bG} (\sigma \otimes \nu) \rightarrow \pi \rightarrow
	\Ind_{\bP^-}^{\bG} (\sigma \otimes \nu) \]
which is clearly non-zero. From the properties of these induced representations
mentioned above it follows that $\dim \Hom_{\bG(F)}( \Ind_{\bP}^{\bG} (\sigma
\otimes \nu), \Ind_{\bP^-}^{\bG} (\sigma \otimes \nu)) \leq 1$. Therefore the
above composition coincides with the usual intertwining operator
\cite[Th\'eor\`eme IV.1.1]{WaldspurgerPlancherel}, \cite{VoganWallach} (up to a
scalar and a normalising factor to make this intertwining operator holomorphic
at $\nu$). But this operator varies analytically if we vary $\nu$, and
generically it is an isomorphism between irreducible parabolically induced
representations, thus generically it intertwines the two $A_{\theta}$'s, and by
continuity this also holds for the original $\nu$.

\subsection{Proof of Theorem~\ref{thm:local_Arthur_packets}}

We now prove Theorem~\ref{thm:local_Arthur_packets} in the cases described at
the end of Section \ref{subsec: local packets}.

\begin{proof}[Proof in the first case for $F$ $p$-adic]
  The proof is a very special case of the generalisation of \cite[Th\'eor\`eme
  4.7.1]{MWtransfertGLtordu} to \emph{essentially} self-dual representations.
  See also \cite{Moeglin_sur_certains}.
	
  Let $\rho$ be the supercuspidal representation of $\GLb_2(F)$ such that
  $\rec(\rho) = \varphi$. Then $\chi \otimes \rho^{\vee} \simeq \rho$. We will
  give an ad hoc definition of $\Pi_{\psi}$, using special cases of results of
  \cite{MWtransfertGLtordu} to check compatibility with twisted endoscopy for
  $\GLb_4 \times \GLb_1$. In \cite{MWtransfertGLtordu} M\oe{}glin and
  Waldspurger consider self-dual parameters, and we will argue that their
  arguments extend to the case at hand without substantial modification, the
  essential input being compatibility of local Langlands for $\GSpinb_5$ for
  twisted endoscopy (and the same for $\GSpinb_3$ and $\GSpinb_1$, which is
  trivial).

	Let $\Delta$ be the diagonal embedding $\SU(2) \hookrightarrow \SU(2) \times
	\SL_2(\C)$, so that $\psi \circ \Delta$ is the essentially tempered Langlands
	parameter obtained by tensoring $\varphi$ with the $2$-dimensional irreducible
  representation of the factor $\SU(2)$ of $\WD_F$. Then $\Cent( \psi \circ
  \Delta, \GSp_4) = Z(\GSp_4)$, and so $\Pi_{\psi \circ \Delta}$ (as defined by
  Gan--Takeda in \cite{MR2800725}) consists of a single irreducible discrete
  series representation $\pi_{\psi \circ \Delta}$ of $\GSpinb_5(F)$. Let $\bP$
  be the standard parabolic subgroup of $\GSpinb_5$ with Levi subgroup $\bL
  \simeq \GLb_2 \times \GSpinb_1$ (conventions as in Section
  \ref{subsec:Levi_parametrisation}). Then $\Jac_{\bP}( \pi_{\psi \circ \Delta})
  = \rho |\det|^{1/2} \otimes \chi$ where $\Jac$ denotes the normalised Jacquet
  module. We briefly recall the proof. Let $\pi_{\psi \circ \Delta}^{\GLb}$ be
  the (discrete series) representation of $\GLb_4(F)$ corresponding to $\pr_1
  \Std \circ \psi \circ \Delta : \WD_F \rightarrow \GL_4(\C)$. Denoting by
  $\bP^{\GLb}$ the upper block triangular parabolic subgroup of $\GLb_4$ with
  Levi subgroup $\GLb_2 \times \GLb_2$, it is well-known that $\Jac_{\bP^{\GLb}}
  \left( \pi_{\psi \circ \Delta}^{\GLb} \right) = \rho |\det|^{1/2} \otimes \rho
  |\det|^{-1/2}$. Let $\pi_{\psi \circ \Delta}^{\tGamma}$ be the
  Whittaker-normalised (see Section \ref{subsec: Whittaker normalisation} or
  \cite[\S 5.1]{MWtransfertGLtordu}) extension of $\pi_{\psi \circ
  \Delta}^{\GLb} \otimes \chi$ to $\tGamma(F)$. By (iii) in the main theorem of
  \cite{MR3267112} we have that $\tr \pi_{\psi \circ \Delta}^{\tGamma}$ is a
  transfer of $\tr \pi_{\psi \circ \Delta}$. The parabolic subgroup $\bP^{\GLb}
  \times \GLb_1$ of $\Gammab$ is stable under $\theta$, write $\tilde{\bP} =
  (\bP^{\GLb} \times \GLb_1) \rtimes \theta$. By (an obvious generalisation of)
  \cite[Lemme 4.2.1]{MWtransfertGLtordu}, $\tr \Jac_{\tilde{\bP}} ( \pi_{\psi
  \circ \Delta}^{\tGamma} )$ is a transfer of $\tr \Jac_\bP ( \pi_{\psi \circ
  \Delta} )$, and thus $\Jac_\bP ( \pi_{\psi \circ \Delta} ) = \rho |\det|^{1/2}
  \otimes \chi$. By Frobenius reciprocity, $\pi_{\psi \circ \Delta}$ is
  naturally a subrepresentation of $\Ind_{\bP}^{\GSpinb_5} \left( \rho
  |\det|^{1/2} \otimes \chi \right)$. By \cite[Theorem 2.8]{BZ77ENS} this
  parabolic induction has length $\leq 2$ and so the cokernel of
  \[ \pi_{\psi \circ \Delta} \hookrightarrow \Ind_{\bP}^{\GSpinb_5} \left( \rho
	|\det|^{-1/2} \otimes \chi \right) \]
  is an irreducible Langlands quotient which we denote $\pi_{\psi}$. We
  let $\Pi_{\psi} = \{ \pi_{\psi} \}$. Since $\Cent( \psi, \GSp_4(\C) ) =
  \C^{\times}$, we only have to check the twisted endoscopic character relation
  (Theorem \ref{thm:local_Arthur_packets}~\eqref{item:thm_local_Apackets1}).
  Following \cite{MWtransfertGLtordu}, this will be a consequence of comparing
  the short exact sequence
	\numequation \label{eq:resolution_GSpin}
  0 \rightarrow \pi_{\psi \circ \Delta} \rightarrow \Ind_{\bP}^{\GSpinb_5}
    \left( \rho |\det|^{1/2} \otimes \chi \right) \rightarrow \pi_{\psi}
    \rightarrow 0
	\end{equation}
	with a similar one for $\tGamma$.

  We have a short exact sequence of representations of $\Gammab(F) = \GLb_4(F)
  \times \GLb_1(F)$:
	\numequation \label{eq:resolution_GL}
		0 \rightarrow \pi_{\psi \circ \Delta}^{\GLb} \otimes \chi \rightarrow
		\mathcal{E}_1(\pi_{\psi \circ \Delta}^{\GLb}) \otimes \chi \rightarrow
		\pi_{\psi}^{\GLb} \otimes \chi \rightarrow 0
	\end{equation}
	obtained as in \cite[Prop.\ 3.1.2]{MWtransfertGLtordu}, by applying functorial
	constructions to $\pi_{\psi \circ \Delta}^{\GLb}$ to get a resolution of
	$\pi_{\psi \circ \Delta}^{\GLb}$ by sums of standard modules except possibly
	for the last term, which is defined as a cokernel and shown to be irreducible
	with Langlands parameter $(\psi \circ \Delta)^{\sharp} = \psi$ (the general
	definition of $\psi^{\sharp}$ is given in \cite[\S
	3.1.2]{MWtransfertGLtordu}). The definition of the middle term is
	\[ \mathcal{E}_1(\pi_{\psi \circ \Delta}^{\GLb}) := \Ind_{\bP^{\GLb}}^{\GLb_4}
		\left( \Jac_{\bP^{\GLb}}(\pi_{\psi \circ \Delta}) \right) \simeq
		\Ind_{\bP^{\GLb}}^{\GLb_4} \left( \rho|\det|^{1/2} \otimes \rho|\det|^{-1/2}
		\right) \]
  and in the present case M\oe{}glin and Waldspurger's resolution does not
  involve any non-trivial ``$\mathrm{proj}$'', so that the resolution actually
  goes back to \cite{Aubert}, \cite{SchneiderStuhler}. Following M\oe{}glin and
  Waldspurger one can extend $\pi_{\psi \circ \Delta}^{\GLb} \otimes \chi$ from
  $\Gammab(F)$ to $\Gammab^+(F)$ by choosing an action of $\theta$ (see
  \cite[\S\S 1.7-1.9]{MWtransfertGLtordu}), that we denote by $\theta_{MW}$. The
  resolution \eqref{eq:resolution_GL} inherits an action of $\theta$ by
  functoriality (see \cite[\S 3.2]{MWtransfertGLtordu}), and fortunately the
  resulting action on $\pi_{\psi}^{\GLb} \otimes \chi$ happens to coincide with
  $\theta_{MW}$ (see \cite[Lemma 3.2.2]{MWtransfertGLtordu}, in which we have
  $j(\psi) = 1$ and so $\beta(\psi \circ \Delta, \rho, \leq d) = +1$). Another
  way to choose an extension of $\pi_{\psi \circ \Delta}^{\GLb} \otimes \chi$
  (resp.\ $\pi_{\psi}^{\GLb} \otimes \chi$) to $\theta$ is to use Whittaker
  functionals and the Langlands classification as we recalled in Section
  \ref{subsec: Whittaker normalisation}. Denote the resulting actions of
  $\theta$ by $\theta_W$. In general $\theta_W$ and $\theta_{MW}$ differ by a
  sign, but here fortunately $\theta_W = \theta_{MW}$ on both $\pi_{\psi \circ
  \Delta}^{\GLb} \otimes \chi$ and $\pi_{\psi}^{\GLb} \otimes \chi$ (a special
  case of \cite[Prop.\ 5.4.1]{MWtransfertGLtordu}). Thus we have a well-defined
  extension
	\numequation \label{eq:resolution_GL_twisted}
		0 \rightarrow \left( \pi_{\psi \circ \Delta}^{\GLb} \otimes \chi \right)^+
		\rightarrow \left( \mathcal{E}_1(\pi_{\psi \circ \Delta}^{\GLb}) \otimes \chi
		\right)^+ \rightarrow \left( \pi_{\psi}^{\GLb} \otimes \chi \right)^+
		\rightarrow 0
	\end{equation}
  of \eqref{eq:resolution_GL} to $\Gammab^+(F)$. The trace of the left term is
  known to be the transfer of $\tr \pi_{\psi \circ \Delta}$. By compatibility of
  stable transfer with Jacquet modules \cite[Lemme 4.2.1]{MWtransfertGLtordu}
  and parabolic induction (a consequence of the explicit formula for parabolic
  induction (\cite{vanDijk}, \cite{MR789080}, \cite[\S 7.3, Corollaire
	3]{Lemaire})), the trace of the middle term is the transfer of the middle term
	of \eqref{eq:resolution_GSpin}. So we can conclude that $\tr
	\left(\pi_{\psi}^{\GLb} \otimes \chi \right)^+$ is the transfer of $\tr
	\pi_{\psi}$.
\end{proof}

\begin{proof}[Proof in the second case for $p$-adic $F$]
	This is similar to the previous case but now $\varphi : W_F \rightarrow
  \GL_2(\C)$ is reducible and so it defines a principal series representation of
  $\GLb_2(F_v)$. Write $\varphi \simeq \rec(\eta_1) \oplus \rec(\eta_2)$, so
  that $\chi = \eta_1 \eta_2$. As explained above we can assume that $\eta_1
  \neq \eta_2$. Define $\pi_{\psi} = \Ind_{\bP}^{\GSpinb_5} \left( (\eta_1 \circ
  \det) \otimes \chi \right)$ where the standard parabolic subgroup $\bP$ has
  Levi $\bL \simeq \GLb_2 \times \GSpinb_2$ and $\Pi_{\psi} = \{ \pi_{\psi} \}$.
  The representation $\pi_{\psi}$ is certainly irreducible (see \cite[\S
  4.2]{MoeglinClayMathProc}), but since this is not necessary to prove the
  Theorem we simply take the definition $\Pi_{\psi} = \{ \pi_{\psi} \}$ to mean
  that $\Pi_{\psi}$ is the multi-set of constituents of $\pi_{\psi}$.

  Consider the parabolic induction for $\GLb_4 \times \GLb_1$
  \numequation \label{eqn: non_std_ind}
    \pi_{\psi}^{\Gammab} := \Ind_{\bP^{\GLb}}^{\GLb_4} \left( ( \eta_1 \circ
    \det ) \otimes (\eta_2 \circ \det) \right) \otimes \chi
  \end{equation}
  where $\bP^{\GLb}$ is the standard parabolic subgroup of $\GLb_4$ with Levi
  $\GLb_2 \times \GLb_2$. The twisted representation $\pi_{\psi}^{\tGamma}$
  of $\tGamma(F)$ obtained from \eqref{eqn: non_std_ind} using the canonical
  action of $\theta$ (defined as in \cite[\S 1.3]{MWtransfertGLtordu}) is such
  that its trace is the transfer of the trace of $\pi_{\psi}$, by compatibility
  of parabolic induction with transfer. This is almost the twisted endoscopic
  character relation, but again we need to be careful with the definition of
  Whittaker normalisation. The Whittaker-normalised action of $\theta$ on
  $\pi_{\psi}^{\Gammab}$ is obtained by realising it as the Langlands quotient
  of
  \numequation \label{eqn: Borel_ind}
    \Ind_{\bB^{\GLb}}^{\GLb_4} \left( \eta_1|\cdot|^{1/2} \otimes
      \eta_2|\cdot|^{1/2} \otimes \eta_1|\cdot|^{-1/2}\otimes
      \eta_2|\cdot|^{-1/2} \right) \otimes \chi
  \end{equation}
  where $\bB^{\GLb}$ is the standard Borel subgroup of $\GLb_4$, which coincides
  with the canonical action of $\theta$ on this parabolic induction by (the
  obvious generalisation of) \cite[Lemme 5.2.1]{MWtransfertGLtordu}.

  Let us sketch the proof of the fact that these two actions of $\theta$ on
  $\pi_{\psi}^{\Gammab}$ coincide. It will be convenient to denote $\sigma_1
  \times \dots \times \sigma_r$ for the parabolic induction (using the standard
  parabolic) of an admissible representation $\sigma_1 \otimes \dots \otimes
  \sigma_r$ of $\GLb_{n_1}(F) \times \dots \times \GLb_{n_r}(F)$ to $\GLb_{n_1 +
  \dots + n_r}(F)$. Recall that for any $s \in C$ the parabolic induction
  $\eta_2|\cdot|^{1/2+s} \times \eta_1 |\cdot|^{-1/2-s}$ is irreducible by
  \cite[Theorem 3]{BZIndGLpadic}, since the assumption that $\chi = \eta_1
  \eta_2$ is not a square implies that $\eta_1|_{\CO_F^{\times}} \neq
  \eta_2|_{\CO_F^{\times}}$. The intertwining operator
	\[ I_s : \eta_2 |\cdot|^{1/2+s} \times \eta_1 |\cdot|^{-1/2-s} \longrightarrow
	\eta_1 |\cdot|^{-1/2-s} \times \eta_2 |\cdot|^{1/2+s} \]
	defined by the usual integral formula for $\Re (s) \gg 0$, is rational in
  $q^{-s}$ (where $q$ is the cardinality of the residue field of $F$)
  by \cite[Th\'eor\`eme IV.1.1]{WaldspurgerPlancherel}, and so there is a
  polynomial $r(s)$ in $q^{-s}$ such that $r(s) I_s$ is well-defined and
  non-zero for any $s$, and therefore an isomorphism. It induces an isomorphism
  $I_{s,\nor}$:
	\[ \eta_1 |\cdot|^{1/2} \times \eta_2 |\cdot|^{1/2+s} \times \eta_1
		|\cdot|^{-1/2-s} \times \eta_2 |\cdot|^{-1/2} \rightarrow \eta_1
		|\cdot|^{1/2} \times \eta_1 |\cdot|^{-1/2-s} \times \eta_2 |\cdot|^{1/2+s}
	\times \eta_2 |\cdot|^{-1/2}. \]
	Denote $\pi_{1,s}$ (resp.\ $\pi_{2,s}$) the LHS (resp.\ RHS). Since $\eta_2
	|\cdot|^{-1/2} = \chi / \left( \eta_1 |\cdot|^{1/2} \right)$ and $\eta_1
	|\cdot|^{-1/2-s} = \chi / \left( \eta_2 |\cdot|^{1/2+s} \right)$, there is a
	canonical extension of $\pi_{1,s} \otimes \chi$ to $\Gammab^+(F)$ (see
	\cite[\S 1.3]{MWtransfertGLtordu}). Denote by $\theta_1$ this canonical
	action of $\theta$ on the space of $\pi_{1,s} \otimes \chi$ (one can easily
	check that it does not depend on $s$), so that for $s=0$ we recover the
  Whittaker normalisation on \eqref{eqn: Borel_ind}. The irreducible
  representation
	\[ \left( \left( \eta_1 |\cdot|^{1/2} \times \eta_1 |\cdot|^{-1/2-s} \right)
	\otimes \left( \eta_2 |\cdot|^{1/2+s} \times \eta_2 |\cdot|^{-1/2} \right)
	\right) \otimes \chi \]
  of the $\theta$-stable parabolic subgroup $\bP \times \GLb_1$ of $\Gammab$ is
  also fixed by $\theta$, and so $\pi_{2,s} \otimes \chi$ also admits a
  canonical extension to $\tGamma(F)$. Denote $\theta_2$ this canonical action
  of $\theta$ on the space of $\pi_{2,s} \otimes \chi$, which for $s=0$ recovers
  the canonical action on the quotient \eqref{eqn: non_std_ind}. An easy
  computation that we skip shows that for $\Re(s) \gg 0$ we have $I_{s, \nor}
  \circ \theta_1 = \theta_2 \circ I_{s, \nor}$, and the case of an arbitrary $s
  \in \C$ follows by analytic continuation.
\end{proof}

\begin{proof}[Proof in the first case for $F=\R$]
  This is similar to the first case for $F$ a $p$-adic field except we now
  follow arguments of \cite{AMR}. For $a \in \frac{1}{2} \Z_{\geq 0}$ let $I_a$
  be the tempered Langlands parameter $W_{\R} \rightarrow \GL_2(\C)$ obtained by
  inducing the character  $z \mapsto (z/\bar{z})^a := (z/|z|)^{2a}$ of
  $\C^{\times}$. Up to twisting we can assume that $\varphi = I_a$ with $a > 0$
  integral, with $\chi$ equal to the sign character $\sign$ of $W_{\R}$.
  Let $\pi_{\psi}^{\GLb_4}$ be the irreducible unitary representation of
  $\GLb_4(\R)$ associated to $\varphi_{\psi}$. Let $\chi : \GLb_1(\R)
  \rightarrow \{ \pm 1 \}$ be the sign character, so that $\chi \otimes
  (\pi_{\psi}^{\GLb_4})^{\vee} \simeq \pi_{\psi}^{\GLb_4}$. As in the $p$-adic
  case we have the Whittaker-normalised extension $\pi_{\psi}^{\tGamma}$ of
  $\pi_{\psi}^{\Gammab} := \pi_{\psi}^{\GLb_4} \otimes \chi$.

  We have a (short) resolution from \cite{Johnson} (see \cite[\S 6.2]{AMR} where
  this resolution is made completely explicit for $\GLb_{2n}$ and parameters
  $I_w[n]$ for $w \in \frac{1}{2} \Z_{>0}$)
  \[ 0 \rightarrow \pi_{\psi}^{\GLb_4} \rightarrow \pi_{I_a | \cdot
    |^{-1/2}}^{\GLb_2} \times \pi_{I_a | \cdot |^{1/2}}^{\GLb_2} \rightarrow
    \pi_{I_{a+1/2}}^{\GLb_2} \times \pi_{I_{a-1/2}}^{\GLb_2} \rightarrow 0 \]
  where $|\cdot|$ is the norm character of $W_{\R}$ (i.e.\ the square of the
  usual absolute value on $\C^{\times}$ and $|j| = 1$) and we denoted parabolic
  induction for standard parabolic subgroups of $\GLb$ as in the $p$-adic case.
  In \cite[Lemme 9.9]{AMR} only the first case occurs, so comparing
  normalisations (Whittaker and imposed by induction in Johnson's construction
  of the resolution) is particularly simple: we obtain the analogue of
  \cite[Th\'eor\`eme 9.7]{AMR} with $A_s = A_s^+$.
\end{proof}

\begin{proof}[Proof in the second case for $F=\R$ or $\C$]
  Up to twisting we can assume that $\varphi \simeq 1 \oplus \chi$ with $\chi =
  \sign$ in the real case and $\chi(z) = (z / \bar{z})^a |z|^{it}$ with $a \in
  \frac{1}{2}\Z \smallsetminus \Z$ and $t \in \R$ in the complex case. The proof
  is identical to the $p$-adic case and we do not repeat the argument. Note that
  the complex case is the analogue of \cite[Prop.\ 6.5]{MoeglinRenard}.
\end{proof}

\section{Stabilisation of the twisted trace formula}\label{sec: STTF}

We now state the stabilisation of the twisted trace formula proved by M\oe{}glin
and Waldspurger in \cite{SFTT1}, \cite{SFTT2} following the case of ordinary
(i.e.\ non-twisted) endoscopy proved by Arthur in \cite{ArthurSTF1},
\cite{ArthurSTF2}, \cite{ArthurSTF3} (also following \cite{Langlands_debuts},
\cite{Kottwitz_ell_sing}, \cite{Labesse_asterisque}, and of course
\cite{TTF}). We recall some of the definitions needed to state the
stabilisation, and mention some simplifications occurring in the cases at hand.

\subsection{The discrete part of the spectral side}
\label{subsec: discrete part of spectral side of STTF}

Consider a connected reductive group $\bG$ over a number field~ $F$ and an
automorphism $\theta$ of $\bG$ of finite order. Let $\tbG = \bG \rtimes \theta$.
Let $\bA_0$ be a maximal split torus in $\bG$. We will only consider Levi
subgroups of $\bG$ which contain $\bA_0$. Let $K = \prod_v K_v$ be a good
maximal compact subgroup of $\bG(\A_F)$ with respect to $\bA_0$ as in \cite[\S
3.1]{TTF}. Choose a minimal parabolic subgroup $\bP_0$ of $\bG$ containing
$\bA_0$.

Following \cite[\S X.5]{SFTT2}, let us recall the terms occurring in the
discrete part of the spectral side of the twisted trace formula. To work with
\emph{discrete} automorphic spectra it is necessary to fix central characters
(at least on a certain subgroup of the centre), and we follow \cite[\S
X.5.1]{SFTT2}. We now elaborate on the notation for the discrete automorphic
spectrum introduced in Section~\ref{subsubsec: discrete spectrum intro}. Recall
that $\mathfrak{A}_\bG$ denotes the vector group $\mathbf{A}_\bG(\R)^0$ where
$\mathbf{A}_\bG$ is the biggest central split torus in $\Res_{F/\Q}(\bG)$. Then
$\bG(\A_F) = \bG(\A_F)^1 \times \mathfrak{A}_\bG$, where
\[ \bG(\A_F)^1 = \left\{ g \in \bG(\A_F)\,\middle|\, \forall \beta \in
X^*(\bG)^{\Gal_F},\,|\beta(g)|=1 \right\}, \]
so that $\bG(F) \backslash \bG(\A_F)^1$ has finite measure. Let
$\mathfrak{A}_{\widetilde{\bG}} = \mathfrak{A}_\bG^{\theta}$. Then
$\mathfrak{A}_\bG = (1 - \theta)(\mathfrak{A}_\bG) \times
\mathfrak{A}_{\widetilde{\bG}}$. 

In the general definition of twisted endoscopy one considers a
character~$\omega$ of~$\bG(\A_F)$; in all cases considered in this paper we have
$\omega = 1$. M\oe{}glin and Waldspurger consider a character $\chi_\bG$ of
$\mathfrak{A}_\bG$ which is trivial on $\mathfrak{A}_{\widetilde{\bG}}$ and
satisfies $\theta(\chi_\bG) = \chi_\bG \omega|_{\mathfrak{A}_\bG}$; since we
will always have~ $\omega = 1$ in this paper, we will have $\chi_\bG
= 1$.

Let $\bL$ be a Levi subgroup of $\bG$. Up to conjugating by $\bG(F)$ we can
assume that $\bL$ is the standard Levi subgroup of a standard parabolic subgroup
$\bP$ of $\bG$. There is a canonical splitting $\mathfrak{A}_\bL =
\mathfrak{A}_\bG \times \mathfrak{A}_\bL^\bG$ (with $\mathfrak{A}_\bL^\bG$
included in the derived subgroup of $\bG(F \otimes_{\Q} \R)$), and we write
$\chi_{\bG,\bL}$ for the extension of $\chi_\bG$ to $\mathfrak{A}_\bL$ such that
$\chi_{\bG,\bL}|_{\mathfrak{A}_\bL^\bG} = 1$. As remarked above in all cases
considered in this paper we simply have $\chi_{\bG,\bL} = 1$. The space of
square integrable automorphic forms $\CA^2(\bL(F) \backslash \bL(\A_F),
\chi_{\bG,\bL})$ decomposes discretely, i.e.\ it is canonically the direct sum,
over the countable set $\Pi_{\disc}(\bL, \chi_{\bG, \bL})$ of discrete
automorphic representations $\pi_\bL$ for $\bL$ such that
$\pi_\bL|_{\mathfrak{A}_\bL} = \chi_{\bG,\bL}$, of isotypical
components
\[ \CA^2( \bL(F) \backslash \bL(\A_F) , \chi_{\bG,\bL})_{\pi_\bL} \]
which have finite length. Denote by $\bU_\bP$ the unipotent radical of $\bP$.
Recall \cite[\S I.2.17]{MWDecompSpec} the space $\CA^2 \left( \bU_\bP(\A_F)
\bL(F) \backslash \bG(\A_F), \chi_{\bG,\bL} \right)$ of smooth $K$-finite
functions $\phi$ on $\bU_\bP(\A_F) \bL(F) \backslash \bG(\A_F)$ such that for
any $k \in K$,
\[ x \mapsto \delta_\bP(x)^{-1/2}(x) \phi(x k) \]
is an element of $\CA^2(\bL(F) \backslash \bL(\A_F), \chi_{\bG,\bL})$. In other
words,
\[ \CA^2 \left( \bU_\bP(\A_F) \bL(F) \backslash \bG(\A_F), \chi_{\bG,\bL}
    \right) = \Ind_{\bP(\A_F)}^{\bG(\A_F)} \left( \CA^2(\bL(F) \backslash
  \bL(\A_F), \chi_{\bG,\bL}) \right)^{K\mathrm{-fin}}. \]
This space is endowed with the usual left action of $\CH(\bG)$, which we will
denote by~$\rho_\bP^\bG$. If $\pi_\bL$ is an irreducible admissible representation
of $\bL(\A_F)$ such that $\omega_{\pi_\bL} |_{\mathfrak{A}_\bL} =
\chi_{\bG,\bL}$, denote by
\[ \CA^2(\bU_\bP(\A_F) \bL(F) \backslash \bG(\A_F), \chi_{\bG,\bL})_{\pi_\bL} \]
the sub-$\CH(\bG)$-module of $\CA^2(\bU_\bP(\A_F) \bL(F) \backslash \bG(\A_F),
\chi_{\bG,\bL})$ consisting of functions $\phi$ such that for any $k \in K$,
\[ \left( x \mapsto \delta_\bP(x)^{-1/2}(x) \phi(x k) \right) \in \CA^2(\bL(F)
\backslash \bL(\A_F), \chi_{\bG,\bL})_{\pi_\bL}. \]
Let $W(\bL, \widetilde{\bG}) = \Norm_{\widetilde{\bG}(F)}(\bL) /
\bL(F)$, where the action of~$\widetilde{\bG}(F)$ on~$\bG$ is the
adjoint action coming from the definition of a twisted space~\cite[\S I.1.1]{SFTT1}.
For $\tilde{w} \in W(\bL, \widetilde{\bG})$ and $f(\tilde{x}) d \tilde{x} \in
\CH(\widetilde{\bG})$, we have a map \cite[bottom of p.\ 1204]{SFTT2}
\addtocounter{subsubsection}{1}
\begin{align} \begin{split} \label{equ: automorphic twisted action}
  \rho^{\widetilde{\bG}}_{\bP, \tilde{w}}(f) : \CA^2(\bU_\bP(\A_F) \bL(F)
    \backslash \bG(\A_F), \chi_{\bG,\bL})  & \longrightarrow
    \CA^2(\bU_{\tilde{w}(\bP)}(\A_F) \bL(F) \backslash \bG(\A_F),
    \chi_{\bG,\bL}), \\
  \phi & \longmapsto \left( g \mapsto \int_{\widetilde{\bG}(\A_F)}
    \phi(\tilde{w}^{-1} g \tilde{x} ) f(\tilde{x}) d \tilde{x} \right)
\end{split} \end{align}
and for $f_1, f_3 \in \CH(\bG)$ and $f_2 \in \CH(\widetilde{\bG})$ we have
\[ \rho^{\widetilde{\bG}}_{ \bP,\tilde{w}}(f_1 \ast f_2 \ast f_3) =
  \rho^\bG_{\tilde{w}(\bP)}(f_1) \circ \rho^{\widetilde{\bG}}_{\bP,
  \tilde{w}}(f_2) \circ \rho^\bG_\bP(f_3).\]

If $\pi_\bL$ is an irreducible admissible representation of $\bL(\A_F)$ such
that $\omega_{\pi_\bL} |_{\mathfrak{A}_\bL} = \chi_{\bG,\bL}$, then for any $f
\in \CH(\widetilde{\bG})$, $\rho^{\widetilde{\bG}}_{\bP, \tilde{w}}(f)$
restricts to
\[ \CA^2(\bU_\bP(\A_F) \bL(F) \backslash \bG(\A_F), \chi_{\bG,\bL})_{\pi_\bL}
  \longrightarrow \CA^2(\bU_{\tilde{w}(\bP)}(\A_F) \bL(F) \backslash \bG(\A_F),
\chi_{\bG,\bL})_{\tilde{w}(\pi_\bL)} \]
where $\tilde{w}(\pi_\bL) = \pi_\bL \circ \Ad(\tilde{w}^{-1})$.

By meromorphic continuation of the usual integral formula, there is an
intertwining operator
\[ M_{\bP | \tilde{w}(\bP)}(0) : \CA^2 \left( \bU_{\tilde{w}(\bP)}(\A_F) \bL(F)
    \backslash \bG(\A_F), \chi_{\bG,\bL} \right) \rightarrow \CA^2 \left(
    \bU_\bP(\A_F) \bL(F) \backslash \bG(\A_F), \chi_{\bG,\bL} \right). \]
Since $\chi_{\bG,\bL}$ is unitary, $M_{\bP | \tilde{w}(\bP)}$ is well-defined
(i.e.\ holomorphic) at~ $0$, and is in fact unitary. Moreover for any
irreducible admissible representation $\pi_\bL$ of $\bL(\A_F)$, $M_{\bP |
\tilde{w}(\bP)}(0)$ restricts to
\[ \CA^2(\bU_{\tilde{w}(\bP)}(\A_F) \bL(F) \backslash \bG(\A_F),
\chi_{\bG,\bL})_{\pi_\bL} \longrightarrow \CA^2(\bU_\bP(\A_F) \bL(F) \backslash
\bG(\A_F), \chi_{\bG,\bL})_{\pi_\bL}. \]
Therefore for $f \in \CH(\widetilde{\bG})$ the composition $M_{\bP |
\tilde{w}(\bP)}(0) \circ \rho^{\widetilde{\bG}}_{\bP, \tilde{w}}(f)$ maps
\[ \CA^2(\bU_\bP(\A_F) \bL(F) \backslash \bG(\A_F), \chi_{\bG,\bL}) \]
to itself and restricts to
\[ \CA^2(\bU_\bP(\A_F) \bL(F) \backslash \bG(\A_F), \chi_{\bG,\bL})_{\pi_\bL}
\longrightarrow \CA^2(\bU_\bP(\A_F) \bL(F) \backslash \bG(\A_F),
\chi_{\bG,\bL})_{\tilde{w}(\pi_\bL)}. \]

We can finally recall the contribution of $\bL$ to the discrete part of the
spectral side of the twisted trace formula for $\widetilde{\bG}$. For $f \in
\CH(\widetilde{\bG})$, let
\[I_{\disc}^{\widetilde{\bG}, \bL}(f) = |W(\bL,\bG)|^{-1} \sum_{\tilde{w} \in
  W(\bL, \widetilde{\bG})_{\reg}} |\det \left( \tilde{w}-1 \,|\,
\mathfrak{A}_\bL^\bG \right)|^{-1} \tr \left( M_{\bP| \tilde{w}(\bP)}(0) \circ
\rho^{\widetilde{\bG}}_{\bP, \tilde{w}}(f) \right) \]
where $W(\bL, \widetilde{\bG})_{\reg}$ is the set of $\tilde{w} \in W(\bL,
\widetilde{\bG})$ such that $(\mathfrak{a}_\bL^\bG)^{\tilde{w}} = 0$. As the
notation suggests, $I_{\disc}^{\widetilde{\bG}, \bL}(f)$ only depends on $f$ and
the $\bG(F)$-conjugacy class of $\bL$. In fact it depends on $f$ only via its
image in $I(\widetilde{\bG})$. The fact that the trace of $M_{\bP|
\tilde{w}(\bP)}(0) \circ \rho^{\widetilde{\bG}}_{\bP, \tilde{w}}(f)$ on
$\CA^2(\bU_\bP(\A_F) \bL(F) \backslash \bG(\A_F), \chi_{\bG,\bL})$ is
well-defined and equals the absolutely convergent sum
\[ \sum_{\substack{\pi_\bL \in \Pi_{\disc}(\bL, \chi_{\bG,\bL}) \\
    \tilde{w}(\pi_\bL) \simeq \pi_\bL}} \tr \left( M_{\bP | \tilde{w}(\bP)}(0)
    \circ \rho^{\widetilde{\bG}}_{\bP, \tilde{w}}(f) \,\middle|\,
    \CA^2(\bU_\bP(\A_F) \bL(F) \backslash \bG(\A_F), \chi_{\bG,\bL})_{\pi_\bL}
\right) \]
is a consequence of work of Finis, Lapid and M\"{u}ller, as explained in
\cite[\S 14.3]{TTF} and \cite[\S X.5.2 and X.5.3]{SFTT2}. 

The most interesting case is of course for $\bL=\bG$, since
$I_{\disc}^{\widetilde{\bG}, \bG}(f)$ is simply the trace of $f$ on the discrete
automorphic spectrum for $\bG$ and $\chi_\bG$. We will recall below the
refinement of discrete terms by infinitesimal character and Hecke eigenvalues
following Arthur and M\oe{}glin--Waldspurger, that allows one to forget about
convergence issues and work with finite sums. But first we make explicit the
condition $\tilde{w}(\pi_\bL) \simeq \pi_\bL$ in the cases at hand.

\begin{enumerate}
  \item For $\bG = \GLb_N \times \GLb_1$ and a standard (i.e.\ block diagonal)
    Levi $\bL \simeq \left( \prod_{k \geq 1} (\GLb_k)^{n_k} \right) \times
    \GLb_1$ (where $n_k = 0$ for almost all $k$ and $\sum_{k \geq 1} k n_k =
    N$), there always exists an element of $\widetilde{\bG}(F)$ normalising
    $\bL$ (for example $\theta_0 : g \mapsto {}^t g^{-1}$), and $W(\bL, \bG)
    \simeq \prod_{k \geq 1} S_{n_k}$. For $\tilde{w} = (\sigma_k)_{k \geq 1}
    \theta_0 \in W(\bL, \widetilde{\bG})$, $\tilde{w}$ is regular if and only if
    for every $k \geq 1$, the decomposition of $\sigma_k$ in cycles only
    involves cycles of odd length. For such a regular $\tilde{w}$ and if $\pi =
    \left( \bigotimes_{k \geq 1} \left( \pi_{k, 1} \otimes \dots \otimes \pi_{k,
    n_k} \right) \right) \otimes \chi$ is an irreducible admissible
    representation of $\bL(\A_F)$, then $\tilde{w}(\pi) \simeq \pi$ if and only
    if each $\pi_{k,i}$ satisfies $\pi_{k,i}^{\vee} \otimes \chi \simeq
    \pi_{k,i}$ \emph{and} for every $k \geq 1$, the isomorphism class of
    $(\pi_{k,i})_{1 \leq i \leq n_k}$ is fixed by $\sigma_k$.

  \item In the non-twisted cases $\bG = \GSpinb_{2n+1}$ or
    $\GSpinb_{2n}^{\alpha}$, recall that in Section
    \ref{subsec:Levi_parametrisation} we chose $\bL \simeq \prod_{k \geq 1}
    (\GLb_k)^{n_k} \times \bG_m$ where $m + \sum_{k \geq 1} kn_k = n$ and
    $\bG_m$ is a $\GSpinb$ group of the same type as $\bG$ of absolute rank $m$.
    There is a natural embedding $W(\bL,\bG) \into \prod_{k \geq 1} \left( \{
    \pm 1 \}^{n_k} \rtimes S_{n_k} \right)$ which is surjective unless $\bG =
    \GSpinb_{2n}^{\alpha}$, $m=0$, and there exists an odd $k \geq 1$ such that
    $n_k > 0$, in which case it is of index two.

    An element $w = \left( (\epsilon_{k,i})_{1 \leq i \leq n_k} \rtimes \sigma_k
    \right)_{k \geq 1}$ is regular if and only if for every $k \geq 1$ and every
    cycle $( i_1 \dots i_r )$ appearing in the decomposition of $\sigma_k$,
    $\prod_{j=1}^r \epsilon_{k,i_j} = -1$. For such $w \in W(\bL,\bG)_{\reg}$
    and $\pi_\bL \simeq \bigotimes_{k \geq 1} \left( \pi_{k, 1} \otimes \dots
    \otimes \pi_{k, n_k} \right) \otimes \pi_{\bG_m}$ an irreducible admissible
    representation of $\bL(\A_F)$, we have $w(\pi_\bL) \simeq \pi_\bL$ if and
    only
		\begin{enumerate}
			\item for every $k \geq 1$ and $1 \leq i \leq n_k$, $\pi_{k,i}^{\vee}
				\otimes \chi \simeq \pi_{k,i}$ where $\chi : \A_F^{\times} \rightarrow
				\C^{\times}$ is $\pi_{\bG_m} \circ \mu$, and
			\item for every $k \geq 1$ the isomorphism class of $(\pi_{k,i})_{1 \leq i
				\leq n_k}$ is fixed by $\sigma_k$.
		\end{enumerate}
\end{enumerate}

We now recall from~\cite[p.\ 1212]{SFTT2}  the refinement of the discrete part
of the spectral side of the twisted trace formula by infinitesimal characters
(using Arthur's theory of multipliers) and families of Satake parameters.

\begin{defn} \label{defn: IC FS}
  \leavevmode
	\begin{enumerate}
		\item Let~$IC(\bG)$ be the set of semisimple conjugacy classes in the Lie
			algebra of the dual group (over~$\C$) of~ $\Res_{F/\Q}(\bG)$. This is the
			set where infinitesimal characters for irreducible representations of
      $\bG(F \otimes_{\Q} \R)$ live. In the twisted case let
      $IC(\widetilde{\bG}) = IC(\bG)^{\hat{\theta}}$. For $\pi_{\infty}$ an
      irreducible admissible representation of $\bG(F \otimes_{\Q} \R)$, denote
      by $\nu(\pi_{\infty}) \in IC(\bG)$ its infinitesimal character.

    \item Let $S$ be a large enough (i.e.\ containing $V_{\mathrm{ram}}$ as in
      \cite[\S VI.1.1]{SFTT2}) finite set of places of~$F$. Let $FS^S(\bG) =
      \prod_{v \not\in S} \left( \widehat{\bG} \rtimes \Frob_v \right)^{\semis}
      / \widehat{\bG}$, and in the twisted case let $FS^S(\widetilde{\bG}) =
      \left( FS^S(\bG) \right)^{\hat{\theta}}$. Write also $FS(\bG) =
      \varinjlim_S FS^S(\bG)$ and in the twisted case $FS(\widetilde{\bG}) =
      \varinjlim_S FS^S(\widetilde{\bG})$. If $\pi = \otimes'_v \pi_v$ is an
      irreducible admissible representation of $\bG(\A_F)$, we will write
      $c(\pi)$ for the associated element of $FS(\bG)$ via the Satake
      isomorphisms.

    \item For $\nu \in IC(\widetilde{\bG})$, $S$ as above, $c^S = (c_v)_{v
      \not\in S} \in FS^S(\widetilde{\bG})$, and $\bL$ a Levi subgroup of $\bG$,
      let $\Pi_{\disc}(\bL, \chi_{\bG,\bL})_{\nu, c^S}$ be the set of $\pi_{\bL}
      \in \Pi_{\disc}(\bL, \chi_{\bG,\bL})$ such that the infinitesimal
      character of $\pi_{\bL, \infty}$ maps to $\nu$ via $\Lie \left(
        \widehat{\Res_{F/\Q}(\bL)} \right) \rightarrow \Lie \left(
      \widehat{\Res_{F/\Q}(\bG)} \right)$, and for every $v \not\in S$,
      $\pi_{\bL,v}$ is unramified for $K_v$ and its Satake parameter maps to~
      $c_v$ via ${}^L \bL \rightarrow {}^L \bG$. For $f \in \bigotimes_{v \in S}
      \CH(\widetilde{\bG}(F_v))$, let

      \[	I_{\disc, \nu, c^S}^{\widetilde{\bG}, \bL}(f) = |W(\bL, \bG)|^{-1}
        \sum_{\tilde{w} \in W(\bL, \widetilde{\bG})_{\reg}} \det \left(
          \tilde{w}-1 \,|\, \mathfrak{A}_\bL^\bG \right)|^{-1}
        \sum_{\substack{\pi_{\bL} \in \Pi_{\disc}(\bL, \chi_{\bG,\bL})_{\nu,
            c^S}\\
        \tilde{w}(\pi_{\bL}) \simeq \pi_{\bL}}} \tr_{\pi_{\bL}}(f),\]
      where we write
      \[\tr_{\pi_\bL}(f)=\left( M_{\bP | \tilde{w}(\bP)}(0) \circ
        \rho^{\widetilde{\bG}}_{\bP, \tilde{w}}(f) \,\middle|\,
        \CA^2(\bU_\bP(\A_F) \bL(F) \backslash \bG(\A_F),
        \chi_{\bG,\bL})_{\pi_\bL} \right).\]

			Finally let
			\numequation \label{eq: defn Idisc refined}
				I_{\disc, \nu, c^S}^{\widetilde{\bG}}(f) = \sum_\bL I_{\disc, \nu,
				c^S}^{\widetilde{\bG}, \bL}(f)
			\end{equation}
      where the sum is over $\bG(F)$-conjugacy classes of Levi subgroups of
      $\bG$.
	\end{enumerate}
\end{defn}
Seeing this as a sum over triples $(\bL, \tilde{w}, \pi_\bL)$, all but finitely
many terms vanish. Indeed, if we fix $\nu$, $S$, $c^S$ and an idempotent $e$ of
$\bigotimes_{\substack{v \in S \\ v \nmid \infty}} \CH(\bG(F_v))$, then there is
a finite set $\Upsilon(\nu, S, c^S, e)$ of triples $(\bL, \tilde{w}, \pi_\bL)$
such that for any $f \in \bigotimes_{v \in S} \CH(\widetilde{\bG}(F_v))$ for
which $e \ast f = f \ast e = f$, the terms corresponding to $(\bL, \tilde{w},
\pi_\bL) \not\in \Upsilon(\nu, S, c^S, e)$ in the double sum defining $I_{\disc,
\nu, c^S}^{\widetilde{\bG}}(f)$ all vanish. 

\begin{rem} \label{rema: FS}
  \begin{enumerate}
    \item By \cite{MR623137} and \cite{MR1026752}, taking the image in
      $FS(\GLb_N)$ is injective on formal sums of elements of
      $\Pi_{\disc}(\GLb_{n_i}, \chi)$ (note that it is essential that all of the
      summands are $\chi$ self-dual for the same character~$\chi$). For this
      reason we will often identify such formal sums and their image.
    \item In \cite{SFTT2} M\oe{}glin--Waldspurger multiply \eqref{eq: defn Idisc
      refined} by $j(\widetilde{\bG})^{-1} := |\det( 1 - \theta | \CA_{\bG} /
      \CA_{\widetilde{\bG}})|^{-1}$, but this factor is also present in
      $\iota(\widetilde{\bG}, \bH)$ with their definition.
  \end{enumerate}
\end{rem}

\begin{defn}
  \begin{enumerate}
    \item We will say that $c^S \in FS(\widetilde{\bG})$ \emph{occurs} in
      $I_{\disc}^{\widetilde{\bG}, \bL}$ if there exists $\nu \in
      IC(\widetilde{\bG})$ and $f \in \CH(\widetilde{\bG})$ such that up to
      enlarging $S$ we have $I_{\disc, \nu, c^S}^{\widetilde{\bG}, \bL}(f) \neq
      0$.

    \item Let $\bD$ be an induced central torus in $\bG$, so that there is a
      dual morphism ${}^L \bG \rightarrow {}^L \bD$. For $c^S \in
      FS(\widetilde{\bG})$ occurring in $I_{\disc}^{\widetilde{\bG}, \bL}$ we
      define the \emph{central character} of $c^S$ to be the (unique by weak
      approximation for $\bD$ \cite[Proposition 7.3]{PlatonovRapinchuk})
      character $\omega_c : \bA_{\bG}(\A_F) / \bA_{\bG}(F) \rightarrow
      \C^{\times}$ such that for almost all places $v$ of $F$, the Langlands
      parameter of $(\omega_c)_v$ equals the image of $c_v$ in ${}^L \bD$.
  \end{enumerate}
\end{defn}
Note that in all cases considered in this paper the connected centre of $\bG$
is split and so one can take $\bD$ to be the full connected centre.

\begin{lem} \label{lem: twisted GL gives parameters}
  Let $\bG = \GLb_N \times \GLb_1$ and $\widetilde{\bG} = \bG \rtimes \theta$.
  If $c \in FS(\widetilde{\bG})$ occurs in $I_{\disc}^{\widetilde{\bG}}$ and
  $\chi$ is the central character of $c$, then there is a unique $\psi \in
  \Psi(\widetilde{\bG}, \chi)$ such that $c$ is associated to $\psi$.
\end{lem}
\begin{proof}
  This simply follows from Remark \ref{rema: FS} (1) and the above
  description in the case at hand
  of the pairs $(\tilde{w}, \pi_{\bL})$ with $\tilde{w} \in W(\bL,
  \widetilde{\bG})_{\reg}$, $\pi_{\bL} \in \Pi_{\disc}(\bL)$ and
  $\pi_{\bL}^{\tilde{w}} \simeq \pi_{\bL}$.
\end{proof}

\begin{rem} \label{rem: Idisc for tGL}
  Let $\bG = \GLb_N \times \GLb_1$ and $\widetilde{\bG} = \bG \rtimes
  \theta$.
  \begin{enumerate}
    \item For $\bP$ a parabolic subgroup of $\bG$ with Levi $\bL$ and
      $\pi_{\bL} \in \Pi_{\disc}(\bL, \chi_{\bG, \bL})$, the parabolically
      induced representation $\cA^2(\bU_{\bP}(\A_F) \bL(F) \backslash
      \bG(\A_F))_{\pi_{\bL}}$ is irreducible by \cite{MR1026752} (implying
      multiplicity one for the discrete automorphic spectrum for $\bL$) and
			\cite[\S 0.2]{Bernstein_unitaryGL}, \cite[Theorem 17.6]{Vogan_unitaryGL}
			(irreducibility of unitary parabolic induction for general linear groups).
    \item It follows from \cite{MR623137}, \cite{MR1026752} and Lemma \ref{lem:
      twisted GL gives parameters} that for $c \in FS(\tbG)$ occurring in
      $I_{\disc}^{\tbG, \bL}$, $\bL$ is determined by $c$.
    \item For $S \subset S'$, the linear form $I_{\disc, \nu, c^{S'}}^{\tbG}$ on
      $I(\tbG_{S'})$ is simply the tensor product of $I_{\disc, \nu,
      c^S}^{\tbG}$ with the unramified linear form on $I(\tbG_{S' \smallsetminus
      S})$ corresponding to the Satake parameters $(c_v)_{v \in S'
      \smallsetminus S}$ (see Remark \ref{rem:twisted_unr_rep}). This is
      particular to $\GLb_n$ and is a direct consequence of strong multiplicity
      one. Also by strong multiplicity one for a given $c^S \in FS^S(\tbG)$
      there is at most one $\nu \in IC(\tbG)$ such that $I_{\disc, \nu, c^S}
      \neq 0$. By these remarks, for $c \in FS(\tbG)$ we have a well-defined
      linear form $I_{\disc, c}^{\tbG}$ on $I(\tbG)$, whose restriction to
      $I(\tbG_S)$ (for large enough $S$) is $I_{\disc, \nu, c^S}^{\tbG}$ for the
      unique $\nu$ such that this is non-zero, or $0$ if no such $\nu$ exists.
  \end{enumerate}
\end{rem}

\subsection{Elliptic endoscopic groups}\label{subsec: elliptic
  endoscopic for GL4 times GL1}

Consider the split group $\Gammab = \GLb_4 \times \GLb_1$ over ~$F$ and its
automorphism $\theta : (g, x) \mapsto (J {}^t g^{-1} J^{-1}, x \det g)$, where
\[ J = \begin{pmatrix}
			0 &  0 & 0 & -1 \\
			0 &  0 & 1 &  0 \\
			0 & -1 & 0 &  0 \\
		  1 &  0 & 0 &  0
	\end{pmatrix} \]
was chosen so that the usual pinning of $\GLb_4 \times \GLb_1$ is stable under
$\theta$. Note that if $(\pi, \chi)$ is a representation of $\Gammab(F_v)$ for
some place $v$ of $F$, then $(\pi, \chi) \circ \theta \simeq (\tilde{\pi}
\otimes (\chi \circ \det), \chi)$. The dual group $\widehat{\Gammab}$ is
naturally identified with $\GL_4(\C) \times \GL_1(\C)$, and $\widehat{\theta}(g,
x) = (\widehat{J} {}^t g^{-1} \widehat{J}^{-1} x, x)$, where $\widehat{J} = J$
(but with coefficients in a different field). Denote $\tGamma = \Gammab \rtimes
\theta$ (that is, the non-identity connected component of $\Gammab \rtimes
\{1,\theta\}$). We consider twisted endoscopy with $\omega = 1$.

Then the elliptic endoscopic data $(\bH, \CH, s, \xi)$ for $\tGamma $ are easily
seen to be of the following form.
\begin{enumerate}
	\item $\bH = \GSpinb_5$, dual $\widehat{\bH} = \GSp_4$, for $s = 1$:
		The first projection identifies $\xi_1(\widehat{\GSpinb_5}) =
		\widehat{\Gammab}^{\widehat{\theta}}$ with the general symplectic group
		defined by $\widehat{J}$, and the ``similitude factor'' morphism
		$\widehat{\GSpinb_5} \rightarrow \GL_1$ equals $\pr_2 \circ
		\xi_1|_{\widehat{\GSpinb_5}}$. Both $\Gammab$ and $\GSpinb_5$ are split, so
		there is an obvious choice for ${}^L \xi : {}^L \GSpinb_5 \rightarrow {}^L
		\Gammab$.
	\item $\GSpinb_4^{\alpha}$, with $\alpha \in F^{\times} / F^{\times, 2}$, dual
		$\widehat{\GSpinb_4^\alpha} = \GSO_4$ with action of $\Gal(E/F)$ if $\alpha$
		is not a square, where $E = F(\sqrt{\alpha})$. Pick $s = \diag(-1,-1,1,1)$,
    then $\widehat{\Gammab}^{\Ad(s) \circ \widehat{\theta}} = \GO_4$ for the
    Gram matrix
		\[ \begin{pmatrix}
			0 &  0 &  0 &  1 \\
			0 &  0 & -1 &  0 \\
			0 & -1 &  0 &  0 \\
		  1 &  0 &  0 &  0
		\end{pmatrix}. \]
		If $\alpha = 1$ the group $\GSpinb_4$ is split and we choose the obvious
		${}^L \xi$. Otherwise let $c$ be the non-trivial element of $\Gal(E/F)$, and
		define ${}^L \xi$ by mapping $1 \rtimes c$ to
		\[ \begin{pmatrix}
			1 & 0 & 0 & 0 \\
			0 & 0 & 1 & 0 \\
			0 & 1 & 0 & 0 \\
			0 & 0 & 0 & 1
		\end{pmatrix}, 1. \]
  \item $\bR^\alpha := (\GSpinb_2^{\alpha} \times \GSpinb_3 ) / \{ (z,z^{-1}) |
    z \in \GLb_1 \}$, for non-trivial $\alpha$. The dual $\widehat{\bR^\alpha}$
    is the subgroup of $\GSO_2 \times \GSp_2$ of pairs of elements with equal
    similitude factors, and $\Gal(E/F)$ acts on the first factor. Let $s =
    \diag(-1,1,1,1)$, so that
		\[ \xi_{\bR}^\alpha(\widehat{\bR^\alpha}) = \left\{ \diag (x_1, A, x_2) \
		\middle|\  A \in \GLb_2,\, x_1x_2 = \det A \right\}. \]
		Define ${}^L \xi$ by mapping $1 \rtimes c$ to
		\[ \begin{pmatrix}
			0 & 0 & 0 & 1 \\
			0 & 1 & 0 & 0 \\
			0 & 0 & 1 & 0 \\
			1 & 0 & 0 & 0
		\end{pmatrix}, 1. \]
\end{enumerate}

We also need to consider the elliptic endoscopic groups for~$\GSpinb_5$
and~$\GSpinb_4$. Let $\bH_{1}$ be the unique non-trivial elliptic endoscopic
group for $\GSpinb_5$, so that $\bH_{1} \simeq \GLb_2 \times \GLb_2 / \{ (zI_2,
z^{-1}I_2) \}$. Then $\widehat{\bH}_1$ is the subgroup of $\GSp_2(\C) \times
\GSp_2(\C)$ of pairs of elements with equal similitude factors, so we have an
obvious embedding of dual groups $\widehat{\bH}_1 \to \widehat{\GSpinb_5} =
\GSp_4(\C)$, inducing an embedding of $L$-groups ${}^L \xi' : {}^L \bH_1
\rightarrow {}^L \GSpinb_5$.

Let $\bH_{2}^{\alpha}$ be the elliptic endoscopic group for $\GSpinb_4$
associated to $\alpha \in F^{\times} / F^{\times, 2}$, $\alpha \neq 1$, so that
$\bH_{2}^{\alpha} \simeq \GSpinb_2^{\alpha} \times \GSpinb_2^{\alpha} / \{
(z,z^{-1}) | z \in \GLb_1 \}$. Recall that $\GSpinb_2^{\alpha}$ is naturally
isomorphic to $\Res_{F(\sqrt{\alpha})/F}(\GLb_1)$. Then~$\widehat{\bH_2^\alpha}$
is the subgroup of $\GSO_2(\C) \times \GSO_2(\C)$ consisting of pairs of
elements with equal similitude factors, so we again have an obvious embedding of
dual groups $\widehat{\bH_2^\alpha} \to \widehat{\GSpinb_4} = \GSO_4(\C)$.
If~$\alpha=1$ then this trivially extends to an embedding of $L$-groups, while
if $\alpha\ne 1$, writing $\Gal(F(\sqrt{\alpha})/F) = \{1,c\}$, define  ${}^L
\xi' : {}^L \bH_2^\alpha \rightarrow {}^L \GSpinb_4$ by mapping $1\rtimes c$ to
\[ \begin{pmatrix}
			0 & 1 & 0 & 0 \\
			1 & 0 & 0 & 0 \\
			0 & 0 & 0 & 1 \\
			0 & 0 & 1 & 0
		\end{pmatrix}, 1. \]

\subsection{Stabilisation of the trace formula}\label{subsec:
  stabilisation of the trace formula}

We will need to use the stabilisation of the (twisted) trace formula for
$\tGamma$ and its elliptic endoscopic groups. Consider the latter first: let
$(\bH, \CH, s, \xi)$ be an elliptic endoscopic
datum for $(\Gammab, \tGamma)$. The stabilisation of the
trace formula for~$\bH$ is as follows. Fix $\nu \in IC(\bH)$, $S$ a big enough
set of places, and~ $c \in FS^S(\bH)$. Choose representatives $(\bH', \CH', s,
\xi)$ for the isomorphism classes of elliptic endoscopic data for $\bH$, and for
each representative choose ${}^L \xi' : {}^L \bH' \rightarrow {}^L \bH$
extending $\xi$ (for example as in the previous section).
It induces maps ${}^L \xi' : FS(\bH') \rightarrow FS(\bH)$ and ${}^L \xi' :
IC(\bH') \rightarrow IC(\bH)$. Inductively define a linear form on $I(\bH(F_S))$
by
\numequation \label{equ:def stable dist TF}
	S_{\disc, \nu, c}^{\bH}(f) := I_{\disc, \nu, c}^{\bH}(f) -
  \sum_{\substack{\frake' = (\bH', \CH', s', \xi') \\ \bH' \neq \bH}}
	\iota(\frake') \sum_{\substack{c' \mapsto c \\ \nu' \mapsto \nu}} S_{\disc,
	\nu', c'}^{\bH'}(f^{\bH'})
\end{equation}
where the sum is over equivalence classes of nontrivial elliptic endoscopic data
for~$\bH$, $f^{\bH'}$ is a transfer of $f$ (see Section \ref{subsec: endoscopic
groups and transfer}), and the constants $\iota(\frake')$ are recalled after the
following theorem.

\begin{thm}[{\cite[Global Theorems 2 and 2' and Lemma 7.3(b)]{ArthurSTF1}}]
	\label{thm: STF}
	The linear form $S_{\disc, \nu, c}^{\bH}$ is stable, i.e.\ factors through
	$SI(\bH(F_S))$. 
\end{thm}
Note that in general \eqref{equ:def stable dist TF} is only well-defined thanks
to Theorem \ref{thm: STF} applied to $\bH'$. However, for the groups $\bH$
considered here, and for any non-trivial endoscopic group $\bH'$, the only
elliptic endoscopic group for $\bH'$ is $\bH'$, and so $S_{\disc}^{\bH'} =
I_{\disc}^{\bH'}$.

Let us recall the definition of $\iota(\frake')$, both for ordinary endoscopy
and for twisted endoscopy. Assume that $\tbG$ is a twisted space and $\frake =
(\bH, \cH, s, \xi)$ is an elliptic endoscopic datum. Let
\[ \iota(\frake) = \frac{\tau(\bG)}{\tau(\bH)} \frac{\left| \pi_0 \left(
Z(\Ghat)^{\Gal_F, 0} \cap \That^{\hat{\theta}, 0} \right) \right|}{|
\pi_0(\Aut(\frake)) |} \]
where $\tau$ is the Tamagawa number and the superscript $0$ denotes the identity
component. We have not included the factor $|\det(1-\theta | \dots)|^{-1}$ from
\cite[VI.5.1]{SFTT2} because of Remark \ref{rema: FS}~(2); compare with the
definition on p.\ 109 of \cite{MR1687096} using \cite[Lem.\ 6.4.B]{MR1687096}.
Recall \cite[p.\ 693]{SFTT2} that there is a short exact sequence
\[ 1 \rightarrow \left( Z(\Ghat) / Z(\Ghat) \cap \That^{\hat{\theta}, 0}
	\right)^{\Gal_F} \rightarrow \Aut(\frake) / \Hhat \rightarrow \Out(\frake)
	\rightarrow 1. \]
In the ordinary (non-twisted) case we have $\That^{\hat{\theta}, 0} = \That
\supset Z(\Ghat)$ and thus $\iota(\frake) = \tau(\Ghat) \tau(\Hhat)^{-1} |
\Out(\frake)|^{-1}$. The only twisted case that we need in this paper is the
case of $\tGamma$, when $\That^{\hat{\theta}, 0} = \{ ((t_1, \dots, t_4), x) |
\forall i,\, t_i = t_{5-i}^{-1} x \}$ and so $Z(\Ghat) \cap \That^{\hat{\theta},
0} \simeq \C^{\times}$. Similarly it is easy to see that $Z(\Ghat) / Z(\Ghat)
\cap \That^{\hat{\theta}, 0} \simeq \C^{\times}$ with trivial action of
$\Gal_F$, so we can conclude that $\iota(\frake) = \tau(\Gammab) \tau(\bH)^{-1}
|\Out(\frake)|^{-1}$ for any elliptic endoscopic datum $\frake = (\bH, \CH, s,
\xi)$ of $\tGamma$.

Let us make the constant $\iota(\frake)$ explicit in the only two cases where it
will be needed in this paper:
\begin{enumerate}
	\item For the elliptic endoscopic group $\bH_1$ of $\GSpinb_5$, $\iota(\frake)
		= 1/4$.
	\item For the elliptic endoscopic group $\GSpinb_5$ of $\tGamma$,
		$\iota(\frake) = 1$.
\end{enumerate}

We can finally state the stabilisation of the twisted trace formula for
$(\Gammab, \tGamma)$. As in the case of ordinary endoscopy we fix
representatives $\frake = (\bH, \CH, s, \xi)$ of isomorphism classes of elliptic
endoscopic data for $\tGamma$ and for each $\frake$ we also choose an
$L$-embedding ${}^L \xi : {}^L \bH \rightarrow {}^L \bG$ extending $\xi$ (for
example the ones defined in the previous section).

\begin{thm}[{\cite[X.8.1]{SFTT2}}] \label{thm:STTF}
	For any $\nu$ and $c$ we have
	\[ I_{\disc, \nu, c}^{\tGamma}(f) = \sum_{\frake = (\bH, \CH, s, \xi)}
		\iota(\frake) \sum_{\substack{\nu' \mapsto \nu \\ c' \mapsto c}}
		S_{\disc, \nu', c'}^{\bH}(f^{\bH}) \]
	where the first sum is over equivalence classes of elliptic endoscopic data
	for $\tGamma$.
\end{thm}

\section{Restriction of automorphic representations}
\label{sec:restriction}

\subsection{Restriction for general groups}

Let us recall a consequence of \cite[\S 4]{HiragaSaito} that we will need. Since
in all cases considered in this paper the assumption of
\cite[Proposition 1 (iii)]{Chenevier_note_res_auto} will be satisfied, one can
use the more precise result of \cite{Chenevier_note_res_auto} (which can be
formally generalised from cuspidal to square-integrable forms) instead. Consider
an injective morphism $\bG \hookrightarrow \bG'$ between connected reductive
groups over a number field $F$ such that $\bG$ is normal in $\bG'$ and
$\bG'/\bG$ is a torus. Choose a maximal compact subgroup $K_{\infty}'$ of
$\bG'(F \otimes_{\Q} \R)$; then $K_{\infty} := \bG(F \otimes_{\Q} \R) \cap
K_{\infty}'$ is a maximal compact subgroup of $\bG(F \otimes_{\Q} \R)$. Note
that if $\pi'$ is an irreducible unitary admissible $(\fg', K_{\infty}') \times
\bG'(\A_{F,f})$-module then $\Res^{\bG'}_\bG(\pi')$ is a unitary admissible
$(\fg, K_{\infty}) \times \bG(\A_{F,f})$-module, but it has infinite length in
general. We have a $(\fg, K_{\infty}) \times \bG(\A_{F,f})$-equivariant map
\[ \res^{\bG'}_\bG : \CA^2( \mathfrak{A}_{\bG'} \bG'(F) \backslash \bG'(\A_F))
\rightarrow \CA^2( \mathfrak{A}_\bG \bG(F) \backslash \bG(\A_F) ) \]
obtained by restricting automorphic forms. The fact that $\res^{\bG'}_\bG$ takes
values in $\CA^2( \mathfrak{A}_\bG \bG(F) \backslash \bG(\A_F) )$ is a routine
verification, except for square-integrability which follows from the proof of
\cite[Lemma 4.19]{HiragaSaito} (see also Remark 4.20 \emph{op.\ cit.}). If $\pi'
\in \Pi_{\disc}(\bG')$ and $\iota : \pi' \hookrightarrow \CA^2(
\mathfrak{A}_{\bG'} \bG'(F) \backslash \bG'(\A_F))$, then
$\res^{\bG'}_\bG(\iota(\pi'))$ is naturally identified with a quotient of
$\Res^{\bG'}_\bG(\pi')$. This quotient can be proper and of infinite length, but
in any case it is \emph{non-zero}. In particular there exists an irreducible
constituent $\pi$ of $\Res^{\bG'}_\bG(\pi')$ such that $\pi \in
\Pi_{\disc}(\bG)$. In this situation we will say that $\pi$ is an
\emph{automorphic restriction} of $\pi'$. Unsurprisingly, this notion of
restriction is compatible with the Satake isomorphism at almost all places:

\begin{lem}[Satake] \label{lem: compat Satake restriction}
  Suppose that $\pi \simeq \otimes'_v \pi_v \in \Pi_{\disc}(\bG)$ is an
  automorphic restriction of $\pi' \simeq \otimes'_v \pi'_v \in
  \Pi_{\disc}(\bG')$, then for almost all places $v$ of $F$ the Satake parameter
  $c(\pi_v)$ of $\pi_v$ is the image of $c(\pi'_v)$ under the natural map
	\[ \left( \widehat{\bG'} \rtimes \Frob_v \right)^{\semis} /
		\widehat{\bG'}-\mathrm{conj} \longrightarrow \left( \Ghat \rtimes \Frob_v
	\right)^{\semis} / \Ghat-\mathrm{conj}. \]
\end{lem}
\begin{proof}
  For almost all places $v$, $\pi_v$ is the unique unramified direct summand in
  $\Res^{\bG'(F_v)}_{\bG(F_v)}(\pi'_v)$. The result follows from \cite[\S
  7.2]{Satake} applied to $\bG \times \bT \rightarrow \bG'$, where $\bT$ is any
  central torus in $\bG$ isogenous to $\bG'/\bG$, and the translation in terms
  of dual groups \cite[Prop.\ 6.7]{MR546608}.
\end{proof}

Let us now formulate a direct consequence of \cite[Theorem 4.14]{HiragaSaito},
ignoring multiplicities.

\begin{thm}[Hiraga--Saito] \label{thm: surjectivity automorphic res}
  The map $\res^{\bG'}_\bG$ is surjective, and so any discrete automorphic
  representation for $\bG$ is an automorphic restriction of a discrete
  automorphic representation for $\bG'$. In other words, there exists a
  surjective map
  \[ \mathrm{ext}^{\bG'}_\bG : \Pi_{\disc}(\bG) \longrightarrow
    \Pi_{\disc}(\bG') / \left( \bG'(\A_F) / \bG(\A_F) \bG(F) \mathfrak{A}_{\bG'}
    \right)^{\vee}  \]
  such that for any $\pi' \in \mathrm{ext}^{\bG'}_\bG(\pi)$, $\pi$ is a
  subrepresentation of $\Res^{\bG'}_\bG(\pi')$.
\end{thm}
In general this map $\mathrm{ext}^{\bG'}_\bG$ is not uniquely determined.

We will mainly use this result for $\Spb_4 \hookrightarrow \GSpinb_5$. This will
be fruitful thanks to exterior square functoriality for $\GLb_4$
\cite{MR1937203} and the commutativity of the following commutative diagram of
dual groups:
\numequation \label{eqn: commutative diagram Sp4 restriction}
\begin{tikzcd}[column sep=5em]
	\widehat{\GSpinb_5} = \GSp_4 \arrow[r] \arrow[d, "{{}^L \xi}"] &
	\widehat{\Spb_4} = \SO_5 \arrow[d, "{\Std \oplus 1}"] \\
	\GL_4 \times \GL_1 \arrow[r, "f"] & \SL_6
\end{tikzcd}
\end{equation}
where $f := \bigwedge^2(\pr_1) \otimes \pr_2^{-1}$.

\section{Global Arthur--Langlands parameters for
	\texorpdfstring{$\GSpinb_5$}{GSpin(5)}}
\label{sec: global parameters}

\subsection{Classification of global parameters}

Let $\chi : \A_F^\times/F^\times \rightarrow \C^{\times}$ be a continuous
unitary character. Recall the set~$\Psi(\tGamma,\chi)$ of formal global
parameters defined in Section~\ref{subsec: global parameters}.

In this section we will denote the functorial transfer $\GLb_2 \times \GLb_2 \to
\GLb_4$ by $(\pi_1,\pi_2) \mapsto \pi_1 \boxtimes \pi_2$ (we will only need this
  in the weak sense, i.e.\ compatibility with Satake parameters at all but
finitely many places). This transfer exists for cuspidal representations by
~\cite{MR1792292}, and is easily extended to discrete representations:
\begin{itemize}
	\item if~$\pi_1=\eta[2]$ for some character~$\eta$ and $\pi_2$ is cuspidal,
		then $\pi_1 \boxtimes \pi_2 = \eta \otimes \pi_2[2]$.
	\item if~$\pi_1=\eta_1[2]$ and $\pi_2=\eta_2[2]$, then $\pi_1 \boxtimes \pi_2
		= \eta_1 \eta_2 \boxplus \eta_1 \eta_2[3]$.
\end{itemize}

Recall that in Section~\ref{subsec: elliptic endoscopic for GL4 times GL1} we
fixed a representative $(\bH, \CH, s, \xi)$ for each equivalence class of
elliptic endoscopic data for $\tGamma$, and in each case an $L$-embedding ${}^L
\xi : {}^L \bH \rightarrow {}^L \Gammab = \widehat{\Gammab} \times W_F$. We also
fixed, for each $\bH$ as above, a representative $(\bH', \CH', s', \xi')$ for
each equivalence class of elliptic endoscopic data for $\bH$, as well as an
$L$-embedding ${}^L \xi' : {}^L \bH' \rightarrow {}^L \bH$. We use this generic
notation in the following Proposition, which shows that we may associate a
parameter in the set~$\Psi(\tGamma,\chi)$ to each discrete automorphic
representation of~$\GSpinb_4$ or~$\GSpinb_5$ with central character~$\chi$; we
will refine this in Propositions~\ref{prop: parameter is discrete symplectic}
and~\ref{prop: parameter is discrete orthogonal} to show that these parameters
are in fact contained in the subsets~$\widetilde{\Psi}_{\disc}(\GSpinb_4,\chi)$,
${\Psi}_{\disc}(\GSpinb_5,\chi)$ respectively.

\begin{prop} \label{prop: STTF to associate parameters} \leavevmode
	\begin{enumerate}
		\item For $\bL$ a proper Levi subgroup of $\GSpinb_5$, any $c \in
			FS(\GSpinb_5)$ occurring in $I_{\disc}^{\GSpinb_5, \bL}$ such that
			$\muhat(c) = c(\chi)$ satisfies ${}^L \xi(c) \in \Psi(\tGamma, \chi)$ and
			is not discrete.
    \item Let $\bH = \left( \GLb_2 \times \GLb_2 \right) / \{ (zI_2, z^{-1} I_2
      \,|\, z \in \GLb_1 \} $ be the unique non-trivial elliptic endoscopic
      group for $\GSpinb_5$. Then any $c \in FS(\bH)$ occurring in
      $I_{\disc}^{\bH} = S_{\disc}^{\bH}$ and such that $\muhat(c) = c(\chi)$
      satisfies $({}^L \xi \circ {}^L \xi')(c) \in \Psi(\tGamma, \chi)$.
		\item Let $\bH'$ be a non-trivial elliptic endoscopic group for $\GSpinb_4$.
			Then any $c \in FS(\bH')$ occurring in $I_{\disc}^{\bH'} =
			S_{\disc}^{\bH'}$ and such that $\muhat(c) = c(\chi)$ satisfies $({}^L \xi
			\circ {}^L \xi')(c) \in \Psi(\tGamma, \chi)$.
    \item For $\bL$ a Levi subgroup of $\GSpinb_4$, any $c \in FS(\GSpinb_4)$
      occurring in $I_{\disc}^{\GSpinb_4, \bL}$ and such that $\muhat(c) =
      c(\chi)$ satisfies ${}^L \xi(c) \in \Psi(\tGamma, \chi)$. If $\bL \neq
      \GSpinb_4$ then ${}^L \xi(c)$ is not discrete.
    \item Any $c \in FS(\GSpinb_4)$ occurring in $S_{\disc}^{\GSpinb_4}$ and
      such that $\muhat(c) = c(\chi)$ satisfies ${}^L \xi(c) \in \Psi(\tGamma,
      \chi)$.
    \item Any $c \in FS(\GSpinb_5)$ occurring in $S_{\disc}^{\GSpinb_5}$ and
      such that $\muhat(c) = c(\chi)$ satisfies ${}^L \xi(c) \in \Psi(\tGamma,
      \chi)$.
    \item Any $c \in FS(\GSpinb_5)$ associated to a discrete automorphic
      representation for $\GSpinb_5$ with central character $\chi$ satisfies
      ${}^L \xi(c) \in \Psi(\tGamma, \chi)$.
	\end{enumerate}
\end{prop}
\begin{proof}
  We use repeatedly the description of~$I_{\disc}^{\widetilde{\bG}, \bL}$
  explained in Section~\ref{subsec: discrete part of spectral side of STTF},
  namely that if~$c\in FS(\widetilde{\bG})$ occurs in
  ~$I_{\disc}^{\widetilde{\bG}, \bL}$, then there is a regular element
  $\tilde{w} \in W(\bL, \widetilde{\bG})$, and~$\pi_{\bL} \in \Pi_{\disc}(\bL)$
  such that $\pi^{\tilde{w}} \simeq \pi$ and~$c(\pi_{\bL})$ maps to~$c$ via
  ${}^L \bL \to {}^L \bG$.
	\begin{enumerate}
		\item \label{item: levis of GSpin5}
			The possible proper Levi subgroups~ $\bL$ and the embeddings~${}^L \bL
			\rightarrow {}^L \GSpinb_5$ are listed in
			Section~\ref{subsec:Levi_parametrisation}. In the case at hand, the
			possibilities are
			\begin{enumerate}
				\item $\GLb_1 \times \GSpinb_3 \cong \GLb_1 \times \GLb_2$,
				\item $\GLb_2 \times \GSpinb_1 \cong \GLb_2 \times \GLb_1$, and
				\item $\GLb_1 \times \GLb_1 \times \GSpinb_1 \cong \GLb_1 \times \GLb_1
					\times \GLb_1$.
			\end{enumerate}
			In the first case we find that the corresponding parameter is of the form
			$\eta \boxplus \pi \boxplus \eta$, where~$\pi$ is a unitary discrete
			automorphic representation of~$\GLb_2(\A_F)$ with~$\omega_\pi=\chi$ and
			$\eta^2 = \chi$; in the second case, that the parameter is of the
			form~$\pi \boxplus \pi$, where ~$\pi$ is a unitary discrete automorphic
			representation of~$\GLb_2(\A_F)$ such that $\pi^{\vee} \otimes \chi \simeq
			\pi$; and in the third case that the parameter is of the
			form~$\eta_1 \boxplus \eta_2 \boxplus \eta_2 \boxplus \eta_1$ with
			$\eta_1^2 = \eta_2^2 = \chi$.

		\item \label{item: param for endoscopic of GSpin5}
			By the description of~$\bH$ as a quotient, $c$ corresponds to a pair
			$(\pi_1, \pi_2)$ with each $\pi_i$ either an element of
			$\Pi_{\disc}(\GLb_2)$ with $\omega_{\pi_i} = \chi$ or $\eta \boxplus \eta$,
			with $\eta^2 = \chi$. It is easy to check that $({}^L \xi \circ {}^L
			\xi')(c) = (c(\pi_1) \oplus c(\pi_2), c(\chi))$, so that the corresponding
			parameter is~$\pi_1\boxplus\pi_2$.

    \item This is similar to the previous two parts. Write~$\bH'=\bH_2^\alpha$
      as in Section~\ref{subsec: elliptic endoscopic for GL4 times GL1}, so that
      an element of $\Pi_{\disc}(\bH')$ is given by a pair of automorphic
      representations $\rho_1, \rho_2$ for the torus $\GSpinb_2^{\alpha} \simeq
      \Res_{E/F}(\GLb_1)$ (here $E = F(\sqrt{\alpha})$) whose restrictions to
      $\GLb_1$ are equal to $\chi$. Then via the natural embedding ${}^L
      \GSpinb_2^{\alpha} = \GSO_2 \rtimes \Gal(E/F)) \rightarrow \GL_2$, we have
      $({}^L \xi \circ {}^L \xi')(c) = (c(\pi_1) \oplus c(\pi_2), c(\chi))$
      where $\pi_1$ and $\pi_2$ are the cuspidal automorphic representations for
      $\GLb_2$ with central character $\chi$ automorphically induced (for $E /
      F$) from $\rho_1$ and $\rho_2$ seen as unitary characters of
      $\A_E^{\times} / E^{\times}$.

		\item \label{item: param for GSpin4 Levi}
			Recall that~$\GSpinb_4$ is isomorphic to the subgroup of elements
			of~$\GLb_2 \times \GLb_2$ such that the determinants of the two elements
			are equal. Accordingly, if $c$ is discrete automorphic, i.e.\ it
      occurs in $I_{\disc}^{\GSpinb_4, \GSpinb_4}$, then  by Theorem \ref{thm:
      surjectivity automorphic res} it comes from the automorphic restriction of
      some $(\pi_1, \pi_2) \in \Pi_{\disc}(\GLb_2 \times \GLb_2)$, with
      $c(\omega_{\pi_1}) c(\omega_{\pi_2}) = c(\chi)$ and so $\omega_{\pi_1}
      \omega_{\pi_2} = \chi$. Then ${}^L \xi(c) = (c(\pi_1) \otimes c(\pi_2),
      c(\chi))$, and the corresponding parameter is~$\pi_1\boxtimes\pi_2$.

      Otherwise $c$ occurs in $I_{\disc}^{\GSpinb_4, \bL}$ for some proper Levi
      subgroup. By the description given in
      Section~\ref{subsec:Levi_parametrisation}, we see that $\bL$ is isomorphic
      to $\GLb_2 \times \GSpinb_0^1 \cong \GLb_2 \times \GLb_1$ or to $\GLb_1
      \times \GLb_1 \times \GSpinb_0^1 \cong \GLb_1 \times \GLb_1 \times
      \GLb_1$. In either case we can compute explicitly as in ~(\ref{item: levis
      of GSpin5}), and we find that  we obtain parameters of the form~$\pi
      \boxplus \pi$, where ~$\pi$ is a discrete automorphic representation
      of~$\GLb_2(\A_F)$ such that $\pi^{\vee} \otimes \chi \simeq \pi$, and
      parameters of the form~$\eta_1 \boxplus \eta_2 \boxplus \eta_2 \boxplus
      \eta_1$ with $\eta_1^2 = \eta_2^2 = \chi$.
			
		\item \label{item: STF GSpin4}
      This follows immediately from  the stable trace formula \eqref{equ:def
      stable dist TF} for $\GSpinb_4$ and the two previous points.

		\item \label{item: STTF GL4}
      This follows from  the stable twisted trace formula of
      Theorem~\ref{thm:STTF}. Observe that we can associate an element of
      $\Psi(\tGamma, \chi)$ to any family of Satake parameters occurring in
      $S_{\disc}^{\GSpinb_4}$ or to $I_{\disc}^{\tGamma}$; in the former case
      this is the content of~(\ref{item: STF GSpin4}), and in the latter case it
      follows from Lemma \ref{lem: twisted GL gives parameters}.

    \item This follows as in~(\ref{item: STTF GL4}), this time using  the stable
      trace formula for $\GSpinb_5$, and applying parts~(\ref{item: param for
			endoscopic of GSpin5}) and~(\ref{item: STTF GL4}). \qedhere
	\end{enumerate}
\end{proof}

We can now prove the symplectic/orthogonal alternative for~$\GLb_4$. This is
well known, and can also be proved using the theta correspondence or converse
theorems; indeed, \cite[Thm.\ 4.26]{MR3227529} proves a corresponding result
for~$\GSpinb$ groups of arbitrary rank, showing that a  $\chi$-self dual
cuspidal automorphic representation~$\pi$ of~$\GLb_n$ arises as the transfer of
a globally generic representation of a~$\GSpinb$ group which is uniquely
determined by the data of which of the corresponding symmetric and alternating
square $L$-functions has a pole, together with  the central character of~$\pi$.

However, our emphasis here is slightly different (we wish to determine which
representations have Satake parameters which occur in the discrete spectrum
of~$\GSpinb_5$), and in any case we find it instructive to show how this follows
from the trace formula together with Kim's exterior square
transfer~\cite{MR1937203}.

The following remark will help us to distinguish parameters coming
from different endoscopic subgroups.

\begin{rem}\label{rem: getting alpha from parameter}The sets 
  \[ \left( {}^L \xi( FS(\GSpinb_5) ) \cup {}^L \xi( FS(\GSpinb_4) ) \right)
    \text{ and } \left( {}^L \xi( FS(\GSpinb_4^\alpha) ) \cup {}^L \xi(
		FS(\bR^\alpha) )\right) \]
	(where  $\alpha \in F^{\times} / F^{\times, 2}$ is non-trivial) are pairwise
	disjoint, because we can recover $\alpha$ as follows (by the definition
	of~${}^L \xi$): for $\bH=\GSpinb_4^\alpha$ or~$\bH=\bR^\alpha$, $c^S \in
	FS(\bH)$ and $(g^S, x^S) = {}^L \xi(c^S)$, for any $v \not\in S$, then $v$
	splits in $F(\sqrt{\alpha})$ if and only if $\det g_v = x_v^2$. On the other
	hand if~$\bH=\GSpinb_5$ or~$\bH=\GSpinb_4$ then we always have $\det g_v =
	x_v^2$.
\end{rem}

\begin{prop} \label{prop:sympl_orth_alt}
	Let $\pi$ be a $\chi$-self dual cuspidal automorphic representation for
  $\GLb_4$, and let $S$ be a finite set of places of $F$ containing all
  Archimedean places and all non-Archimedean places where $\pi$ is ramified.
	\begin{enumerate}
    \item If there are cuspidal automorphic representations $\pi_i$ for $\GLb_2$
      such that $\omega_{\pi_1} \omega_{\pi_2} = \chi$ and $\pi \simeq \pi_1
      \boxtimes \pi_2$, then $L^S(s, \Sym^2(\pi) \otimes \chi^{-1})$ has a pole
      at $s=1$, and there exists~$c' \in FS(\GSpinb_4)$ occurring in
      $S_{\disc}^{\GSpinb_4}$ and such that ${}^L \xi(c') = (c(\pi), c(\chi))$.
    \item If~$\omega_{\pi} \ne \chi^2$, then~$\pi$ is an Asai transfer from a
      cuspidal automorphic representation of~$\GLb_2/E$, for~$E =
      F(\sqrt{\alpha})$ the quadratic extension of~$F$ corresponding to the
      quadratic character $\chi^2 / \omega_{\pi}$, and $L^S(s, \Sym^2(\pi)
      \otimes \chi^{-1})$ has a pole at $s=1$. Furthermore there exists~$c' \in
      FS(\GSpinb_4^{\alpha})$ occurring in $S_{\disc}^{\GSpinb_4^\alpha}$ and
      such that ${}^L \xi(c') = (c(\pi), c(\chi))$.
    \item Otherwise \emph{(}i.e.\ if $\omega_{\pi} = \chi^2$ and $\pi$ does not
      come from a pair of automorphic cuspidal representations for $\GLb_2$ as
      in (1)\emph{)} $L^S(s, \bigwedge^2(\pi) \otimes \chi^{-1})$ has a pole at
      $s=1$, $c := {}^L \xi^{-1}(c(\pi), c(\chi)) \in FS(\GSpinb_5)$ occurs in
      $S_{\disc}^{\GSpinb_5}$, and for any large enough $S$ and any $\nu \in
      IC(\GSpinb_5)$
      \[ S_{\disc, \nu, c^S}^{\GSpinb_5} = I_{\disc, \nu, c^S}^{\GSpinb_5} =
        I_{\disc, \nu, c^S}^{\GSpinb_5, \GSpinb_5}. \]
	\end{enumerate}
\end{prop}
\begin{proof}
	\begin{enumerate}
		\item It suffices to note that
			\[ L^S(s, \bigwedge^2(\pi_1 \boxtimes \pi_2) \otimes (\omega_{\pi_1}
			\omega_{\pi_2})^{-1}) = L^S(s, \ad^0(\pi_1)) L^S(s, \ad^0(\pi_2)) \]
			is holomorphic at $s=1$ since for each $i = 1,2$ the automorphic
			representation $\ad^0(\pi_i)$ defined in \cite{MR533066} is either
			\begin{enumerate}
				\item a self-dual cuspidal automorphic representation for $\GLb_3$
					(\cite[Theorem 9.3]{MR533066}),
				\item $\eta \boxplus \sigma$ where $\eta$ is a character of order two
					and $\sigma$ is a self-dual cuspidal automorphic representation for
					$\GLb_2$ such that $\eta \otimes \sigma \simeq \sigma$ (\cite[Remark
					9.9]{MR533066} with $(\Omega/\Omega')^2 \neq 1$),
				\item $\eta_1 \boxplus \eta_2 \boxplus \eta_1 \eta_2$ where $\eta_1$ and
					$\eta_2$ are distinct characters of order two (\cite[Remark
					9.9]{MR533066} with $\Omega/\Omega'$ of order two).
			\end{enumerate}
			As in the proof of Proposition~\ref{prop: STTF to associate
			parameters}~(\ref{item: param for GSpin4 Levi}) we see that the element
			$c' \in FS(\GSpinb_4)$ which is the image of $(c(\pi_1), c(\pi_2))$ via
			either of the two tensor product morphisms $\GL_2 \times \GL_2 \rightarrow
      \GSO_4 = \widehat{\GSpinb_4}$ occurs in $I^{\GSpinb_4}_{\disc}$ and
      $S^{\GSpinb_4}_{\disc}$, and satisfies ${}^L \xi(c') = (c(\pi), c(\chi))$.

    \item By Remark \ref{rem: Idisc for tGL} (2) we know that $(c(\pi),
      c(\chi))$ does not occur in $I^{\tGamma, \bL}_{\disc}$ for any proper Levi
      subgroup $\bL$ of $\Gammab$. Since $(\pi, \chi)$ occurs with multiplicity
      one in the discrete automorphic spectrum for $\Gammab$, the automorphic
      extension $\widetilde{\pi}$ of $\pi$ to $\tGamma$ (provided by \eqref{equ:
      automorphic twisted action} for $\bL = \bG$, with $\tilde{w} = \theta$)
      has non-vanishing trace (see \cite[Proposition A.5]{Lemaire} for the
      $p$-adic case, the Archimedean case is proved similarly) and so $(c(\pi),
      c(\chi))$ occurs in $I_{\disc}^{\tGamma}$. In the stabilisation of the
      twisted trace formula (Theorem \ref{thm:STTF}) this contribution comes
      from at least one elliptic endoscopic datum, i.e.\ there is an elliptic
      endoscopic group $\bH$ and $c' \in FS(\bH)$ occurring in $S^{\bH}_{\disc}$
      such that ${}^L \xi(c') = (c(\pi), c(\chi))$.

			The character~$\omega_\pi/\chi^2$ corresponds to some quadratic
			extension~$E=F(\sqrt{\alpha})$, and by Remark~\ref{rem: getting alpha from
			parameter}, in the stabilisation of the twisted trace formula for
			$\tGamma$ this contribution must come from~$S_{\disc}^{\GSpinb_4^\alpha}$
			or~$S_{\disc}^{\bR^\alpha}$ (a priori non-exclusively). In the latter
			case, we see that~$\pi$ has the same Satake parameters
			as~$\pi_1\boxplus\pi_2$, where~$\pi_1$ is either a discrete automorphic
			representation for $\GLb_2$ with central character~$\chi$,
			or~$\pi_1=\eta\boxplus\eta$ with $\eta^2=\chi$, and~$\pi_2$ is a cuspidal
			$\chi$-self-dual automorphic representation for $\GLb_2$, corresponding to
			the extension~$E/F$; but either possibility contradicts \cite{MR623137}.

			Thus~$(c(\pi), c(\chi))$ comes from~$S_{\disc}^{\GSpinb_4^\alpha}$.
      \begin{itemize}
        \item If it comes from~$S_{\disc}^{\bH} = I_{\disc}^{\bH}$ for some
          elliptic endoscopic group $\bH \neq \GSpinb_4^{\alpha}$ for
          $\GSpinb_4^{\alpha}$ then
          \[ \bH \simeq \GSpinb_2^{\beta} \times \GSpinb_2^{\gamma} / \{
            (z,z^{-1}) | z \in \GLb_1 \}\]
          for some $\beta, \gamma \in F^{\times} / F^{\times, 2} \smallsetminus
          \{1\}$ satisfying $\beta \gamma = \alpha$. Recall that
          $\GSpinb_2^{\beta} \simeq \Res_{F(\sqrt{\beta})/F} \GLb_1$. Then we
          see that $\pi = \pi_1 \boxplus \pi_2$ where $\pi_1$ (resp.\ $\pi_2$)
          is the automorphic induction of a character of
          $\A_{F(\sqrt{\beta})}^{\times} / F(\sqrt{\beta})^{\times}$ (resp.\
          $\A_{F(\sqrt{\gamma})}^{\times} / F(\sqrt{\gamma})^{\times}$) and
          this contradicts the cuspidality of $\pi$.
        \item If $(c(\pi), c(\chi))$ comes from $I_{\disc}^{\GSpinb_4^{\alpha},
          \bL}$ for the proper Levi subgroup $\bL \simeq \GLb_1 \times
          \GSpinb_2^{\alpha}$ of $\GSpinb_4^{\alpha}$ then $\pi = \eta \boxplus
          \pi_1 \boxplus \eta$ where $\eta^2 = \chi$ and $\pi_1$ is the
          automorphic induction of a character of $\A_E^{\times} / E^{\times}$
          and we also get a contradiction with the cuspidality of $\pi$.
      \end{itemize}
      Therefore $(c(\pi), c(\chi))$ comes from a discrete automorphic
      representation for $\GSpinb_4^{\alpha}$. As explained in~\cite[\S
      2.2]{MR2767509} (i.e.\ using Theorem \ref{thm: surjectivity automorphic
      res} for $\GSpinb_4^{\alpha} \hookrightarrow \Res_{E/F} \GLb_2$), this is
      equivalent to~$\pi$ being the Asai transfer of a cuspidal automorphic
      representation~$\pi_E$ of~$\GLb_2(\A_E)$. Then $L^S(s,
      \bigwedge^2\pi\otimes \chi^{-1})=L^S(s,\Ind_E^F(\Sym^2\pi_E\otimes
      \omega_{\pi_E}' )\otimes \chi^{-1})$, where~$\omega_{\pi_E}'$ is the
      $\Gal(E/F)$-conjugate of~$\omega_{\pi_E}$.

      If~$\pi_E$ is not dihedral then $\bigwedge^2\pi$ is cuspidal
      by~\cite[Prop.\ 3.2]{MR2767509}, so it is enough to consider the case
      that~$\pi_E$ is dihedral, induced from a character~$\chi_{E'}$
      of~$\A_{E'}^\times/(E')^\times$, where $E'/E$ is a quadratic extension.
      Then~$\Sym^2\pi_E=\Ind_{E'}^E\chi^2_{E'}\boxplus\chi_{E'}|_{\A^\times_E}$,
      and it is easy to verify explicitly that the isobaric
      representation~$\Ind_E^F(\Sym^2\pi_E\otimes \omega_{\pi_E}' )\otimes
      \chi^{-1}$ cannot contain the trivial character.

    \item \label{item: sympl orth alt dim4 sympl case}
      As in the previous case, $(c(\pi), c(\chi))$ occurs in
      $I_{\disc}^{\tGamma}$. By Remark~\ref{rem: getting alpha from parameter},
      in the stabilisation of the twisted trace formula for $\tGamma$ this
      contribution comes from $\bH = \GSpinb_4$ or $\bH = \GSpinb_5$ on the
      right-hand side. In the former case, as $\GSpinb_4$ embeds into $\GLb_2
      \times \GLb_2$, we would be in the situation of part~(1); so it must
      occur in $S_{\disc}^{\GSpinb_5}$. Moreover, it cannot come from
      $S_{\disc}^{\bH_{1,1}}$ or $I_{\disc}^{\GSpinb_5, \bL}$ for a proper Levi
      $\bL$, as by (the proof of) Proposition~\ref{prop: STTF to associate
      parameters} this would contradict strong multiplicity one, so we can
      conclude that $I_{\disc, c}^{\GSpinb_5, \GSpinb_5} = S_{\disc,
      c}^{\GSpinb_5}$ is not identically zero.

			In particular there is a discrete automorphic representation $\Pi$ for
			$\GSpinb_5$ such that ${}^L \xi(c(\Pi)) = (c(\pi), c(\chi))$. Let $\Pi'$
      be an automorphic restriction (in the sense of Section
      \ref{sec:restriction}) of $\Pi$ to $\Spb_4$. Then $\Pi'$ is a discrete
      automorphic representation for $\Spb_4$, and Arthur associates a
			discrete parameter $\psi' \in \Psi_{\disc}(\Spb_4)$ to $\Pi'$ (see
      Theorem~\ref{thm: Arthur for chi a square}). Now $\bigwedge^2(c(\pi))
      \otimes c(\chi)^{-1} = 1 \oplus c(\psi')$ (see the commutative diagram
      \eqref{eqn: commutative diagram Sp4 restriction}) and so $L^S(s,
      \bigwedge^2(\pi) \otimes \chi^{-1}) = \zeta_F^S(s) L^S(s, \psi')$.
      Moreover by~\cite[Thm.\ 5.3.1]{MR1937203}, $1 \oplus c(\psi')$ is
      associated to a (unique) isobaric sum of \emph{unitary} cuspidal
      representations, and by~\cite[Thm.\ 4.4]{MR623137} the same holds for
      $\psi'$. This implies that $L^S(s, \psi')$ does not vanish on the line
      $\Re(s)=1$, by the main result of~\cite{MR0432596}.
			\qedhere
	\end{enumerate}
\end{proof}

\begin{rem} \label{rem: explicit list of parameters following Arthur}
  By Theorem~\ref{thm: symplectic orthogonal GL2} and
  Proposition~\ref{prop:sympl_orth_alt}, we see that $\Psi_{\disc}(\GSpinb_5)$
  is the subset of $\Psi(\tGamma,\chi)$ consisting of pairs $(\psi, \chi)$ with
  $\psi$ of the following kinds. (We have labelled them in the same way as
  in~\cite{MR2058604}. The groups~$\cS_\psi$ are easy to compute; for the values
  of~$\epsilon_\psi$, see~\cite[(1.5.6)]{MR3135650}.)
  \begin{enumerate}[(a)]
    \item cuspidal automorphic representations $\pi$ of $\GLb_4$ such that
      $\pi^{\vee} \otimes \chi \simeq \pi$ and $L^S(s, \chi^{-1} \otimes
      \bigwedge^2 \pi)$ has a pole at $s=1$. (General type, $\cS_\psi=1$,
      $\epsilon_\psi=1$.)
    \item $\pi_1 \boxplus \pi_2$ where $\pi_i$ are cuspidal automorphic
      representations of $\GLb_2$, $\omega_{\pi_1} = \omega_{\pi_2} = \chi$ and
      $\pi_1 \not\simeq \pi_2$. (Yoshida type, $\cS_\psi=\Z/2\Z$,
      $\epsilon_\psi=1$.)
    \item $\pi[2]$ for $\pi$ a cuspidal automorphic representation for $\GLb_2$
      such that $\omega_{\pi} / \chi$ has order $2$ (i.e.\ $(\pi, \chi)$ is of
      orthogonal type, which means that $\pi$ is automorphically induced from a
      character $\eta : \A_E^\times/E^\times \rightarrow \C^{\times}$ for the
      quadratic extension $E/F$ corresponding to $\omega_{\pi} / \chi$, such
      that $\eta^c \neq \eta$ and $\eta|_{\A_F^\times/F^\times} = \chi$).
      (Soudry type, $\cS_\psi=1$, $\epsilon_\psi=1$.)
    \item $\pi \boxplus \eta[2]$ with $\pi$ cuspidal for $\GLb_2$ and
      $\omega_{\pi} = \eta^2 = \chi$. (Saito--Kurokawa type, $\cS_\psi=\Z/2\Z$,
      $\epsilon_\psi=\mathrm{sgn}$ if $\epsilon(1/2,\pi\otimes\eta^{-1})=-1$,
      and $\epsilon_\psi=1$ otherwise.)
    \item $\eta_1[2] \boxplus \eta_2[2]$ with $\eta_1^2 = \eta_2^2 = \chi$ and
      $\eta_1 \neq \eta_2$. (Howe--Piatetski-Shapiro type, $\cS_\psi=\Z/2\Z$,
      $\epsilon_\psi=1$.)
    \item $\eta[4]$ with $\eta^2 = \chi$. (One dimensional type, $\cS_\psi=1$,
      $\epsilon_\psi=1$.)
  \end{enumerate}
\end{rem}

\begin{prop} \label{prop: parameter is discrete symplectic}
	For $c \in FS(\GSpinb_5)$ associated to a discrete automorphic representation
	$\Pi$ of $\GSpinb_5$ with central character $\chi$, the associated element of
  $\Psi(\tGamma,\chi)$ (by Proposition \ref{prop: STTF to associate parameters})
  belongs to the subset $\Psi_{\disc}(\GSpinb_5)$.
\end{prop}
\begin{proof}
  As in the proof of Proposition~\ref{prop:sympl_orth_alt} \ref{item: sympl orth
  alt dim4 sympl case}, we use an automorphic restriction $\Pi'$ of $\Pi$ to
  $(\GSpinb_5)_{\der} \simeq \Spb_4$, and the associated parameter $\psi'$,
  which we know to be discrete. We know that $1 \oplus c(\psi') =
  \bigwedge^2(c(\psi)) \otimes c(\chi)^{-1}$.

  By Theorem~\ref{thm: Arthur for chi a square}, we can and do assume
  that~$\chi$ is not a square. In particular, this implies that~$\psi$ does not
  have a summand of the form~$\eta$, $\eta[2]$ or~$\eta[4]$ (as the condition
  that~$\eta$ is ~$\chi$-self dual forces $\eta^2=\chi$). In addition, if $\psi
  = \psi_1 \boxplus \psi_1$, then $c(\psi') = \left( \bigwedge^2(c(\psi_1))
  \otimes c(\chi)^{-1} \right)^{\oplus 2} \oplus \ad^0(c(\psi_1))$, which
  contradicts the discreteness of $\psi'$. Thus we have the following
  possibilities for~$\psi$.
  \begin{enumerate}
    \item $\psi=\psi_1 \boxplus \psi_2$ where $\psi_i$ is a cuspidal automorphic
      representation for $\GLb_2$ such that $\psi_i^{\vee} \otimes \chi \simeq
      \psi_i$ and $\psi_1 \not\simeq \psi_2$. We need to show that
      $\omega_{\pi_i} = \chi$, i.e.\ that $(\pi_i, \chi)$ is of symplectic type.
      Suppose not. We have $\omega_{\pi_i}^2 = \chi^2$, and by Remark \ref{rem:
      getting alpha from parameter} we also have~$\omega_{\pi_1} \omega_{\pi_2}
      = \chi^2$ and so $\omega_{\pi_1} = \omega_{\pi_2}$. Then we find that
      $\bigwedge^2(\psi) \otimes \chi^{-1} = (\omega_{\pi_1}/\chi) \boxplus
      (\omega_{\pi_2}/\chi) \boxplus (\chi^{-1} \pi_1 \boxtimes \pi_2)$. Since
      $\omega_{\pi_1}/\chi=\omega_{\pi_2}/\chi$ is a non-trivial quadratic
      character, this cannot be written as~$1\boxplus\psi'$ with~$\psi'$
      discrete, a contradiction.
    \item $\psi=\pi[2]$, where $\pi$ is a cuspidal automorphic representation
      for $\GLb_2$ such that $\pi^{\vee} \otimes \chi \simeq \pi$. In this case
      we need to check that $\omega_{\pi} / \chi$ has order $2$, i.e.\ is
      non-trivial. But if~$\chi=\omega_\pi$ then $\psi'=\bigwedge^2(\pi[2])
      \otimes \omega_{\pi}^{-1} = \ad^0(\pi) \boxplus [3]$, which cannot be
      written as an isobaric sum of $1$ and discrete automorphic representations
      for general linear groups, a contradiction.
    \item $\psi = \pi[1]$ where $\pi$ is a cuspidal automorphic representation
      for $\GLb_4$ such that $\pi^{\vee} \otimes \chi \simeq \pi$. In this case
      we need to check that $(\pi, \chi)$ is of symplectic type, i.e.\ that
      $L^S(s, \bigwedge^2(\pi) \otimes \chi^{-1})$ has a pole at $s=1$. Exactly
      as in the proof of Proposition~\ref{prop:sympl_orth_alt} ~\eqref{item:
      sympl orth alt dim4 sympl case}, we have $L^S(s, \bigwedge^2(\pi) \otimes
      \chi^{-1}) = \zeta_F^S(s) L^S(s, \psi')$, and $L^S(s, \psi')$ does not
      vanish on the line $\Re(s)=1$, as required.
      \qedhere
		\end{enumerate}
\end{proof}

\begin{prop} \label{prop: parameter is discrete orthogonal}\leavevmode
	\begin{enumerate}
    \item For $c \in FS(\GSpinb_4)$ associated to a discrete automorphic
      representation $\pi$ for $\GSpinb_4$ having central character $\chi$, the
      element ${}^L \xi(c) \in \Psi(\tGamma,\chi)$ associated to~$c$ by
      Proposition \ref{prop: STTF to associate parameters} belongs to
      $\widetilde{\Psi}_{\disc}(\GSpinb_4, \chi)$.
    \item For $c \in FS(\GSpinb_4)$, occurring in $S_{\disc}^{\GSpinb_4}$,
      such that $\muhat(c) = c(\chi)$ and such that ${}^L \xi(c)$ is discrete,
      we have that ${}^L \xi(c) \in \widetilde{\Psi}_{\disc}(\GSpinb_4, \chi)$.
	\end{enumerate}
\end{prop}
\begin{proof}
	\begin{enumerate}
    \item By Theorem~\ref{thm: Arthur for chi a square}, we can and do assume
      that~$\chi$ is not a square. As explained in the proof of
      Proposition~\ref{prop: STTF to associate parameters}~(\ref{item: param for
      GSpin4 Levi}), the parameter of~$\pi$ is of the
      form~$\pi_1\boxtimes\pi_2$, where~$\pi_1,\pi_2$ are discrete automorphic
      representations of~$\GLb_2$ with~$\omega_{\pi_1}\omega_{\pi_2}=\chi$. If
      neither~$\pi_1,\pi_2$ were cuspidal, then~$\chi$ would be a square, so we
      may assume that~$\pi_1$ is cuspidal. If~$\pi_2=\eta[2]$ then
      $\pi_1\boxtimes\pi_2=\pi_3[2]$ where~$\pi_3=\eta\otimes\pi_1$,
      so~$\omega_{\pi_3}=\chi$, and it follows from Theorem~\ref{thm: symplectic
      orthogonal GL2} that this parameter belongs
      to~$\widetilde{\Psi}_{\disc}(\GSpinb_4, \chi)$.

      It remains to consider the case that $\pi_1,\pi_2$ are both cuspidal. If
      $\pi_1\boxtimes\pi_2$ is cuspidal, then the parameter belongs
      to~$\widetilde{\Psi}_{\disc}(\GSpinb_4, \chi)$ by
      Proposition~\ref{prop:sympl_orth_alt}. If $\pi_1\boxtimes\pi_2$ is not
      cuspidal, then since $\omega_{\pi_1}\omega_{\pi_2}=\chi$ is not a square,
      $\pi_1$ cannot be a twist of~$\pi_2$, and it follows from Theorem~A of the
      appendix to~\cite{MR2899809} that~$\pi_1,\pi_2$ are both automorphic
      inductions of characters from a common quadratic extension~$E/F$. In this
      case~$\pi_1\boxtimes\pi_2$ is the isobaric direct sum $\pi_3\boxplus\pi_4$
      where $\pi_3,\pi_4$ are distinct  automorphic inductions of characters
      from ~$E/F$ (see~\cite[(A.2.2)]{MR2899809}), so it again follows from
      Theorem~\ref{thm: symplectic orthogonal GL2} that this parameter belongs
      to~$\widetilde{\Psi}_{\disc}(\GSpinb_4, \chi)$, as required.

    \item By Proposition \ref{prop: STTF to associate parameters} (4) and the
      stabilisation of the trace formula for $\GSpinb_4$, either $c$ is
      associated to a discrete automorphic representation for $\GSpinb_4$, or
      there exists $\alpha \in F^{\times} / F^{\times, 2} \smallsetminus \{1\}$
      and $c' \in FS(\bH^{\alpha}_2)$ occurring in $S_{\disc}^{\bH^{\alpha}_2} =
      I_{\disc}^{\bH^{\alpha}_2}$ such that $c = {}^L \xi'(c')$. In the first
      case we conclude by the previous point, so we are left to consider the
      second case. Denote $E = F(\sqrt{\alpha})$ and $\Gal(E/F) = \{ 1, \sigma
      \}$. By the description of $\bH_2^{\alpha}$ in Section~ \ref{subsec:
      elliptic endoscopic for GL4 times GL1} we obtain that $c = c(\pi_1) \oplus
      c(\pi_2)$ where each $\pi_i$ is a cuspidal automorphic representation
      automorphically induced for $E/F$ from a character $\chi_i$ of
      $\A_E^{\times} / E^{\times}$ such that $\chi_i|_{\A_F^{\times}} = \chi$
      and $\chi_1, \chi_1^{\sigma}, \chi_2, \chi_2^{\sigma}$ are pairwise
      distinct, and again we conclude by Theorem \ref{thm: symplectic orthogonal
      GL2}.
      \qedhere
	\end{enumerate}
\end{proof}

\section{Multiplicity formula}\label{sec: multiplicity formula}

In this section we prove the multiplicity theorem for~$\GSpinb_5$
(Theorem~\ref{thm:glob_mult_form_GSpin5}), which describes the
discrete automorphic spectrum in terms of the packets
$\Pi_{\psi}(\epsilon_{\psi})$ defined in Definition~\ref{defn: Pi psi
  epsilon}. We begin with some preliminaries.

\subsection{Canonical global normalisation versus Whittaker normalisation}

Recall from Remark \ref{rem: Idisc for tGL} that for $\bG = \GLb_N \times
\GLb_1$ and $\tbG = \bG \rtimes \theta$, for a Levi subgroup $\bL$ of $\bG$ and
$\pi_{\bL} \in \Pi_{\disc}(\bL)$ the parabolically induced representation
$\cA^2(\bU_{\bP}(\A_F) \bL(F) \backslash \bG(\A_F))_{\pi_{\bL}}$ is irreducible.
For $\tilde{w} \in W(\bL, \tbG)$ we have a canonical (``automorphic'') extension
of this representation of $\bG(\A_F)$ to $\tbG$, denoted $M_{\bP |
\tilde{w}(\bP)}(0) \circ \rho_{\bP, \tilde{w}}^{\bG}$ in Section \ref{sec:
STTF}. We have another canonical normalisation of this extension, namely the
Whittaker normalisation recalled in Section \ref{subsec: Whittaker
normalisation}.

\begin{lem}[Arthur] \label{lem: can_vs_Whi_norm}
  These two extensions coincide.
\end{lem}
\begin{proof}
  The proof of \cite[Lemma 4.2.3]{MR3135650} readily extends to the case at
  hand.
\end{proof}

\subsection{The twisted endoscopic character relation for real discrete tempered
parameters}

\begin{prop}\label{prop: twisted endoscopic character real discrete tempered}
  Let $\varphi : W_\R \rightarrow \GSp_4$ be a discrete parameter. Then the
  twisted endoscopic character relation holds for $\Pi_{\varphi}$ \emph{(}as
  defined by Langlands in \cite{langlandsrg}\emph{)}, i.e.\ part
  \ref{item:thm_local_Apackets1} of Theorem \ref{thm:local_Arthur_packets}
  holds.
\end{prop}
Recall that for $\varphi$ such that $\muhat \circ \varphi$ is a square, this
twisted endoscopic character relation is a direct consequence of \cite{Mezo} and
\cite[Annexe C]{AMR}.
\begin{proof}
  We use a global argument similar to (but simpler than) \cite[Annexe C]{AMR}.
  Up to twisting we can assume that $\Std_{\GSpinb_5} \circ \varphi \simeq
  (I_{a_1} \oplus I_{a_2}, \sign^{2a_1})$, where $a_1, a_2 \in \frac{1}{2}
  \Z_{>0}$ are such that $a_1-a_2 \in \Z_{>0}$ (and as before, $I_a =
  \Ind_{\C^{\times}}^{W_\R} (z \mapsto (z/\bar{z})^a)$). Fix a continuous
  character $\chi : \A^{\times} / \R_{>0} \Q^{\times} \rightarrow \C^{\times}$
  such that $\chi|_{\R^{\times}} = \sign^{2a_1}$. There are cuspidal automorphic
  representations $\pi_1, \pi_2$ for $\GLb_2 / \Q$ with central characters
  $\omega_{\pi_1} = \omega_{\pi_2} = \chi$ and such that $\rec(\pi_{i, \infty})
  = I_{a_i}$ (apply \cite[Proposition 4]{Serre_asymp} with $n=1$, $k=2a_i+1$
  fixed and $N$ of the form $\ell \mathrm{cond}(\chi)$ where
  $\mathrm{cond}(\chi)$ is the conductor of $\chi$ and $\ell \rightarrow +
  \infty$ prime). Let $\psi = \pi_1 \boxplus \pi_2 \in \Psi_{\disc}(\GSpinb_5,
  \chi)$, so that $\psi_{\infty} = \varphi$.

  By \cite{Mezo} there is $z(\varphi) \in \C^{\times}$ such that for any
  $f_{\infty} \in I(\tGamma_{\R})$ we have
  \[ \tr \pi_{\varphi}^{\tGamma}(f_{\infty}) = z(\varphi) \left( \tr
  \pi_{\infty}^+ (f') + \tr \pi_{\infty}^-(f_{\infty}') \right) \]
  where $\pi_{\infty}^+$ (resp.\ $\pi_{\infty}^-$) is the generic (resp.\
  non-generic) element of $\Pi_{\varphi}$, i.e.\ $\langle \cdot, \pi_{\infty}^+
  \rangle$ (resp.\ $\langle \cdot, \pi_{\infty}^- \rangle$) is the trivial
  (resp.\ non-trivial) character of $\cS_{\varphi}$. We need to show that
  $z(\varphi) = 1$. Recall that for any finite prime $p$ the twisted endoscopic
  character relation
  \[ \tr \pi_{\psi_p}^{\tGamma}(f_p) = \sum_{\pi_p \in \Pi_{\psi_p}} \tr
  \pi_p(f'_p) \]
  holds by the main theorem of \cite{MR3267112}.

  In the discrete part of the trace formula for $\tGamma$, the contribution
  $I_{\disc, c(\psi)}^{\tGamma}$ of $c(\psi)$ only comes from $\bL = \GLb_2
  \times \GLb_2$ and $\tilde{w} = \theta_0$, using notation as in the discussion
  preceding Definition \ref{defn: IC FS}. By Lemma \ref{lem: can_vs_Whi_norm}
  and since $\det(\tilde{w} - 1 | \mathfrak{A}_{\bL}^{\Gammab}) = 2$ this
  contribution is (on $I(\tGamma_S)$ for $S$ containing $\infty$ and all places
  where $\pi_1$ or $\pi_2$ ramify)
  \[ \prod_{v \in S} h_v \longmapsto \frac{1}{2} \prod_v \tr
    \pi_{\psi_v}^{\tGamma}(h_v) \]
  where $\pi_{\psi_v}^{\tGamma}$ is the Whittaker-normalised extension to
  $\tGamma(F_v)$ of the irreducible parabolically induced representation
  $\pi_{1,v} \times \pi_{2,v}$. Thus we get for $h = \prod_{v \in S} h_v \in
  I(\tGamma_S)$
  \addtocounter{subsubsection}{1}\begin{equation} \label{eqn: STTF_Yoshida}
    I_{\disc,c(\psi)}^{\tGamma}(h) = \frac{z(\varphi)}{2} \prod_{v \in S}
    \sum_{\pi_v \in \Pi_{\psi_v}} \tr \pi_v(h_v^{\GSpinb_5}).
  \end{equation}
  By the stabilisation of the twisted trace formula (Theorem
  \ref{thm:STTF}) and Remark \ref{rem: getting alpha from parameter} and
  Proposition \ref{prop: parameter is discrete orthogonal}~(2) which imply
  that the endoscopic groups $\GSpinb_4^{\alpha}$ and $\bR^{\alpha}$ have
  vanishing contributions corresponding to $c(\psi)^S$, \eqref{eqn:
  STTF_Yoshida} equals
  \[ S_{\disc, \nu(\varphi), c(\psi)^S}(h^{\GSpinb_5}). \]
  By surjectivity of the transfer map $h \mapsto h^{\GSpinb_5}$ (Proposition
  \ref{prop: surj transfer}), this determines the stable linear form $S_{\disc,
  \nu(\psi), c(\psi)^S}^{\GSpinb_5}$. Let
  \[ \bH = \left( \GLb_2 \times \GLb_2 \right) / \{ (zI_2, z^{-1} I_2 \,|\, z
  \in \GLb_1 \} \]
  be the unique non-trivial elliptic endoscopic group for $\GSpinb_5$. The
  $(\nu(\psi), c(\psi)^S)$-part of the stabilisation of the trace formula
  (Theorem \ref{thm: STF}) for $\GSpinb_5$ now reads, for $f = \prod_{v \in S}
  f_v \in I(\GSpinb_5)$,
  \[ I_{\disc, \nu(\psi), c(\psi)^S}^{\GSpinb_5}(f) = \frac{z(\varphi)}{2}
    \prod_{v \in S} \sum_{\pi_v \in \Pi_{\psi_v}} \tr \pi_v(f_v) + \frac{1}{4}
    \sum_{\substack{\nu' \mapsto \nu(\psi) \\ c'^S \mapsto c(\psi)^S}} S_{\disc,
    \nu', c'^S}^{\bH}(f^{\bH}).\]
  Now $S_{\disc, \nu', c'^S}^{\bH} = I_{\disc, \nu', c'^S}^{\bH}$ is
  non-vanishing if and only if $(\nu', c'^S)$ is associated to $(\pi_1, \pi_2)$
  or to $(\pi_2, \pi_1)$, in which case it equals $\tr \left( \pi_1 \otimes
  \pi_2 \right)$ or $\tr \left( \pi_2 \otimes \pi_1 \right)$. By the endoscopic
  character relations, in either case we have
  \[ S_{\disc, \nu', c'^S}^{\bH}(f^\bH) = \prod_{v \in S} \sum_{\pi_v \in
  \Pi_{\psi_v}} \langle s, \pi_v \rangle \tr \pi_v(f_v), \]
  where $s$ is the non-trivial element of $\cS_{\psi}$. Thus we
  obtain
  \[ I_{\disc, \nu(\psi), c(\psi)^S}^{\GSpinb_5}(f) = \sum_{(\pi_v)_v \in
    \prod_{v \in S} \Pi_{\psi_v}} \frac{z(\varphi) + \prod_{v \in S} \langle s,
  \pi_v \rangle}{2} \prod_{v \in S} \tr \pi_v(f_v). \]
  By Proposition \ref{prop: STTF to associate parameters}~(1) the left-hand side
  simply equals the trace of $f$ in the $(\nu(\psi), c(\psi)^S)$-part of the
  discrete automorphic spectrum for $\GSpinb_5$. Varying $S$, the above equality
  means that the multiplicity of $\pi = \otimes'_v \pi_v \in \Pi_{\psi}$ in
  $\cA^2(\GSpinb_5)$ equals $(z(\varphi)+ \langle s, \pi \rangle)/2$. Comparing
  with \cite[Theorem 3.1]{MR3267112} (which relies on the theta correspondence
  and not trace formulas) for any $\pi$ we finally obtain
  $z(\varphi) = 1$.
\end{proof}

\begin{rem}
  Arguing as in Lemma C.1 of \cite{AMR} one could certainly prove the
  Proposition without using \cite[Theorem 3.1]{MR3267112}, since $|z(\psi)| = 1$
  and $(z(\psi)-1)/2 \in \Z_{\geq 0}$ imply $z(\psi) = 1$ (consider the
  multiplicity of $\pi_{\infty}^- \otimes \bigotimes'_p \pi_p$ where $\langle s,
  \pi_p \rangle = +1$ for all $p$).
\end{rem}

\subsection{Local parameters}

In this section we obtain Arthur's multiplicity formula for $\GSpinb_5$, by
formally using the stable twisted trace formula and twisted endoscopic character
relations to get the desired expression for $S_{\disc, c}^{\GSpinb_5}$ for $c$
corresponding to $\psi \in \Psi_{\disc}(\GSpinb_5)$, and then the stable trace
formula for $\GSpinb_5$.

We begin with the following important point, which is Conjecture \ref{conj:
global parameters give local parameters GSpin} for $\bG = \GSpinb_5$.
\begin{prop} \label{prop: L parameter of symplectic is symplectic}
	If~$\pi$ is a $\chi$-self dual cuspidal automorphic representation
	of~$\GLb_4(\A_F)$ of symplectic type, then for any place~$v$ of~$F$, the pair
	$(\rec(\pi_v), \rec(\chi_v))$ is of symplectic type, i.e.\ factors through
  $\GSp_4(\C)$.
\end{prop}
\begin{proof}
  This follows from~\cite[Thm.\ 12.1]{MR2800725}, which shows
  that~$\pi$ arises as the transfer of a (globally generic)
	automorphic representation~$\Pi$ of~$\GSpb_4(\A_F)$, and that at each
	place~$v$, the pair $(\rec(\pi_v), \rec(\chi_v))$ is obtained from the
	$L$-parameter associated to~$\Pi_v$ by the main theorem of~\cite{MR2800725}.
\end{proof}

\begin{rem}
	There are at least two alternative ways of proving Proposition~\ref{prop: L
	parameter of symplectic is symplectic}. One is to use the main results
	of~\cite{MR1937203} and~\cite{MR2567395}, which imply in particular that for
	each place~$v$ the representation $\bigwedge^2\rec(\pi_v)\otimes
	\rec(\chi_v)^{-1}$ contains the trivial representation, together with a case
	by case analysis. The other is to follow the argument of~\cite[\S
	8.1]{MR3135650}.
\end{rem}

\subsection{The global multiplicity formula}

Given Proposition \ref{prop: parameter is discrete symplectic}, the multiplicity
formula is morally equivalent to the following formula for any $\psi \in
\Psi_{\disc}(\GSpinb_5)$, $f \in \CH(\GSpinb_5)$ and $S$ large enough:
\[ S_{\disc, \nu, c(\psi)^S}^{\GSpinb_5} = \begin{cases}
    \frac{\epsilon_{\psi}(s_{\psi})}{|\CS_{\psi}|} \sum_{\pi \in \Pi_{\psi}}
    \langle s_{\psi}, \pi \rangle \tr \pi & \text{ if } \nu = \nu(\psi) \\
    0 & \text{ otherwise.} \end{cases} \]
This is the simplification (for discrete parameters) of the general stable
multiplicity formula (see \cite[Theorem 4.1.2]{MR3135650}).

We now prove the multiplicity formula; the following theorem is Conjecture
\ref{conj: global multiplicity formula GSpin}, specialised to the case $\bG =
\GSpinb_5$. We write~$\Pi_\psi(\epsilon_\psi)$ for the set of representations
defined in~\ref{defn: Pi psi epsilon} (with no tilde, since we are working
with~$\GSpinb_5$).
\begin{thm} \label{thm:glob_mult_form_GSpin5}
	There is an isomorphism of $\cH(\GSpinb_5)$-modules 
	\numequation\label{eqn: description of discrete spectrum of G}
		\CA^2(\GSpinb_5) \cong
		\bigoplus_{\substack{\chi:\A_F^\times/F^\times\R_{>0}\to\C^\times \\
				\psi\in\Psi_{\disc}(\GSpinb_5,\chi) \\
				\pi \in {\Pi}_\psi(\epsilon_\psi)}} \pi
	\end{equation}
  where~$\chi$ runs over the continuous (automatically unitary) characters.
\end{thm}
\begin{proof}
  Fix a continuous character $\chi:\A_F^\times/F^\times\R_{>0}\to\C^\times$, and
  write
	\[\CA^2(\GSpinb_5,\chi)\]
	for the space of $\chi$-equivariant square-integrable automorphic forms on
	which $\A_F^\times/F^\times$ acts via~$\chi$. For any~$\nu\in IC(\Gb)$
	and~$c\in FS(\Gb)$, write
	\[\CA^2(\GSpinb_5,\chi)_{\nu,c}:=\varinjlim_S\CA^2(\GSpinb_5,\chi)_{\nu,c^s}.
	\]
	Then we have
	\begin{align*}
		\CA^2(\GSpinb_5,\chi) = & \bigoplus_{\substack{\nu\in IC(\Gb)\\ c\in
			FS(\Gb)}}\CA^2(\GSpinb_5,\chi)_{\nu,c} \\
		= & \bigoplus_{\substack{\nu\in IC(\Gb) \\ \psi\in \Psi_{\disc}(\GSpinb_5,
			\chi)}}\CA^2(\GSpinb_5)_{\nu,c(\psi)}.
	\end{align*}
	Indeed, it follows from Proposition \ref{prop: parameter is discrete
	symplectic} that for any~$c$ with $\CA^2(\GSpinb_5,\chi)_c\ne 0$, there is
	some $\psi \in \Psi_{\disc}(\GSpinb_5, \chi)$ such ${}^L \xi(c(\pi)) =
	c(\psi)$. It follows that we are reduced to showing that for each~$\psi \in
	\Psi_{\disc}(\GSpinb_5, \chi)$, we have
	\numequation \label{eqn: packet sume to be proved}
		\CA^2(\GSpinb_5)_{\nu,c(\psi)}\cong
		\begin{cases}
			\bigoplus_{\pi \in {\Pi}_\psi(\epsilon_\psi)}\pi & \text{ if } \nu =
				\nu(\psi) \\
			0 & \text{ if } \nu \ne \nu(\psi).
		\end{cases}
	\end{equation}

	Fix $\nu\in IC(\Gb)$ and $\psi \in \Psi_{\disc}(\GSpinb_5, \chi)$. If $\chi$
	is a square, then we are done by Theorem~\ref{thm: Arthur for chi a square}
	(that is, by reducing to $\SOb_5$, already proved by Arthur). So we only have
	to consider the following cases:
	\begin{enumerate}
		\item Cuspidal $\pi$ for $\GLb_4$ such that $\pi^{\vee} \otimes \chi \simeq
			\pi$ and $(\pi, \chi)$ is of symplectic type.
		\item $\pi_1 \boxplus \pi_2$ where the $\pi_i$'s are distinct cuspidal
			automorphic representations for $\GLb_2$ with $\omega_{\pi_i} = \chi$
			(Yoshida type). 
		\item $\pi[2]$ where $\pi$ is a cuspidal automorphic
			representation for $\GLb_2$ such that $\omega_{\pi} / \chi$ is a quadratic
			character, i.e.\ $\pi^{\vee} \otimes \chi \simeq \pi$ and $(\pi, \chi)$ is
			of orthogonal type (Soudry type).
	\end{enumerate}
	In case~(2), the multiplicity formula is a special case of \cite[Theorem
	3.1]{MR3267112}, proved using the global theta correspondence. So we can and
	do assume that we are in case~(1) or case~(3), so that in
	particular~$\cS_{\psi}=1$ and~$\epsilon_\psi=1$. Furthermore, in either case
	we know that for any place~$v$, the parameter~$\psi_v$ is of symplectic type,
	i.e.\ factors through~$\GSp_4$ (in case~(1) this is Proposition~\ref{prop: L
	parameter of symplectic is symplectic}, and in case~(3) it follows from
	Theorem~\ref{thm: symplectic orthogonal GL2}).
                        
	We will prove~\eqref{eqn: packet sume to be proved} by
	computing~$I_{\disc,\nu,c(\psi)}^{\GSpinb_5, \GSpinb_5}(f)$ for each~$f \in
  \CH(\GSpinb_5)$, which by definition is the trace of~$f$ on the left hand side
  of~\eqref{eqn: packet sume to be proved} (note that this is well-defined, and
  equal to~$I_{\disc,\nu,c(\psi)^S}^{\GSpinb_5, \GSpinb_5}(f)$ for any
  sufficiently large~$S$). To this end, note firstly that by Proposition
  \ref{prop: STTF to associate parameters} (1), we know that for any proper Levi
  $\bL$ of $\GSpinb_5$, and for any $c \in FS(\GSpinb_5)$ occurring in
  $I_{\disc,\nu}^{\GSpinb_5, \bL}$, with central character $\chi$, we have ${}^L
  \xi (c) \in \Psi(\tGamma, \chi) \smallsetminus \Psi_{\disc}(\GSpinb_5, \chi)$.
	Consequently, we see that for any $\psi \in \Psi_{\disc}(\GSpinb_5, \chi)$, we
	have
	\numequation \label{eqn: I disc G equals I disc G G}
		I_{\disc,\nu,c(\psi)}^{\GSpinb_5}=I_{\disc,\nu,c(\psi)}^{\GSpinb_5,
		\GSpinb_5}.
	\end{equation}

	Denoting as usual the unique non-trivial elliptic endoscopic group of
	$\GSpinb_5$ by $\bH$, we have that $S_{\disc, \nu', c'}^{\bH}$ vanishes
	identically for any $\nu' \in IC(\bH)$ and any $c' \in FS(\bH)$ such that
	${}^L \xi'(c') = c(\psi)$ (because the proof of Proposition \ref{prop: STTF to
	associate parameters} (2) shows that any $c'$ occurring in
	$S_{\disc}^{\bH}$ is such that ${}^L \xi \circ {}^L \xi'(c')$ is a sum of at
	least two discrete automorphic representations of general linear groups). It
	follows that we have
	\numequation\label{eqn: I disc G equals S disc G}
    I_{\disc,\nu,c(\psi)}^{\GSpinb_5} = S_{\disc,\nu,c(\psi)}^{\GSpinb_5}.
	\end{equation}

	By Proposition \ref{prop: parameter is discrete orthogonal} (2), for any $c'$
  occurring in $S_{\disc}^{\GSpinb_4}$ we have ${}^L \xi(c') \neq c(\psi)$, so
  that (using also Remark \ref{rem: getting alpha from parameter}) the
  contribution of $\psi$ to the stabilisation of the twisted trace formula for
  $\tGamma$ simply reads
  \numequation \label{eqn: STTF_nonendo}
    I_{\disc, \nu, c(\psi)}^{\tGamma}(h) = S_{\disc, \nu,
    c(\psi)}^{\GSpinb_5}(h^{\GSpinb_5})
  \end{equation}
  where on the right-hand side $c(\psi)$ denotes the unique element of
  $FS(\GSpinb_5)$ which is the preimage of $c(\psi) \in FS(\tGamma)$ by ${}^L
  \xi$, and similarly for $\nu$ seen as an element of $ IC(\GSpinb_5)$. By
  surjectivity of $h \mapsto h^{\GSpinb_5}$ (see Proposition \ref{prop: surj
  transfer}), and Remark \ref{rem: Idisc for tGL}, this implies that $S_{\disc,
  \nu, c(\psi)}^{\GSpinb_5}$ vanishes identically if $\nu \neq \nu(\psi)$. In
  the definition of $I_{\disc, \nu, c(\psi)}^{\tGamma}$ as a sum over Levi
  subgroups, the only non-vanishing summand corresponds to $\bL = \GLb_4$. By
  Lemma \ref{lem: can_vs_Whi_norm} we have for $h = \prod_v h_v \in I(\tGamma)$
  \[ I_{\disc, \nu(\psi), c(\psi)}^{\tGamma}(h) = \prod_v \tr
  \pi_{\psi_v}^{\tGamma}(h_v). \]
	Applying
	Theorem~\ref{thm:local_Arthur_packets}~\eqref{item:thm_local_Apackets1} (or
  rather its extension to parameters in $\Psi^+(\GSpinb_5)$ via parabolic
  induction; see \cite[\S 1.5]{MR3135650}) to the right-hand side of this
  equality and using \eqref{eqn: STTF_nonendo} we obtain
  \[ S_{\disc, \nu(\psi), c(\psi)}^{\GSpinb_5}(\prod_v f_v) = \prod_v
  \sum_{\pi_v \in \Pi_{\psi_v}} \tr \pi_v (f_v). \]
  Combining this with~\eqref{eqn: I disc G equals I disc G G} and~\eqref{eqn: I
  disc G equals S disc G}, we conclude that
	\[ I_{\disc, \nu, c(\psi)}^{\GSpinb_5,\GSpinb_5}(\prod_v f_v) =
		\begin{cases}
      \prod_v \sum_{\pi_v \in \Pi_{\psi_v}} \tr \pi_v (f_v) & \text{ if } \nu =
        \nu(\psi) \\
			0 & \text{ if } \nu \ne \nu(\psi)
		\end{cases}
	\]
  Recalling that~$\cS_\psi=1$ and~$\varepsilon_\psi=1$, this is equivalent
  to~\ref{eqn: packet sume to be proved}, so we are done.
\end{proof}

\begin{rem}
  A consequence of the multiplicity formula and \cite{MR3227529} is that for any
  discrete automorphic representation $\pi$ for $\GSpinb_5$ which is formally
  tempered (i.e.\ of general or Yoshida type), there exists a \emph{globally
  generic} discrete automorphic representation $\pi'$ for $\GSpinb_5$ such that
  for any place $v$ of $F$, $\pi_v$ and $\pi'_v$ have the same Langlands
  parameter. Indeed letting $\psi \in \Psi_{\disc}(\GSpinb_5, \chi)$ be the
  parameter of $\pi$ (well-defined by the multiplicity formula), Shahidi's
  conjecture (proved in \cite{MR2800725}) implies that there is a unique
  representation in $\Pi_{\psi}$ which is generic at each place. In fact the
  multiplicity formula asserts that it is automorphic with multiplicity one. By
  (the converse part of) \cite[Theorem 4.26]{MR3227529} there exists a globally
  generic discrete (even cuspidal) automorphic representation $\pi'$ for
  $\GSpinb_5$ such that $\pi'_v \simeq \pi_v$ for almost all $v$. In particular
  $\pi'$ has parameter $\psi$, and for any place $v$ of $F$, $\pi'_v$ is generic.

  Note that in the case $\chi = 1$, Arthur used the the analogue of
  \cite{MR3227529} in order to prove Shahidi's conjecture: see
  \cite[Proposition 8.3.2]{MR3135650}. More precisely, he used the
  descent theorem of Ginzburg, Rallis and Soudry (and thus indirectly
  the converse theorem of Cogdell, Kim, Piatestski-Shapiro and
  Shahidi).
\end{rem}

\begin{rem}
  Let $\bG$ be an inner form of $\GSpinb_5$ over a number field $F$. Using the
  stabilisation of the trace formula for $\bG$ qualitatively (i.e.\ only
  considering families of Satake parameters), we see that for any $\pi \in
  \Pi_{\disc}(\bG, \chi)$, there is a well-defined $\psi \in \Psi(\tGamma,
  \chi)$ such that $c(\pi) = (c(\psi), c(\chi))$. Moreover if $\psi$ is discrete
  then $\psi \in \Psi_{\disc}(\GSpinb_5, \chi)$. If $\psi \in
  \Psi_{\disc}(\GSpinb_5, \chi)$ is tempered (i.e.\ either of general type or of
  Yoshida type) then using the stabilisation of the trace formula quantitatively
  and the endoscopic character relations proved in \cite{MR3267112} for inner
  forms as well, one could certainly prove the multiplicity formula for the part
  of the discrete automorphic spectrum for $\bG$ corresponding to $(c(\psi),
  c(\chi)) \in FS(\bG)$. The proof would be similar to those of
  Proposition~\ref{prop: twisted endoscopic character real discrete tempered}
  and Theorem \ref{thm:glob_mult_form_GSpin5}. Note however that to even state
  the multiplicity formula, one has to fix a normalisation of local transfer
  factors satisfying a product formula. This normalisation was achieved in
  \cite{Kalgri} and used in \cite{TaiMult} to prove the multiplicity formula for
  certain inner forms of symplectic groups. It would thus be necessary to
  compare Kaletha's normalisation of local transfer factors for the non-split
  inner form of $\GSpb_4$ realised as a rigid inner twist with Chan--Gan's ad
  hoc normalisation \cite[\S 4.3]{MR3267112}.
\end{rem}

\section{Compatibility of the local Langlands correspondences for
\texorpdfstring{$\Spb_4$}{Sp(4)} and \texorpdfstring{$\GSpinb_5$}{GSpin(5)}}
	\label{sec: local Langlands for Sp4}

In this section, we study the compatibility of the local Langlands
correspondence with restriction from~$\GSpb_4(F) \simeq \GSpinb_5(F)$
to~$\Spb_4(F)$, where~$F$ is a $p$-adic field. We do not consider the case of
Archimedean places, which could certainly be done by a careful examination of
the Langlands--Shelstad correspondence.

\subsection{Compatibility with restriction}

Let~$F$ be a $p$-adic field. The proof of the existence of the local
Langlands correspondence for~$\GSpb_4(F) \simeq \GSpinb_5(F)$
in~\cite{MR2800725} used the theta correspondence, and its
compatibility with the correspondence stated in~\cite{MR2058604} (characterised
by (twisted) endoscopic character relations) was proved in~\cite{MR3267112}. In
the paper~\cite{MR2673717}, a local Langlands correspondence for~$\Spb_4(F)$ was
deduced from the correspondence for~$\GSpb_4(F)$ by restriction. This
correspondence is uniquely characterised by the commutativity of the diagram
\numequation\label{eqn: diagram for restriction local Langlands}
\begin{tikzcd}[column sep=5em]
	\Pi(\GSpinb_5) \arrow[r] \arrow[d] & \Phi(\GSpinb_5) \arrow[d, "{\pr}"] \\
	\Pi(\Spb_4) \arrow[r] & \Phi(\Spb_4)
\end{tikzcd}
\end{equation}
where $\Pi(\GSpinb_5)$ (resp.\ $\Pi(\Spb_4)$) is the set of equivalence classes
of irreducible admissible representations of~$\GSpinb_5(F)$ (resp.\
$\Spb_4(F)$), $\Phi(\GSpinb_5)$ (resp.\ $\Phi(\Spb_4)$) is the set of
equivalence classes of continuous semisimple representations of~$\WD_F$ valued
in~$\GSp_4(\C)$ (resp.\ $\SO_5(\C)$), the horizontal arrows are the local
Langlands correspondences, and $\pr$ is the projection $\GSp_4(\C) \to
\PGSp_4(\C) \cong \SO_5(\C)$. The left hand vertical arrow is not in fact a map
at all, but a correspondence, given by taking any restriction of an element
of~$\Pi(\GSpinb_5)$ to~$\Spb_4(F)$.

Of course, \cite{MR3135650} gives another definition of the local Langlands
correspondence for~$\Spb_4$, which is characterised by twisted endoscopy for
$(\GLb_5, g \mapsto {}^t g^{-1})$. It is not obvious that this correspondence
agrees with that of~\cite{MR2673717}; this amounts to proving the
commutativity of the diagram~(\ref{eqn: diagram for restriction local
Langlands}), where now the horizontal arrows are the correspondences
characterised by twisted endoscopy. Proving this is the main point of this
section; we will also prove a refinement, describing the constituents of the
restrictions of representations of~$\GSpinb_5(F)$ to~$\Spb_4(F)$ in terms of
the parameterisation of $L$-packets.

We begin by recalling some results about restriction of admissible
representations, most of which go back to~\cite{MR644669}, and are explained in
the context of~$\GSpb_{2n}$ in~\cite{MR2673717}. They are also proved
in~\cite{MR3568940}, which we refer to as a self-contained reference. If~$\tpi$
is an irreducible admissible representation of~$\GSpinb_5(F)$,
then~$\tpi|_{\Spb_4(F)}$ is a direct sum of finitely many irreducible
representations of~$\Spb_4(F)$ (\cite[Lem.\ 6.1]{MR3568940}), and these
representations are pairwise non-isomorphic (\cite[Thm.\ 1.4]{MR2254643}).
Furthermore if~$\pi$ is an irreducible admissible representation of~$\Spb_4(F)$,
then there exists an irreducible representation~$\tpi$ of~$\GSpinb_5(F)$ whose
restriction to~$\Spb_4(F)$ contains~$\pi$, and~$\tpi$ is unique up to twisting
by characters (\cite[Cor.\ 6.3, 6.4]{MR3568940}).
There is also an analogue of these statements for~$L$-parameters, which is that
$L$-parameters for~$\Spb_4$ may be lifted to~$\GSpinb_5$, and such lifts are
unique up to twist; see~\cite[Prop.\ 2.8]{MR2673717} (see also
\cite[Th\'eor\`eme 7.1]{Labesse_lift} for a more general lifting result).

In particular, it follows that if~$\pi \in \Pi(\Spb_4)$, and~$\tpi$ lifts~$\pi$,
with $L$-parameter ~$\varphi_{\tpi}$, then~$\pr\circ\varphi_{\tpi}$ depends
only on~$\pi$ (because~$\varphi_{\tpi}$ is well-defined up to twist, as~$\tpi$
itself is); we need to show that it is equal to the $L$-parameter of~$\pi$
defined by the local Langlands correspondence of~\cite{MR3135650}.

\begin{thm}
  \label{thm: local Langlands for Sp4 is well defined}
	The diagram~\emph{(\ref{eqn: diagram for restriction local Langlands})}
	commutes, where the horizontal arrows are the correspondences
	of~\cite{MR3135650,MR2058604} determined by twisted endoscopy; that is, the
	local Langlands correspondences for~$\Spb_4$ of~\emph{\cite{MR2673717}}
	and~\emph{\cite{MR3135650}} coincide.
\end{thm}
\begin{proof}
	By the preceding discussion, it suffices to show that for each irreducible
	admissible representation~$\pi$, there is some lift~$\tpi$ of~$\pi$ such that
	$\varphi_\pi=\pr\circ\varphi_{\tpi}$. 

  We begin with the case that~$\pi$ is a discrete series representation. Then
  by~\cite[Thm.\ 1B]{MR818353} and Krasner's lemma, we can find a totally real
  number field $K$, a finite place~$v$ of~$K$, and a discrete automorphic
  representation~$\Pi$ of~$\Spb_4(\A_K)$, such that:
  \begin{enumerate}
		\item $K_v \cong F$ (so we identify~$K_v$ with~$F$ from now on).
		\item $\Pi_v \simeq \pi$.
		\item at each infinite place~$w$ of~$K$, $\Pi_w$ is a discrete series
			representation.
    \item \label{item: globalise_Sp4_4}
      for some finite place $w$ of $K$, $\Pi_w$ is a discrete series
      representation whose parameter is irreducible when composed with
      $\Std_{\Spb_4} : \SO_5 \rightarrow \GL_5$ (for example the parameter which
      is trivial on $W_{K_w}$ and the ``principal $\SL_2$'' on $\SU(2)$).
  \end{enumerate}

  By Theorem~\ref{thm: surjectivity automorphic res}, there is a  discrete
  automorphic representation~$\tPi$ of~$\GSpinb_5(\A_K)$ such
  that~$\tPi|_{\Spb_4(\A_K)}$ contains~$\Pi$. We can and do assume that the
  infinitesimal character of~$\Pi$ is sufficiently regular, so that in
  particular the parameter $\psi$ of $\Pi$ is tempered. By \eqref{item:
  globalise_Sp4_4} above, $\psi$ is just a self-dual representation for $\GLb_5
  / K$ with trivial central character. Write~$\tpsi$ for the parameter
  of~$\tPi$.

  As in the proof of Proposition \ref{prop:sympl_orth_alt} \ref{item: sympl orth
  alt dim4 sympl case} (i.e.\ comparing at the unramified places
  using~(\ref{eqn: commutative diagram Sp4 restriction})), we see that
  $1 \boxplus \psi = \bigwedge^2(\tpsi) \otimes \omega_{\tpsi}^{-1}$. Given the
  possibilities in Remark \ref{rem: explicit list of parameters following
  Arthur} we see (using \cite{MR533066} to rule out the case $\tpsi = \pi[2]$,
  see the proof of Proposition \ref{prop:sympl_orth_alt} (1)) that $\tpsi$ is
  necessarily tempered. If $\tpsi = \pi_1 \boxplus \pi_2$ was of Yoshida type
  then we would have $\psi = 1 \boxplus (\pi_1 \boxtimes \pi_2^{\vee})$, a
  contradiction. Therefore $\tpsi$ is of general type, i.e.\ a $\chi$-self-dual
  cuspidal automorphic representation for $\GLb_4 / K$ of symplectic type for
  some character $\chi$ of $\A_K^{\times} / K^{\times}$. By the main results
  of~\cite{MR1937203} and~\cite{MR2567395}, the Langlands parameter of $1
  \boxplus \psi$ at $v$ equals $\bigwedge^2(\rec(\tpsi_v)) \otimes
  \rec(\omega_{\tpsi})^{-1}$, which implies that
  $\varphi_{\Pi_v}=\pr\circ\varphi_{\tPi_v}$. Taking~$\tpi=\tPi_v$, we are done
  in this case.

  We now treat the case that the parameter~$\varphi_\pi$ is not discrete, but is
  bounded modulo centre. Recall that a minimal Levi subgroup ${}^L \bM$ of ${}^L
  \Spb_4$ such that $\varphi_\pi(\WD_F) \subset {}^L \bM$ is unique up to
  conjugation by $\Cent(\varphi_\pi, \widehat{\Spb_4})$ \cite[Proposition
  3.6]{MR546608}. Then $\varphi_\pi$ factors through a well-defined discrete
  parameter~$\varphi_\Mb : \WD_F \to {}^L \Mb$. Since $\Spb_4$ is quasi-split we
  have a natural identification of ${}^L \bM$ with the $L$-group of a Levi
  subgroup $\bM$ of $\GSpb_4$ (well-defined up to conjugation by normalisers in
  $\Spb_4$, resp.\ $\widehat{\Spb_4}$). Since $\varphi_\pi$ is assumed to be
  non-discrete we have ${}^L \bM \neq {}^L \Spb_4$. It follows from the
  construction in~\cite{MR3135650} (see the proof of Proposition 2.4.3 loc.\
  cit., in particular (2.4.13)) that $\pi$ is a constituent of the parabolic
  induction~$\Ind_{\bP(F)}^{\bG(F)} \pi_\Mb$, where~$\Pb$ is any parabolic
  subgroup of $\Spb_4$ with Levi~$\Mb$, and~$\pi_\Mb$ is in the $L$-packet
  of~$\varphi_\Mb$. Recall that this $L$-packet is defined via the natural
  identification $\Mb$ with a product of copies of $\GLb$ groups with
  $\Spb_{2a}$ for some $0 \leq a < 2$, using $\rec$ for the $\GLb$ factors and
  Arthur's local Langlands correspondence for the $\Spb$ factor.

  Write~$\Mb = \Mbt \cap \Spb_4$ where~$\Mbt$ is a Levi subgroup of~$\GSpb_4$,
  and similarly $\Pb = \Pbt \cap \Spb_4$. Let~$\widetilde{\pi_\Mb}$ be an
  essentially discrete series representation of~$\Mbt(F)$ whose restriction
  to~$\Mb(F)$ contains~$\pi_\Mb$. Then there is an irreducible
  constituent~$\tpi$ of the (semisimple) parabolic
  induction~$\Ind_{\Pbt(F)}^{\GSpinb_5(F)} \widetilde{\pi_\Mb}$ such that~$\pi$
  is a restriction of~$\tpi$. We will prove that $\varphi_\pi = \pr \circ
  \varphi_{\tpi}$. Note that for non-discrete parameters, the local Langlands
  correspondence for~$\GSpinb_5(F)$ of~\cite{MR2800725} is also compatible with
  parabolic induction (see \cite[\S 6.6]{MR3267112} and \cite[Prop.\
  13.1]{MR2846304}), i.e.\ the parameter of $\widetilde{\pi}$ is
  $\varphi_{\widetilde{\pi_{\bM}}}$ (the Langlands parameter of
  $\widetilde{\pi_{\bM}}$) composed with ${}^L \tbM \subset {}^L \GSpinb_5$.
  Note that $\tbM$ is isomorphic to a product of $\GLb$ and for such a group the
  (bijective) local Langlands correspondence is well-defined, i.e.\ it does not
  depend on the choice of an isomorphism. This follows from compatibility of
  $\rec$ with twisting, central characters and duality. The same argument shows
  that any morphism with normal image between two such groups is also compatible
  with the local Langlands correspondence. We have a commutative diagram
  \[ \begin{tikzcd}[column sep=5em]
    {}^L \tbM  \arrow[r] \arrow[d, "\mathrm{pr}", two heads] & {}^L \GSpinb_5
    \arrow[d, "\mathrm{pr}", two heads] \\
    {}^L \bM \arrow[r] & {}^L \Spb_4
  \end{tikzcd} \]
  so that to conclude that $\varphi_\pi = \pr \circ \varphi_{\widetilde{\pi}}$
  it is enough to show that $\varphi_\bM = \pr \circ
  \varphi_{\widetilde{\pi_\bM}}$, which is simply a compatibility of local
  Langlands correspondences for $\bM$ and $\tbM$. There are three cases to
  consider. We write the standard parabolic subgroups of $\GSpinb_5$ and
  $\Spb_4$ as in Section \ref{subsec:Levi_parametrisation}. We do not justify
  the embedding $\bM \hookrightarrow \tbM$, as this is a simple but tedious
  exercise in root data.
  \begin{itemize}
    \item $\tbM = \GLb_1^2 \times \GSpinb_1$, $\bM = \GLb_1^2$, the embedding
      $\bM \hookrightarrow \tbM$ is $(x_1, x_2) \mapsto (x_1 x_2, x_1/x_2,
      x_1^{-1})$. This case is trivial.
    \item $\tbM = \GLb_2 \times \GSpinb_1$, $\bM = \Spb_2 \times \GLb_1$, the
      embedding $\bM \hookrightarrow \tbM$ is $(g, x_1) \mapsto (g x_1,
      x_1^{-1})$. This case is not formal as for the factor $\Spb_2$ the local
			Langlands correspondence that is used is Arthur's from \cite{MR3135650}
			and it is not obvious that it is compatible with $\rec$ for $\GLb_2$, in
			other words that Arthur's local Langlands correspondence for $\Spb_2
			\simeq \SLb_2$ (characterised by twisted endoscopy for $\GLb_3$) coincides
			with Labesse-Langlands \cite{LabLan}. Fortunately Arthur verified this
			compatibility in \cite[Lemma 6.6.2]{MR3135650}.
    \item $\tbM = \GLb_1 \times \GSpinb_3$, $\bM = \GLb_1 \times \GLb_2$, the
      embedding $\bM \hookrightarrow \tbM$ is $g \mapsto (\det g, g / \det g)$
      where we have identified $\GSpinb_3$ with $\GLb_2$. This case also follows
      from the above remark about the local Langlands correspondence for groups
      isomorphic to a product of $\GLb$.
  \end{itemize}

  Finally, we must treat the case that~$\varphi$ is not bounded modulo centre.
  The description of the $L$-packets in this case is again in terms of parabolic
  inductions from Levi subgroups (``Langlands classification''). This is
  well-known and completely general (see \cite{Silberger_Langlandsquot},
  \cite{SilbergerZink_Langlandsparam}). The argument is similar to the above
  reduction, except that $\bP$ and $\tbP$ are uniquely determined by a
  positivity condition and that $\pi$ and $\tpi$ are unique quotients of
  standard modules and not arbitrary constituents. We do not repeat the
  argument.
\end{proof}

We now examine the restriction from~$\GSpinb_5(F)$ to~$\Spb_4(F)$ more
closely, proving a slight refinement of the results of~\cite{MR2673717}.
In~\cite[App.\ A]{MR2673717}, a detailed qualitative description of the
constituents of~$\tpi|_{\Spb_4(F)}$ is given, which is obtained by examining the
local Langlands correspondence (see~\cite[\S 5, 6]{MR2673717} for the
corresponding calculations with $L$-parameters). However, since the local
Langlands correspondence of~\cite{MR2673717} is not characterised in terms of
twisted and ordinary endoscopic character relations, they cannot describe
precisely which elements of the $L$-packets for~$\Spb_4(F)$ arise as the
restrictions of given elements of the $L$-packets for~$\GSpinb_5(F)$.

Theorem~\ref{thm: constituents of restriction from GSp4 to Sp4 in terms of
characters} below answers this question.  In its proof, we need to make use of
the results of Section~\ref{sec:restriction} for $\SOb_4 \hookrightarrow \Hb$
where
\[ \Hb = \left( \GLb_2 \times \GLb_2 \right) / \{ (zI_2, z^{-1} I_2 \,|\, z
	\in \GLb_1 \} \]
is the non-trivial elliptic endoscopic group of $\GSpinb_5$. Here $\SOb_4$ is
identified with the subgroup of pairs~$(a,b)$ with $(\det a)(\det b)=1$. Indeed,
$\Hb$ may be identified with the subgroup~$\GSOb_4$ of~$\GOb_4$ given by the
elements for which~$\det=\nu^2$, where~$\nu$ is the similitude factor. 

Note that $\SOb_4$ is an elliptic endoscopic group for $\Spb_4$ and that we have
the following commutative diagram:
\numequation\label{eqn: commutative diagram for SO4 Sp4 restriction}
\begin{tikzcd}[column sep=5em]
	\widehat{\Hb}  \arrow[r] \arrow[d, "{{}^L \xi'}"] & \widehat{\SOb_4} = \SO_4
	\arrow[d, "{{}^L \xi'}"] \\
	\widehat{\GSpinb_5} = \GSp_4 \arrow[r] & \widehat{\Spb_4} = \SO_5
\end{tikzcd}
\end{equation}

\subsection{Multiplicity one}

In studying restriction from~$\Hb$ to~$\SOb_4$ we will make use of the following
variant of the results of~\cite{MR2254643}. In fact, we could prove the special
case that we need in a simpler but more ad-hoc fashion by using the description
of~$\Hb$ in terms of~$\GLb_2$, but it seems worthwhile to prove this more
general result.

\begin{prop} \label{prop: multiplicity one for restriction for GSO2n}
	Let $n \geq 1$, and let  $V$ be a vector space of dimension $2n$ over $F$
	endowed with a non-degenerate quadratic form $q$. Let~$\pi$ be an irreducible
	admissible representation of~$\GSO(V, q) = \GSOb(V, q)(F)$. Then the
	irreducible constituents of the restriction~$\pi|_{\SO(V,q)}$ are pairwise
	non-isomorphic.
\end{prop}
\begin{proof}
	By~\cite[Theorem 2.3]{MR2254643}, it suffices to show that there is an
	algebraic anti-involution~$\tau$ of~$\GSOb(V,q)$ which preserves~$\SOb(V,q)$
	and takes each $\SO(V,q)$-conjugacy class in~$\GSO(V,q)$ to itself. To
	define~$\tau$, we set~$\tau(g)=\nu(g) \delta^n g^{-1} \delta^{-n}$ where
	$\delta \in \operatorname{O}(V,q)$ is an involution with $\det \delta = -1$.
	This obviously preserves~$\SOb(V,q)$, so we need only check that it also
	preserves $\SO(V,q)$-conjugacy classes in~$\GSO(V,q)$.

	To see this, we claim that it is enough to show that we can write~$g=xy$
	with~$x \in \operatorname{O}(V,q)$, $y \in \GO(V,q)$ (so $\nu(y)=\nu(g)$)
	satisfying $x^2=1$, $\det(x) = (-1)^n$, $y^2=\nu(y)$. Indeed, we then have
	\[\tau(g)=\nu(g) \delta^n g^{-1} \delta^{-n} = \delta^n \nu(y)y^{-1}x^{-1}
	\delta^{-n} = \delta^n yx \delta^{-n}
	=\delta^nx^{-1}(xy)x\delta^{-n}=(x\delta^{-n})^{-1}g(x\delta^{-n}), \]
	as required. The result then follows from Lemma~\ref{lem: product of
	involutions} below, which is a slight refinement of \cite[Thm.\
	A]{2016arXiv160706647R}.
\end{proof}

\begin{lem} \label{lem: product of involutions}
  Let $n \geq 0$, let $K$ be a field of characteristic $0$, and let $V$ be a
  vector space of dimension $2n$ over $K$ endowed with a non-degenerate
  quadratic form $q$. If $g \in \GSO(V,q)$ then we can write  $g=xy$
  with~$x\in\mathrm{O}(V,q)$, $y\in\GO(V,q)$ satisfying $x^2=1$,
  $\det(x)=(-1)^n$, $y^2 = \nu(y)$.
\end{lem}
\begin{proof}
	We argue by induction on $n$, the case $n=0$ being trivial. Suppose now that
	$n > 0$. By~\cite[Thm.\ A]{2016arXiv160706647R}, we can write $g=xy$
	with~$x\in\mathrm{O}(V,q)$, $y\in\GO(V,q)$ satisfying $x^2=1$,
	$y^2=\nu(y)=\nu(g)$. If $\det(x)=(-1)^n$ then we are done, so suppose that
	$\det(x)=(-1)^{n+1}$ and so~$\det(y)=(-1)^{n+1}\nu(y)^n$.

	Since~$y^2=\nu(y)$, any eigenvalue (in an extension of $K$) of~$y$ is a square
	root of~$\nu(y)$. Since~$\det(y)=(-1)^{n+1}\nu(y)^n$, we see that the two
	eigenspaces of~$y$ do not have equal dimension. It follows that~$\nu(y)$ is a
	square, as otherwise the characteristic polynomial of~$y$ would be a power
	of the irreducible polynomial~$X^2-\nu(y)$. So the eigenvalues of~$y$ are
	in $K$, and up to dividing $g$ and $y$ by one of these eigenvalues we can
	assume that $g \in \SO(V,q)$ and $y \in \mathrm{O}(V,q)$ with $\det(y) =
	(-1)^{n+1}$. Then~$y$ has an eigenspace (for an eigenvalue $\pm 1$) of
	dimension at least~$n+1$. The same analysis applies to~$x$, and it follows
	that there is a subspace~$W$ (the intersection of these eigenspaces for~$x$
	and~$y$) of dimension at least~$2$ of~$V$ on which $g$ acts by a scalar which
	is $\pm 1$. 

	Up to replacing $g$ by $-g$ and $y$ by $-y$, we can assume that $\ker(g-1)$
	has dimension at least~$2$. We have a canonical $g$-stable decomposition of
	$V$ as the direct sum of $\ker((g-1)^{2n})$ and its orthogonal complement, and
	they both have even dimension over $K$ since $g \in \SO(V,q)$ with $\dim_K V$
	even. If $g$ is not unipotent, we conclude using the induction hypothesis for
	the restriction of~$g$ to ~$\ker((g-1)^{2n})$ and to its orthogonal
	complement.

	Suppose for the rest of the proof that $g$ is unipotent. If $n=1$ the
	conclusion is trivial, so assume that $n > 1$, so that $\SOb(V,q)$ is
	semisimple. By Jacobson--Morozov (see for example~\cite[Ch.\ VIII \S
	11]{MR2109105}) there is an algebraic morphism $\SLb_2 \rightarrow \SOb(V,q)$
	mapping $\begin{pmatrix} 1 & 1 \\ 0 & 1 \end{pmatrix}$ to $g$, unique up to
	conjugation by the centraliser of $g$ in the subgroup
	$\Aut_e(\mathfrak{so}(V,q))$ of $\SO(V,q) / \{ \pm 1 \}$ where $\Aut_e$ is the
	subgroup of automorphisms of the Lie algebra generated by exponentials of
	nilpotent elements. For $d \geq 1$ fix an irreducible representation $U_d$ of
	$\SLb_2$ of dimension $d$ as well as a non-degenerate $(-1)^{d-1}$-symmetric
	$\SLb_2$-invariant pairing $B_d$ on $U_d$. We have a canonical decomposition
	\[ V = \bigoplus_{d \geq 1} U_d \otimes V_d \]
	where $V_d = (V \otimes_K U_d^*)^{\SLb_2}$. The quadratic form $q$ corresponds
	to an element of
	\[ (\Sym^2 V^*)^{\SLb_2} = \bigoplus_{d \geq 1 \text{ odd}} \Sym^2(V_d^*)
	\oplus \bigoplus_{d \geq 2 \text{ even}} \bigwedge^2 V_d^* \]
	and non-degeneracy of $q$ is equivalent to non-degeneracy of each factor.
	Writing each $V_d$ for $d$ odd (resp.\ even) as an orthogonal direct sum of
	quadratic lines (resp.\ planes endowed with a non-degenerate alternate form),
	we are left to prove a decomposition $g' = x'y'$ in the following cases.
	\begin{enumerate}
		\item $V'$ has odd dimension $2m+1$ and is endowed with a non-degenerate
			quadratic form $q'$ and a unipotent automorphism $g'$. Applying
			\cite[Thm.\ A]{2016arXiv160706647R} we obtain $g' = x'y'$ with $x',y'$
			involutions in $\mathrm{O}(V,q)$. Up to replacing $(x',y')$ by $(-x',-y')$
			we can assume that $\det(x')$ is $\pm 1$ as we may desire.
		\item $V' = U_{2m} \otimes V'''$ where $V'''$ is $2$-dimensional and endowed
			with a non-degenerate alternating form $B'''$, and $g' = g'' \otimes
			\Id_{V'''} \in \SO(V', q')$ for $q'$ the quadratic
			form corresponding to the symmetric bilinear form $B' = B_{2m} \otimes
			B'''$ and $g''$ a unipotent element of $\Sp(U_{2m}, B_{2m})$. Applying
			\cite[Thm.\ A]{2016arXiv160706647R} again we can write $g'' = x'' y''$
			where $x'', y''$ are involutions in $\GSp(U_{2m}, B_{2m})$ having
			similitude factor $-1$. Similarly write $\Id_{V'''} = x''' y'''$ where
			$x''', y'''$ are involutions in $\GSp(V''', B''')$ having similitude
			factor $-1$. Then $g' = (x'' \otimes x''') (y'' \otimes y''')$ is the
			desired decomposition as a product of involutions in $\SO(V', q')$.
		\qedhere
	\end{enumerate}
\end{proof}

\subsection{Restriction of local Arthur packets}

We now give our description of the restriction of representations
of~$\GSpinb_5(F)$. Recall that if~$\varphi:\WD_F\to \GSp_4$ is a bounded
parameter, then the corresponding component group~$\cS_\varphi$ is either
trivial or is~$\Z/2\Z = \{1,s\}$. In the former case, the
$L$-packet~$\Pi_\varphi$ associated to~$\varphi$ is a singleton, and in the
latter case it is a pair~$\{\pi^+,\pi^-\}$, where $\pi^{\pm}$ is characterised
by the fact that $\tr \pi^+-\tr\pi^-$ is the transfer to $\GSpinb_5(F)$ of $\tr
\pi_{\varphi_{\bH}}$ where $\varphi_{\bH} \in \Phi(\bH)$ is the parameter
mapping to $(\varphi, s)$ via ${}^L \xi'$. In either case, if we
write~$\varphi'=\pr\circ\varphi$, then by~\cite[Prop.\ 2.8]{MR2673717}, we have
\numequation \label{eqn: restriction of sums of L packets}
	\bigoplus_{\pi\in \Pi_\varphi} \pi|_{\Spb_4(F)} \cong
	\bigoplus_{\pi'\in\Pi_{\varphi'}} \pi'.
\end{equation}
(Indeed, this follows from Theorem~\ref{thm: local Langlands for Sp4 is well
defined}, the fact that lifts of representations of~$\Spb_4(F)$ to~$\GSpb_4(F)$
are unique up to twist, and the fact that the restrictions of representations
of~$\GSpb_4(F)$ to~$\Spb_4(F)$ are semisimple and multiplicity free.) The
following theorem improves on this result by giving a precise description of the
restrictions of the individual elements of~$\Pi_\varphi$.
\begin{thm}
	\label{thm: constituents of restriction from GSp4 to Sp4 in terms of
  characters}
	Let~$\varphi$ be a bounded $L$-parameter, and write~$\varphi' = \pr \circ
	\varphi$, so that $\cS_{\varphi}\into\cS_{\varphi'}$. Write~$\Pi_\varphi$
	and~$\Pi_{\varphi'}$ for the respective $L$-packets. If~$\cS_\varphi$ is
	trivial, and~$\Pi_\varphi=\{\pi\}$, then
	\[\pi|_{\Spb_4(F)}\cong\bigoplus_{\pi'\in\Pi_{\varphi'}}\pi'. \]
  If $\cS_\varphi=\Z/2\Z = \{1,s\}$, and~$\Pi_\varphi=\{\pi^+,\pi^-\}$ as above,
  then
	\[\pi^{\pm}|_{\Spb_4(F)} \cong \bigoplus_{\substack{\pi'\in\Pi_{\varphi'}\\
		\langle s,\pi'\rangle=\pm 1}} \pi'. \]
\end{thm}
\begin{proof}
	In the case that~$\cS_\varphi$ is trivial, this is~(\ref{eqn: restriction of
	sums of L packets}), so we may suppose that~$\cS_\varphi$ is non-trivial, so
	that~$\varphi$ is endoscopic. We can write~$\varphi=\varphi_1\oplus\varphi_2$
	where  $\varphi_1, \varphi_2 : \WD_F \rightarrow \GL_2$ are bounded with same
  determinant; that is, $\varphi  ={}^L \xi' \circ \varphi_{\bH}$,
  where~$\varphi_{\bH} = \varphi_1 \times \varphi_2 : \WD_F \times \SL_2(\C) \to
  \widehat{\Hb}$. Via ${}^L \xi'$ we can see $s$ as the non-trivial element of
  $Z(\widehat{\Hb}) / Z(\widehat{\GSpinb_5})$, i.e.\ the image of $(1,-1)
  \in \widehat{\Hb} \subset \GL_2 \times \GL_2$. Then by
  Conjecture~\ref{conj:local_Arthur_packets GSpin}~(2) for~$\GSpinb_5$ (i.e.\
  the main theorem of \cite{MR3267112}), we have an equality of traces

  \[  \tr	\pi^+(f)-\tr \pi^{-}(f) = \sum_{\pi_{\bH} \in \Pi_{\varphi_{\bH}}}
  \tr \pi_{\bH}(f^{\bH}). \]
	Applying Conjecture~\ref{conj:local_Arthur_packets GSpin}~(2) (or rather
	Theorem~\ref{thm: Arthur for chi a square}) for~$\Spb_4$, and
  writing~$\varphi'_{\bH}$ for the composite of~$\varphi_{\bH}$ and the natural
  map $\widehat{\Hb} \to \widehat{\SOb_4}$, we also have an equality of traces
  \[ \sum_{\substack{\pi' \in \Pi_{\varphi'}\\
      \langle s, \pi' \rangle=1}} \tr \pi'(f) -
      \sum_{\substack{\pi' \in \Pi_{\varphi'}\\
      \langle s, \pi' \rangle=-1}} \tr \pi'(f) =
      \sum_{\pi'_{\SOb_4} \in \Pi_{\varphi'_{\bH}}} \tr \pi'_{\SOb_4}(f'). \]
	The result now follows from~(\ref{eqn: restriction of sums of L packets}) and
	Theorem~\ref{thm: restriction from GSO4 to SO4} below.
\end{proof}

We end with a result on the restriction of representations from~$\Hb \simeq
\GSOb_4$ to~$\SOb_4$ that we used in the course of the proof of
Theorem~\ref{thm: constituents of restriction from GSp4 to Sp4 in terms of
characters}. The arguments are very similar to those for~$\GSpinb_5$, but are
rather simpler, as~$\Hb$ has no non-trivial elliptic endoscopic groups.
Since~$\Hb$ is isomorphic to the quotient of~$\GLb_2\times\GLb_2$ by a split
torus, the local Langlands correspondence for~$\Hb$, and the corresponding
endoscopic character identities, are easily deduced from those for~$\GLb_2$. The
correspondence and endoscopic character identities for~$\SOb_4$ are of course
proved in~\cite{MR3135650} (up to the outer automorphism $\delta$).

By Proposition~\ref{prop: multiplicity one for restriction for GSO2n}, if~$\pi$
is an irreducible admissible representation of~$\Hb(F)$, then $\pi|_{\SOb_4(F)}$
is a direct sum of representations occurring with multiplicity one. The proof
of~\cite[Lem.\ 2.6]{MR2673717} goes through unchanged and shows
that~$\pi_1|_{\SOb_4(F)}$, $\pi_2|_{\SOb_4(F)}$ have a common constituent if and
only if~$\pi_1, \pi_2$ differ by a twist by a character. By~\cite[Lem.\
2.7]{MR2673717}, the analogous statement is also true for~$L$-parameters: every
$L$-parameter $\varphi' : \WD_F \to \widehat{\SOb_4}(\C)$  arises from some
$\varphi : \WD_F \to \Hhat(\C)$, which is unique up to twist.

\begin{thm} 
  \label{thm: restriction from GSO4 to SO4}
	Let~$\varphi : \WD_F \to \Hhat(\C)$ be a bounded $L$-parameter, and
	let~$\varphi':\WD_F \to \widehat{\SOb_4}(\C)$ be the parameter obtained
	from~\emph{(\ref{eqn: commutative diagram for SO4 Sp4 restriction})}. Let
  $\pi$ be the tempered irreducible representation of $\bH$ associated to
  $\varphi$. Then
	\[ \pi|_{\widetilde{\cH}(\SOb_4(F))} \cong
		\bigoplus_{\pi'\in\Pi_{\varphi'}}\pi'.\]
\end{thm}
\begin{proof}
	By the preceding discussion, we need to show that for each bounded
  $L$-parameter $\varphi': \WD_F \to \widehat{\SOb_4}(\C)$ (up to outer
  conjugacy), and each $\pi'\in\Pi_{\varphi'}$, there is some~$\pi$
  lifting~$\pi'$ (or $\pi'^{\delta}$) whose $L$-parameter~$\varphi$
  lifts~$\varphi'$.

  Suppose firstly that~$\varphi'$ is discrete. As in the proof of
  Theorem~\ref{thm: local Langlands for Sp4 is well defined}, by Krasner's lemma
  and~\cite[Thm.\ 1B]{MR818353}, we can find a totally real number field $K$, a
  finite place~$v$ of~$K$, and a discrete automorphic representation~$\Pi'$
  of~$\SOb_4(\A_K)$, such that:
  \begin{itemize}
    \item $K_v\cong F$ (so we identify~$K_v$ with~$F$ from now on).
    \item $\Pi'_v=\pi'$.
    \item at each infinite place~$w$ of~$F$, $\Pi'_w$ is a discrete series
      representation.
  \end{itemize}

  By Theorem~\ref{thm: surjectivity automorphic res}, there is a  discrete
  automorphic representation~$\Pi$ of~$\Hb(\A_K)$ such
  that~$\Pi|_{\SOb_4(\A_K)}$ contains~$\Pi'$. Then~$\Pi$ corresponds to a
  pair~$\pi_1,\pi_2$ of discrete automorphic representations of~$\GLb_2(\A_K)$
  with equal central characters. The condition that $\Pi'_w$ is a discrete
  series representation at an infinite place $w$ of $K$ implies that $\pi_1$ and
  $\pi_2$ are cuspidal.

  We now consider the following commutative diagram of dual groups:
  \numequation\label{eqn: commutative diagram SO4 restriction}
  \begin{tikzcd}[column sep=5em]
    \widehat{\Hb} \arrow[r, two heads] \arrow[d, hook] &
    \widehat{\SOb_4} = \SO_4  \arrow[d, hook] \\
    \GL_2 \times \GL_2 \arrow[r] & \GL_4
  \end{tikzcd}
  \end{equation}
  where the vertical arrows are the natural inclusions, and the lower horizontal
  arrow is given by $(g,h)\mapsto (\det g)^{-1}g\otimes h$. Since the functorial
  transfer from~$\GLb_2 \times \GLb_2$ to~$\GLb_4$ exists (as we recalled at the
  beginning of Section~\ref{sec: global parameters}), we may compare at the
  unramified places and then use strong multiplicity one to compare at the
  ramified places, and we obtain that the composite
  $\WD_F\stackrel{\varphi'}{\to} \widehat{\Hb} \to\GL_2\times\GL_2 \to \GL_4$
  is given by $\varphi_{1,v} \otimes \varphi_{2,v}^\vee$, where $\varphi_{1,v}$,
  $\varphi_{2,v}$ are the $L$-parameters of~$\pi_{1,v}$ and~$\pi_{2,v}$
  respectively. Since the $L$-parameter of~$\Pi_v$ is~$\varphi_{1,v} \oplus
  \varphi_{2,v}$, we can take~$\pi=\Pi_v$, so we are done in the case
  that~$\varphi'$ is discrete.

  Suppose now that~$\varphi'$ is not discrete. Then one can argue as in the
  proof of \ref{thm: local Langlands for Sp4 is well defined}, since both local
  Langlands correspondences for $\bH$ and $\SOb_4$ are compatible with parabolic
  induction. In fact the proof is simpler since all proper Levi subgroups are
  simply products of $\GLb$, and we do not repeat the argument.
\end{proof}

\begin{rem}
  Theorem~\ref{thm: constituents of restriction from GSp4 to Sp4 in terms of
  characters} (or rather its straightforward extension from tempered to generic
  parameters) gives the complete spectral description of the automorphic
  restriction map of Section \ref{sec:restriction} for $\Spb_4 \subset
  \GSpinb_5$ \emph{for formally tempered global parameters}. This is the
  analogue of the results of Labesse--Langlands~\cite{LabLan} (ignoring inner
  forms) and the multiplicity one theorem of~Ramakrishnan
  for~$\SLb_2$~\cite{MR1792292}. It would perhaps be interesting to extend this
  to parameters which are not formally tempered, but in the interests of brevity
  we do not consider this question here.
\end{rem}

\emergencystretch=3em \bibliographystyle{amsalpha}
\bibliography{GSp4Arthur}

\end{document}